\documentclass[reqno]{amsart}

\usepackage{graphicx,pdfsync,amssymb, latexsym,amsfonts,amsbsy,color,esint,mathtools,nicefrac,cancel,bbm,tikz, caption, subcaption}
\usepackage[alphabetic,nobysame]{amsrefs}
\usepackage[shortlabels]{enumitem}

\usepackage{parskip}

\usepackage{hyperref}
\hypersetup{
colorlinks=true,
linkcolor=blue,
filecolor=blue,      
urlcolor=blue,
citecolor=blue,
}
\newtheorem{theorem}{Theorem}[section]
\newtheorem{lemma}[theorem]{Lemma}
\newtheorem{proposition}[theorem]{Proposition}
\newtheorem{definition}[theorem]{Definition}

\newtheorem{remark}[theorem]{Remark}

\numberwithin{equation}{section}

\usepackage{times}

\let\div\relax
\DeclareMathOperator\div{div}

\DeclareMathOperator\Id{Id}
\DeclareMathOperator\supp{supp}

\newcommand\newD{\tilde{\mathcal{D}}}

\title[Instantaneous blow-up for MHD]{Instantaneous blowup and non-uniqueness of smooth solutions of MHD}

\author [Mimi Dai]{Mimi Dai}

\address{Department of Mathematics, Statistics and Computer Science, University of Illinois at Chicago, Chicago, IL 60607, USA}
\email{mdai@uic.edu} 

\thanks{The author is partially supported by the NSF grant DMS--2308208 and Simons Foundation. }

\begin{document}

\begin{abstract}
We construct a family of solutions $(u,B)$ of the incompressible magnetohydrodynamic (MHD) system, the $L^\infty$ norm of which blows up instantaneously at the critical rate. The solutions remain smooth except at the blowup time. An inverse energy cascade mechanism and a convex integration scheme along a time sequence are the main ingredients of the construction, inspired by our recent work \cite{CDP} for the Navier-Stokes equations. The challenge of the construction for the MHD system stems from the coupling and the necessity of preserving the same ansatz of the principal solution at every iterative step while implementing convex integration. Existing convex integration schemes for MHD can treat the coupling but fail to produce the same ansatz of the principal solution recursively. 
To achieve the goal, we introduce a coupled geometric lemma that decomposes a symmetric tensor and a skew-symmetric tensor simultaneously. We emphasize that such coupled geometric lemma is new and of independent interest.
\end{abstract}

\maketitle

\begingroup
\setlength{\parskip}{0pt}
\tableofcontents
\endgroup

\section{Introduction}

We consider the incompressible magnetohydrodynamic (MHD) system
\begin{equation}\label{mhd}
\begin{split}
\partial_t u -   \Delta u + \div( u \otimes u-B \otimes B) + \nabla p &= 0, \\
\partial_t B -  \Delta B + \div( u \otimes B-B \otimes u)  &= 0, \\
\div u & =0
\end{split}
\end{equation}
where $u(x,t)$, $B(x,t)$ and $p(x,t)$ represent the velocity field, magnetic field and scalar pressure function respectively,
for $(x,t)\in\mathbb T^d\times [0,T]$ with spatial dimension $d\geq 2$ and some $T>0$. 

We define a weak solution of \eqref{mhd} in the following standard sense.

\begin{definition}\label{def:weak_solutions}
Let $\mathcal{D}_T$ be the space of test functions $\varphi \in C^\infty (\mathbb{T}^d \times \mathbb{R}) $ such that  $\div \varphi=0$ and $\varphi$ vanishes for $t\geq T$.
Let $ u_0 \in L^2(\mathbb{T}^d)$  be weakly divergence-free and $B_0 \in L^2(\mathbb{T}^d)$ . A pair $ (u,B) \in L^2 (\mathbb{T}^d \times [0,T])\times L^2 (\mathbb{T}^d \times [0,T])$ is a weak solution of \eqref{mhd} with initial data $(u_0,B_0)$ if:
\begin{enumerate}[1.,ref=\arabic*,left=1em]
    \item For $a.e.$ $t\in [0,T]$, $u$ is weakly divergence-free;
    
    \item For any $\varphi \in \mathcal{D}_T$,
\begin{equation}\notag
\begin{split}
\int_{\mathbb{T}^d} u_0(x)\cdot \varphi(x,0) \, dx &= - \int_0^T \int_{\mathbb{T}^d} u\cdot \big(  \partial_t \varphi+ \Delta \varphi +  u \cdot \nabla \varphi  \big)-B\cdot (B\cdot\nabla\varphi) \, dx dt,\\ 
\int_{\mathbb{T}^d} B_0(x)\cdot \varphi(x,0) \, dx &= - \int_0^T \int_{\mathbb{T}^d} B\cdot \big(  \partial_t \varphi+ \Delta \varphi +  u \cdot \nabla \varphi  \big)-u\cdot (B\cdot\nabla\varphi) \, dx dt.
\end{split}
\end{equation}
\end{enumerate}
\end{definition}
It is known that weak solutions exist globally in time for any $d\ge2$. Moreover, weak solutions satisfy the mild formulation (cf. \cite{MR316915})
\[
\begin{split}
u(t) &= e^{t\Delta}u_0 - \int_0^t e^{(t-s)\Delta} \mathbb{P}\div(u\otimes u-B\otimes B) (s)\, ds,\\
B(t) &= e^{t\Delta}B_0 - \int_0^t e^{(t-s)\Delta} \mathbb{P}\div(u\otimes B-B\otimes u) (s)\, ds
\end{split}
\]
in the sense of distributions.

On the other hand, a solution $(u,B)$ of \eqref{mhd} is \emph{classical} if it is smooth on $\mathbb T^d\times[0,T]$ in space and time. A unique classical solution from any smooth initial data is known to exist for a short time. The existence of a global in time classical solution remains an open problem for the MHD system in 3D and higher dimension. In this paper we construct solutions of the MHD system that blow up instantaneously and are classical except at the blowup time. Moreover, the blowup exhibits the critical rate according to the scaling of the system which will be described below. 

If $(u,p,B)(x,t)$ is a solution of \eqref{mhd}, the rescaled triple $(u_\lambda,p_\lambda,B_\lambda)(x,t)$ defined by
\[u_\lambda(x,t)=\lambda u(\lambda x,\lambda^2t), \quad p_\lambda(x,t)=\lambda^2 p(\lambda x,\lambda^2t), \quad B_\lambda(x,t)=\lambda B(\lambda x,\lambda^2t)\]
is a solution to the system as well. A few examples of scaling invariant (critical) spaces for $u$ and $B$ with embedding are 
\begin{equation}
\label{eq:critical_spaces}
\dot H^{\frac d2-1}\subset L^d\subset L^{d,\infty}\subset \dot W^{-1, \infty}\subset BMO^{-1} \subset \dot B^{-1}_{\infty,\infty}.
\end{equation}
By the Ladyzhenskaya--Prodi--Serrin type criterion, a solution $(u,B)$ with $u\in L^p_tL^q_x$ for $\frac2p + \frac{d}{q}=1$ and $q>d$ is classical and unique, see \cite{HeXin2005, Zhou2005}. Such spaces are also scaling invariant. The blowup rate of our constructed solutions corresponds to the endpoint Ladyzhenskaya--Prodi--Serrin space $L^2_tL^\infty_x$. 

\subsection{Main results}\label{sec-results}

The following result concerns the instantaneous blowup solution for the MHD system in $\mathbb T^d$ with critical blowup rate.

\begin{theorem} \label{main-thm}
For any $d\geq2$ and $u_0, B_0 \in C^\infty(\mathbb{T}^d)$ with $\div u_0=0$, there exists $T>0$ such that: for any $T_*\in[0,T)$, there exists a weak solution $(u(t), B(t))$ of \eqref{mhd} on $\mathbb T^d\times[0,T]$ satisfying:
\begin{enumerate}[1.,ref=\arabic*,left=1em]
    \item $(u(t), B(t))$ is a classical solution of \eqref{mhd} for $t\in[0,T_*]$ with $(u(0), B(0))=(u_0,B_0)$;\label{blow_up_theorem_classical_1}
    \item $(u(t), B(t))$ is a classical solution of \eqref{mhd} for $t\in(T_*,T]$ such that, for some $c>0$,
    \[
    \|u(t_n)\|_{L^\infty} \geq \frac{c}{\sqrt{t_n-T_*}}, \quad  \|\nabla \times u(t_n)\|_{L^\infty} \geq \frac{c}{t_n-T_*},
    \]
     \[
    \|B(t_n)\|_{L^\infty} \geq \frac{c}{\sqrt{t_n-T_*}}, \quad  \|\nabla \times B(t_n)\|_{L^\infty} \geq \frac{c}{t_n-T_*}
    \]
    along a sequence of times $t_n \to T_*+$;\label{blow_up_theorem_classical_2_with_lower_bound}
    and,
     \[
    \|u(t)\|_{L^\infty}+  \|B(t)\|_{L^\infty} \leq \frac{C}{\sqrt{t-T_*}}, \quad  \|\nabla u(t)\|_{L^\infty}+\|\nabla B(t)\|_{L^\infty} \leq \frac{C}{t-T_*}
    \]
    for all $t\in(T_*,T]$ and some constant $C>0$;
    \item The solution $(u(t), B(t))$ is weak-* continuous in time with values in $BMO^{-1}\times BMO^{-1}$; in addition, we have $u, B\in L_t^\infty \dot W^{-1,\infty}_x$ for $d\geq3$; \label{blow_up_theorem_type_i}
\label{blow_up_theorem_continuity}
   \item Moreover, for all $p<\infty$
\begin{equation}
    u, B\in L^2([0,T];L^p);
\end{equation}
   while for $p=\infty$, it satisfies the logarithmically-weakened conditions of Ladyzhenskaya--Prodi--Serrin type
\begin{equation}\label{logarithmically_weakened_prodi_serrin}
\int_0^T\frac{\|u(t)\|^2_\infty}{1+(\log\log(e+\|u(t)\|_\infty))^c}dt<\infty,
\end{equation}
and of Beale--Kato--Majda type
\begin{equation}\label{logarithmically_weakened_BKM}
\int_0^T\frac{\|\nabla u(t)\|_\infty}{1+(\log\log(e+\|\nabla u(t)\|_\infty))^c}dt<\infty
\end{equation}
for some $c>0$. \label{blow_up_theorem_prodi_serrin}
\end{enumerate}
\end{theorem}

In fact, we can construct an infinite family of such blowup solutions. 
Let $(U,H)$ be any classical solution of the MHD system on $\mathbb T^d\times [0,T]$ for $d\geq 2$ and $T\in(0,1]$. The pair $(U,H)$ can also be viewed as a periodic solution on $\mathbb R^d$. 

\begin{theorem}\label{thm-non-unique}
Fix any $T_*\in[0,T)$. We can construct a one-parameter family $\big((u^{(\sigma)},B^{(\sigma)})\big)_{\sigma\in[0,1]}$ of weak solutions to \eqref{mhd} on $\mathbb R^d\times[0,T]$ such that $(u^{(0)},B^{(0)})=(U,H)$ and, for every $\sigma\in(0,1]$,

     \[(u^{(\sigma)}, B^{(\sigma)})\big|_{[0,T_*]}\equiv (U,H)\big|_{[0,T_*]},\qquad (u^{(\sigma)}, B^{(\sigma)})\big|_{(T_*,T]}\not\equiv (U,H)\big|_{(T_*,T]}\]
   with
    \[u^{(\sigma)}(t), B^{(\sigma)}(t)\in C^\infty(\mathbb R^d),\quad\forall t\in[0,T],\]
    \[
    u^{(\sigma)}, B^{(\sigma)}\in L_t^{2,\infty}L_x^\infty\cap L_{t}^2L_{x,\mathrm{loc}}^p, \quad p<\infty.
    \]
    In addition we have for $d\geq3$
\begin{align*}
        u^{(\sigma)}, B^{(\sigma)}\in L_t^\infty \dot W^{-1,\infty}_x.
    \end{align*}
\end{theorem}

Theorem \ref{main-thm} item 2 indicates the instantaneous blowup at the critical rate. Moreover, the lower bounds on $\|\nabla u\|_{L^\infty}$ and $\|\nabla B\|_{L^\infty}$ are consistent with the Beale--Kato--Majda type blowup criteria scaling for MHD. Indeed, it was shown in \cite{CKS1997} that a solution blows up at $T$ if $\int_0^T\|\nabla\times u(t)\|_{L^\infty}\,dt=\infty$ and $\int_0^T\|\nabla\times B(t)\|_{L^\infty}\,dt=\infty$. The result was later improved in \cite{CMZ2007} in the sense that the blowup occurs even if only the velocity satisfies the Beale--Kato--Majda condition. The property stated in item 3 is sharp, since one would not expect standard continuity in time into $BMO^{-1}$ for a blowup solution. In fact it was shown in \cite{WangYuanZhaoZhou} that a solution $(u,B)$ of \eqref{mhd} is classical if $u\in C([0,T]; BMO^{-1})$ with standard continuity. The sharpness of estimates in item 4 is clear in view of the Ladyzhenskaya--Prodi--Serrin condition at the endpoint space $L^2_tL^\infty_x$ proved in \cite{HeXin2005, Zhou2005}.

At the same time, the non-uniqueness stated in Theorem \ref{thm-non-unique} also occurs in borderline spaces. Small data well-posedness of the MHD system \eqref{mhd} in $BMO^{-1}\times BMO^{-1}$ was obtained in \cite{MiaoYuanZhang2007}; well-posedness for a particular class of large initial data in the same space was proved in \cite{Guo2026}. However, the solution is small in both cases since a fixed point argument was used to show the existence of the solution. In contrast, it is noteworthy that our solutions constructed  in $BMO^{-1}\times BMO^{-1}$ are not small, and lose uniqueness.

Our result may also be viewed as mathematically adjacent to the magnetic reconnection phenomenon. Indeed, Theorem \ref{main-thm} produces an isolated-time singular event in which the magnetic field $B$ and $\nabla\times B$ become critically large, corresponding to rapid concentration of current density. This is qualitatively consistent with the reconnection scenarios, namely the formation of intense small-scale magnetic structures. Nevertheless, our theorem does not by itself establish a topological reconnection event, since we do not analyze changes of magnetic field-line connectivity or magnetic flux transport across the singular time. 


\textbf{Organization of the paper.} In the rest of the paper, we introduce notations in Section \ref{sec:notation} and highlight the main ideas of the construction in Section \ref{sec:idea}; building blocks are designed in Section \ref{sec:principal} and the principal part of the solution is constructed in Section \ref{sec:est}; the perturbation corrector is constructed in Section \ref{sec:corrector}; in the end, we prove the main results in Section \ref{sec:proof}.


\medskip

\section{Notations}\label{sec:notation}

Denote $\mathbb T^d$ by the $d$-dimensional torus $(\mathbb R/2\pi\mathbb Z)^d$. We use $|\Omega|$ to represent the volume of $\Omega$ restricted to $\mathbb T^d$. For simplification we write $\|f\|_p$ for $\|f\|_{L^p(\mathbb T^d)}$. The weak $L^p$ space is denoted by $L^{p,\infty}$. 
We also use $\|\cdot\|_{C^\alpha}$ to denote the $\alpha$-H\"older seminorm on $\mathbb R^d$ or $\mathbb T^d$.
We adapt the definition of the $BMO^{-1}$ space and Koch-Tataru space $X_T$ from \cite{KochTataru2001}.

Denote $\hat f$ by the Fourier transform and $\mathcal F^{-1}$ the inverse Fourier transform. We have the heat propagator $$e^{t\Delta}f\coloneqq\mathcal F^{-1}(e^{-t|\xi|^2}\hat f(\xi)),$$ 
and the Leray projection $$\mathbb Pf\coloneqq \mathcal F^{-1}\left(\left(\Id-\frac{\xi\otimes \xi}{|\xi|^2}\right)\hat f(\xi)\right).$$ 
Let $\psi\in C_c^\infty((2/3,3/2))$ be a bump function such that $\sum_{N\in 2^\mathbb N}\psi(r/N)\equiv 1$ for all $r\in[3/4,\infty)$.
We then define the Littlewood--Paley projections
$$P_Nf\coloneqq\mathcal F^{-1}(\psi(|\xi|/N)\hat f(\xi)).$$

Denote $\mathrm{Sym}^{d}$ by the space of symmetric $d\times d$ matrices, $\mathrm{Sym}_0^{d}$ the space of symmetric $d\times d$ traceless matrices, and $\mathfrak{so}(3)$ the $3\times 3$ skew-symmetric matrices. 
For a vector function $f$, we denote $\nabla\odot f$ by the symmetric tensor field with components $\frac12(\partial_if_j+\partial_jf_i)$.

We then define the operators
\begin{align*}
    \mathcal Df&\coloneqq2\nabla\odot f-2(\div f)\Id,\\
    \newD f&\coloneqq2\nabla\odot f-(\div f)\Id,\\
    \mathcal D_sf&\coloneqq \nabla f-(\nabla f)^T+(\div f)\Id 
    \end{align*}
    for a vector field $f$ on $\mathbb T^d$, and additionally
    \begin{align*}
    \mathcal R&\coloneqq \Delta^{-1}\newD,\\
     \mathcal R_s&\coloneqq \Delta^{-1}\mathcal D_s,\\
    \mathbb Q&\coloneqq2\Delta^{-1}\nabla\odot \mathbb P\div,\\
    \mathbb Q_s&\coloneqq \Delta^{-1}(\nabla \mathbb P-(\nabla \mathbb P)^T)\div.
\end{align*}
  We list the following important identities: 
\begin{align*}
    \div \mathcal D=\Delta\mathbb P,
\end{align*}
\begin{align*}
    \div\newD = \Delta,\qquad \newD=\mathbb Q\mathcal D+\left(2\frac{\nabla\otimes\nabla}{\Delta}-\Id\right)\div,
\end{align*}
\begin{equation}\notag
    \div \mathcal D_s=\Delta, \qquad \mathcal D_s=\mathbb Q_s\mathcal D_s+\Id\div,
\end{equation}
\begin{align*}
    \div \mathcal R=P_{\neq0},
\end{align*}
where $P_{\neq0}$ is the projection operator to frequencies $\xi\neq0$, and 
\begin{align}\label{eq:Q_identities}
    \mathbb Q=\mathcal R\mathbb P\div,\qquad \mathbb Q\mathcal D=2\nabla\odot\mathbb P,
\end{align}
\begin{align}\label{eq:Qs_identities}
    \mathbb Q_s=\mathcal R_s\mathbb P\div,\qquad \mathbb Q_s\mathcal D_s=\nabla \mathbb P-(\nabla \mathbb P)^T.
\end{align}
We point out that $\Delta\mathbb P$ is a local differential operator.

\medskip


\section{Main ideas of the construction}\label{sec:idea}

The main mechanism of generating instantaneously blowup is the inverse energy cascade, which was exploited in \cite{CDP} to construct blowup solutions for the Navier-Stokes equations. To realize such inverse energy transfer recursively from high to low modes, we need to implement a convex integration scheme along a sequence of times which are associated with the frequency scales. To ensure the iterative process, it is essential to preserve the ansatz of the principal part of the solution. However, for the coupled MHD system, any existing convex integration scheme cannot preserve the ansatz of the principal profiles. To overcome the obstacle, we introduce a new coupled geometric lemma that decomposes a symmetric tensor and a skew-symmetric tensor simultaneously; thus the convex integration scheme can be designed to fully explore the coupling of the system. 

For comparison on the convex integration part, we refer the reader to \cite{BBV2020, FLS2021, FLS2024, LZZ2022} where non-uniqueness of the MHD systems was constructed in various contexts. In those works, the coupling part in the velocity equation was treated as an error term at low frequency and hence erased together with other low frequency errors. However, in our case, in order to preserve the same ansatz of the velocity at low modes, we cannot treat this coupling term as an error; instead, we need to count on it to produce a portion of the velocity at low modes. This will become more clear in Subsection \ref{sec:outline}.

The coupled geometric Lemma \ref{le-geometry-couple} below is the main novelty of the current construction to ensure that the high-high frequency interactions from the coupling part yield the desired velocity profile at low frequency level.
Notably, this coupled geometric lemma is of independent interest. 



\subsection{Geometric lemmas}\label{sec:geometric}
We start from a standard geometric lemma for symmetric tensor without proof as it has been extensively used in convex integration.
\begin{lemma}[Symmetric geometric lemma]\label{le-geometry1}
Let $B_{\varepsilon_u}(Id)$ be the ball of radius $\varepsilon_u>0$ centered at the identity matrix in the space $\mathrm{Sym}^3$. There exists a finite subset $\Lambda_u\subset S^2\cap \mathbb Q^3$ of vectors $\eta$ associated with orthonormal bases $(\eta, \eta_1,\eta_2)$ such that for any $R\in B_{\varepsilon_u}(Id)$, we have the decomposition 
\begin{equation}\label{decomp1}
R=\sum_{\eta\in \Lambda_u}\Gamma_\eta^2(R) \eta_1\otimes \eta_1
\end{equation}
for smooth functions $\Gamma_\eta: B_{\varepsilon_u}(Id)\to \mathbb R$.
\end{lemma}

We then prove the following geometric lemma for symmetric traceless tensors. Although stated in 3D, it can be generalized to any dimension.
\begin{lemma}[Symmetric geometric lemma for traceless tensors]\label{le-geometry2}
Let $B_{\varepsilon_B}(0)$ be the ball of radius $\varepsilon_B>0$ centered at the 0 matrix in the space $\mathrm{Sym}^3_0$. There exists a finite subset $\Lambda_B\subset S^2\cap \mathbb Q^3$ of vectors $\eta$ associated with orthonormal bases $(\eta, \eta_1,\eta_2)$ such that for any $R\in B_{\varepsilon_B}(0)$, we have the decomposition 
\begin{equation}\label{decomp3}
R=\sum_{\eta\in \Lambda_B}\Gamma_\eta^2(R) (\eta_1\otimes \eta_1-\eta_2\otimes \eta_2)
\end{equation}
for smooth functions $\Gamma_\eta: B_{\varepsilon_B}(0)\to \mathbb R$. Moreover, $\Lambda_B$ can be chosen such that $\Lambda_u\cap \Lambda_B=\emptyset$.
\end{lemma}

\begin{proof}
Choose
\[
\eta^1=(0,0,1),\quad \eta^2=(0,1,0),\quad
\eta^3=\left(\frac35,\frac45,0\right),\quad
\eta^4=\left(\frac35,0,\frac45\right),\quad
\eta^5=\left(0,\frac35,\frac45\right),
\]
and associated orthonormal pairs
\[
(\eta_1^1,\eta_2^1)=\big((1,0,0),(0,1,0)\big),\qquad
(\eta_1^2,\eta_2^2)=\big((1,0,0),(0,0,1)\big),
\]
\[
(\eta_1^3,\eta_2^3)=\left(\left(-\frac45,\frac35,0\right),(0,0,1)\right),\quad
(\eta_1^4,\eta_2^4)=\left(\left(-\frac45,0,\frac35\right),(0,1,0)\right),
\]
\[
(\eta_1^5,\eta_2^5)=\left(\left(0,-\frac45,\frac35\right),(1,0,0)\right).
\]
Set
\[
T_i\coloneqq \eta_1^i\otimes\eta_1^i-\eta_2^i\otimes\eta_2^i\in\mathrm{Sym}^3_0,\qquad i=1,\dots,5.
\]
Then $T_3,T_4,T_5$ have nonzero $(1,2)$, $(1,3)$, $(2,3)$ entries, respectively, while $T_1,T_2$ are diagonal. Hence $T_1,\dots,T_5$ are linearly independent, so they form a basis of $\mathrm{Sym}^3_0$.

Now define $\eta^{i+5}\coloneqq-\eta^i$ and choose
\[
(\eta_1^{i+5},\eta_2^{i+5})\coloneqq(\eta_2^i,\eta_1^i),\qquad i=1,\dots,5.
\]
Then
\[
T_{i+5}\coloneqq \eta_1^{i+5}\otimes\eta_1^{i+5}-\eta_2^{i+5}\otimes\eta_2^{i+5}=-T_i.
\]
Let $\Lambda_B=\{\eta^1,\dots,\eta^{10}\}$ (we discuss disjointness from $\Lambda_u$ at the end).

Given $R\in\mathrm{Sym}^3_0$, write uniquely
\[
R=\sum_{i=1}^5 r_i(R)\,T_i,
\]
where $r_i$ are linear functionals on $\mathrm{Sym}^3_0$. Define
\[
a_i(R)\coloneqq1+\frac12 r_i(R),\qquad a_{i+5}(R)\coloneqq1-\frac12 r_i(R),\qquad i=1,\dots,5.
\]
Then
\[
\sum_{i=1}^{10}a_i(R)T_i
=\sum_{i=1}^5\big(a_i(R)-a_{i+5}(R)\big)T_i
=\sum_{i=1}^5 r_i(R)T_i
=R.
\]
Since $r_i(R)\to0$ as $R\to0$, after shrinking $\varepsilon_B$ if needed we have
$a_i(R)>0$ on $B_{\varepsilon_B}(0)$. Set
\[
\Gamma_{\eta^i}(R)\coloneqq \sqrt{a_i(R)},\qquad i=1,\dots,10.
\]
These functions are smooth on $B_{\varepsilon_B}(0)$, and by construction
\[
R=\sum_{\eta\in\Lambda_B}\Gamma_\eta^2(R)\,(\eta_1\otimes\eta_1-\eta_2\otimes\eta_2).
\]

Finally, because $\Lambda_u$ is finite and $\mathbb S^2\cap\mathbb Q^3$ is dense in $\mathbb S^2$, we can perturb the above rational directions slightly (still rational) to avoid $\Lambda_u$. 
Therefore we may assume $\Lambda_B\cap\Lambda_u=\emptyset$.
\end{proof}

We further show a geometric lemma for skew-symmetric tensors. This lemma was proved in \cite{BBV2020}. We give an independent proof here, which will be used to prove the coupled geometric Lemma \ref{le-geometry-couple}.
\begin{lemma}[Skew-symmetric geometric lemma]\label{le-geometry3}
Let $B_{\varepsilon_B}(0)$ be the ball of radius $\varepsilon_B>0$ centered at the 0 matrix in the space $\mathfrak{so}(3)$. There exists a finite subset $\Lambda_B\subset S^2\cap \mathbb Q^3$ of vectors $\eta$ associated with orthonormal bases $(\eta, \eta_1,\eta_2)$ such that for any $G\in B_{\varepsilon_B}(0)$, we have the decomposition 
\begin{equation}\label{decomp2}
G=\sum_{\eta\in \Lambda_B}\Gamma_\eta^2(G) (\eta_1\otimes \eta_2-\eta_2\otimes \eta_1)
\end{equation}
for smooth functions $\Gamma_\eta: B_{\varepsilon_B}(0)\to \mathbb R$. Moreover, $\Lambda_B$ can be chosen such that $\Lambda_u\cap \Lambda_B=\emptyset$.
\end{lemma}

\begin{proof}
We use the same set $\Lambda_B$ and the same associated orthonormal frames as in the proof of Lemma \ref{le-geometry2}. In particular, with the notation there, $\Lambda_B=\{\eta^1,\dots,\eta^{10}\}$, $\eta^{i+5}=-\eta^i$, and
\[
(\eta_1^{i+5},\eta_2^{i+5})=(\eta_2^i,\eta_1^i),\qquad i=1,\dots,5.
\]
Define
\[
S_i\coloneqq \eta_1^i\otimes\eta_2^i-\eta_2^i\otimes\eta_1^i\in\mathfrak{so}(3),\qquad i=1,\dots,10.
\]
Then $S_{i+5}=-S_i$ for $i=1,\dots,5$.

Let $\mathcal J:\mathbb R^3\to\mathfrak{so}(3)$ be $\mathcal J(a)b=a\times b$. This is a linear isomorphism. For $i=1,2,3$ there are signs $\sigma_i\in\{\pm1\}$ such that
\[
S_i=\sigma_i\mathcal J(\eta^i).
\]
Since the explicit vectors $\eta^1,\eta^2,\eta^3$ are linearly independent, $S_1,S_2,S_3$ are linearly independent. Hence they form a basis of $\mathfrak{so}(3)$.

Therefore every $G\in\mathfrak{so}(3)$ has a unique representation
\[
G=\sum_{i=1}^3 g_i(G)\,S_i,
\]
where $g_i$ are linear functionals on $\mathfrak{so}(3)$. Define
\[
a_i(G)\coloneqq1+\frac12 g_i(G),\qquad a_{i+5}(G)\coloneqq1-\frac12 g_i(G),\qquad i=1,2,3,
\]
and set
\[
a_4(G)=a_5(G)=a_9(G)=a_{10}(G)=1.
\]
Using $S_{i+5}=-S_i$, we get
\[
\sum_{i=1}^{10}a_i(G)S_i
=\sum_{i=1}^5\big(a_i(G)-a_{i+5}(G)\big)S_i
=\sum_{i=1}^3g_i(G)S_i
=G.
\]
Since each $g_i(G)\to0$ as $G\to0$, after shrinking $\varepsilon_B$ if necessary we have $a_i(G)>0$ on $B_{\varepsilon_B}(0)$. Define
\[
\Gamma_{\eta^i}(G)\coloneqq\sqrt{a_i(G)},\qquad i=1,\dots,10.
\]
Then all $\Gamma_{\eta^i}$ are smooth on $B_{\varepsilon_B}(0)$ and
\[
G=\sum_{\eta\in\Lambda_B}\Gamma_\eta^2(G)\,(\eta_1\otimes\eta_2-\eta_2\otimes\eta_1).
\]

Finally, $\Lambda_B\cap\Lambda_u=\emptyset$ is achieved by the same small rational perturbation argument used in Lemma \ref{le-geometry2}.
\end{proof}


Now we are ready to prove the coupled geometric lemma.
\begin{lemma}[Coupled geometric lemma]\label{le-geometry-couple}
Possibly after enlarging $\Lambda_B$, there exist $\varepsilon_*>0$, smooth functions
\[
\Gamma_\eta:\;B_{\varepsilon_*}(0)\subset \mathrm{Sym}^3\times\mathfrak{so}(3)\to (0,\infty),
\qquad
p:\;B_{\varepsilon_*}(0)\to\mathbb R,
\]
such that for every $(R,G)\in B_{\varepsilon_*}(0)$,
\[
R-p(R,G)\Id
=\sum_{\eta\in\Lambda_B}\Gamma_\eta^2(R,G)\,(\eta_1\otimes\eta_1-\eta_2\otimes\eta_2),
\]
\[
G
=\sum_{\eta\in\Lambda_B}\Gamma_\eta^2(R,G)\,(\eta_1\otimes\eta_2-\eta_2\otimes\eta_1).
\]
Moreover, $\Lambda_B$ can be chosen so that $\Lambda_B\cap\Lambda_u=\emptyset$.
\end{lemma}

\begin{proof}
Let
\[
V\coloneqq \mathrm{Sym}^3_0\times \mathfrak{so}(3),
\]
with
$\dim V=5+3=8$.

For an orthonormal frame $(\eta,\eta_1,\eta_2)$, define
\[
T(\eta)\coloneqq \eta_1\otimes\eta_1-\eta_2\otimes\eta_2\in \mathrm{Sym}^3_0,\qquad
S(\eta)\coloneqq \eta_1\otimes\eta_2-\eta_2\otimes\eta_1\in \mathfrak{so}(3),
\]
and $W(\eta)\coloneqq (T(\eta),S(\eta))\in V$.

Choose the following eight rational orthonormal frames:
\[
\begin{aligned}
&\eta^1=(1,0,0),\quad &&\eta_1^1=(0,1,0),\quad &&\eta_2^1=(0,0,1),\\
&\eta^2=(0,0,-1),\quad &&\eta_1^2=(0,1,0),\quad &&\eta_2^2=(1,0,0),\\
&\eta^3=(0,1,0),\quad &&\eta_1^3=(-1,0,0),\quad &&\eta_2^3=(0,0,1),\\
&\eta^4=\left(\frac23,\frac13,-\frac23\right),\quad
&&\eta_1^4=\left(\frac13,\frac23,\frac23\right),\quad
&&\eta_2^4=\left(\frac23,-\frac23,\frac13\right),\\
&\eta^5=\left(\frac23,\frac23,\frac13\right),\quad
&&\eta_1^5=\left(-\frac23,\frac13,\frac23\right),\quad
&&\eta_2^5=\left(\frac13,-\frac23,\frac23\right),\\
&\eta^6=\left(\frac23,-\frac23,-\frac13\right),\quad
&&\eta_1^6=\left(\frac23,\frac13,\frac23\right),\quad
&&\eta_2^6=\left(-\frac13,-\frac23,\frac23\right),\\
&\eta^7=\left(\frac23,-\frac13,\frac23\right),\quad
&&\eta_1^7=\left(-\frac13,\frac23,\frac23\right),\quad
&&\eta_2^7=\left(-\frac23,-\frac23,\frac13\right),\\
&\eta^8=\left(\frac13,\frac23,-\frac23\right),\quad
&&\eta_1^8=\left(-\frac23,\frac23,\frac13\right),\quad
&&\eta_2^8=\left(\frac23,\frac13,\frac23\right).
\end{aligned}
\]
Set $W_i\coloneqq W(\eta^i)$, $i=1,\dots,8$. In the coordinates
\[
(M_{11},M_{22},M_{12},M_{13},M_{23},N_{12},N_{13},N_{23})
\]
of $V$, a direct computation gives
\[
\det[W_1\,W_2\,\cdots\,W_8]=\frac{64}{81}\neq0.
\]
Hence $W_1,\dots,W_8$ form a basis of $V$.

Now define, for $i=1,\dots,8$,
\[
\eta^{i+8}\coloneqq -\eta^i,\qquad
(\eta_1^{i+8},\eta_2^{i+8})\coloneqq (\eta_2^i,\eta_1^i).
\]
Then
\[
T(\eta^{i+8})=-T(\eta^i),\qquad S(\eta^{i+8})=-S(\eta^i).
\]
Replace $\Lambda_B$ by an enlargement containing $\{\eta^1,\dots,\eta^{16}\}$ and relabel it as $\Lambda_B$.

Given $(R,G)\in \mathrm{Sym}^3\times\mathfrak{so}(3)$, set
\[
p(R,G)\coloneqq \frac13\operatorname{tr}(R),\qquad \bar R\coloneqq R-p(R,G)\Id\in \mathrm{Sym}^3_0,\qquad
X\coloneqq (\bar R,G)\in V.
\]
Since $W_1,\dots,W_8$ is a basis of $V$, there are unique linear functionals
$\ell_i:V\to\mathbb R$ such that
\[
X=\sum_{i=1}^8 \ell_i(X)\,W_i.
\]
Define
\[
b_i(X)\coloneqq 1+\frac12\ell_i(X),\qquad
b_{i+8}(X)\coloneqq 1-\frac12\ell_i(X),\qquad i=1,\dots,8.
\]
Because $\ell_i(0)=0$, after choosing $\varepsilon_*>0$ small enough we have
$b_j(X)>0$ for all $j=1,\dots,16$ whenever $(R,G)\in B_{\varepsilon_*}(0)$.
Set
\[
\Gamma_{\eta^j}(R,G)\coloneqq \sqrt{b_j(X)},\qquad j=1,\dots,16.
\]
Then $\Gamma_{\eta^j}$ are smooth on $B_{\varepsilon_*}(0)$, and
\[
\sum_{j=1}^{16}\Gamma_{\eta^j}^2(R,G)\,W_j
=\sum_{i=1}^8\big(b_i(X)-b_{i+8}(X)\big)W_i
=\sum_{i=1}^8\ell_i(X)W_i
=X.
\]
Taking the symmetric and skew components yields
\[
R-p(R,G)\Id
=\sum_{\eta\in\Lambda_B}\Gamma_\eta^2(R,G)\,(\eta_1\otimes\eta_1-\eta_2\otimes\eta_2),
\]
\[
G
=\sum_{\eta\in\Lambda_B}\Gamma_\eta^2(R,G)\,(\eta_1\otimes\eta_2-\eta_2\otimes\eta_1).
\]

Finally, as in Lemmas \ref{le-geometry2} and \ref{le-geometry3}, since
$\Lambda_u$ is finite and linear independence of $W_1,\dots,W_8$ is an open condition,
we may perturb the rational directions slightly (still rational) so that
$\Lambda_B\cap\Lambda_u=\emptyset$ while preserving the construction.
\end{proof}

\subsection{Heuristics of the construction}\label{sec:outline}

Let $J_d$ be the number of elements in $\Lambda_u\cup \Lambda_B$ and $\mathcal J_d=\{1, 2, ... , J_d\}$.
For $k\in\mathbb N$ and $j\in \mathcal J_d$, the frequency scales $N_{j,k}$ will be chosen such that
\[
\begin{split}
A^{b^k}\sim N_{1,k} \ll N_{2,k} \ll \dots \ll N_{J_d,k} \ll N_{1,k+1} \sim A^{b^{k+1}}, \quad d=2;\\
A^{b^k}\sim N_{1,k} = N_{2,k} = \dots = N_{J_d,k} = N_{1,k+1} \sim A^{b^{k+1}}, \quad d\geq 3.
\end{split}
\]
for a large constant $A\gg 1$ and another constant $b>1$. Accordingly, we choose a decreasing sequence of times $N_{1, k+1}^{-2} \ll t_k \ll N_{J_d,k}^{-2}$, such that energy transfers from the $(k+1)$-th frequency level to $k$-th level over $[t_{k+1},t_k]$.

We then choose the basic building blocks approximately in the form
\[
\begin{split}
\psi_{j,k,u}(x)&\approx N_{j,k}^{-2}a_{j,k,u}(x)\eta^j_1\sin(N_{j,k}\eta^j\cdot x), \quad \eta^j\in \Lambda_u,\\
\psi_{j,k,B}(x)&\approx N_{j,k}^{-2}a_{j,k,B}(x)\eta^j_2\sin(N_{j,k}\eta^j\cdot x), \quad \eta^j\in \Lambda_B,\\
\psi_{j,k,c}(x)&\approx N_{j,k}^{-2}a_{j,k,B}(x)\eta^j_1\sin(N_{j,k}\eta^j\cdot x), \quad \eta^j\in \Lambda_B
\end{split}
\]
where $\{\eta^j,\eta^j_1, \eta^j_2\}$ is an orthonormal basis and $N_{j,k}\eta^j\in\mathbb Z^3$ for all $j\in \mathcal J_3$ and $k\in\mathbb N$. While for $x\in \mathbb T^2$, we take $\eta^j_1=\eta^{j,\perp}$ and $\eta^j_2=(0,0,1)$. The subscript letter ``c"  in $\psi_{j,k,c}$ indicates coupling, since this term is introduced to handle the coupling between the velocity field and the magnetic field. The amplitude functions $a_{j,k,u}$ and $a_{j,k,B}$ will be designed iteratively with the properties that they are bounded and their frequency support is at a much smaller scale compared to $N_{j,k}$.

The principal part $(v,h)$ of the solution will be chosen as
\[
v(x,t) = \sum_k v_k(x,t), \quad h(x,t) = \sum_k h_k(x,t)
\]
with
\[
\begin{split}
v_k(x,t)& \approx -\sum_{j\in \Lambda_u} N_{j,k} e^{-N_{j,k}^2t} \mathbb P\Delta \psi_{j,k,u}(x)-\sum_{j\in \Lambda_B} N_{j,k} e^{-N_{j,k}^2t} \mathbb P\Delta \psi_{j,k,c}(x), \\
h_k(x,t)& \approx -\sum_{j\in \Lambda_B} N_{j,k} e^{-N_{j,k}^2t} \mathbb P\Delta \psi_{j,k,B}(x)
\end{split}
\]
for $t \geq t_k$,
where we abuse the notations for simplicity: $j\in \Lambda_u$ means the summation takes over for all $\eta^j\in \Lambda_u$ and $j\in \Lambda_B$ means the summation takes over for all $\eta^j\in \Lambda_B$.
Neglecting lower order terms, we note
\[-\mathbb P\Delta\psi_{j,k,u}(x)\approx a_{j,k,u}(x)\eta^j_1\sin(N_{j,k}\eta^j\cdot x)=O(1),\]
\[-\mathbb P\Delta\psi_{j,k,B}(x)\approx O(1), \quad -\mathbb P\Delta \psi_{j,k,c}(x)\approx O(1).\]

The choice of $a_{j,k+1,u}$ and $a_{j,k+1,B}$ is to generate low mode component $(v_k, h_k)$ from the nonlinear interactions at $(k+1)$-th level over the time interval $[t_{k+1},t_k]$. Namely, we would like to have
\begin{equation} \label{ansatz-intro}
\begin{split}
v_k(x,t) &\approx -\int_0^{t} e^{(t-s)\Delta }\mathbb P\div ( v_{k+1}\otimes v_{k+1}-h_{k+1}\otimes h_{k+1})(s) \, ds,\\
h_k(x,t) &\approx -\int_0^{t} e^{(t-s)\Delta }\mathbb P\div ( v_{k+1}\otimes h_{k+1}-h_{k+1}\otimes v_{k+1})(s) \, ds.
\end{split}
\end{equation}
When $d=2$, we have the scaling computation 
\begin{equation}\label{scale-separation}
\int_0^{t} N_{j,k} N_{j',k} e^{-N_{j,k}^2 s-N_{j',k}^2 s} \, ds \sim 
\begin{cases}
1, &j=j',\\
N_{j,k}^{-1} N_{j',k} \ll 1, &j > j',
\end{cases}
\end{equation}
for $t\geq t_{k-1}$, which says only the interactions from the same direction are significant. On the other hand, for $d\geq3$, the equation \eqref{scale-separation} with $j=j'$ still holds, and the interactions from different directions can be made small by choosing disjoint supports of the building blocks. Therefore we infer
\[
\begin{split}
&\quad\int_0^t P_{< N_{1,k+1}}\mathbb P\div( v_{k+1}\otimes v_{k+1}-h_{k+1}\otimes h_{k+1}) \, ds \\
&\approx P_{< N_{1, k+1}}\mathbb P\div \sum_{j\in \Lambda_u} \mathbb P\Delta \psi_{j,k+1,u} \otimes \mathbb P\Delta \psi_{j,k+1,u}\\
&\quad+P_{< N_{1, k+1}}\mathbb P\div \sum_{j\in \Lambda_B} \left(\mathbb P\Delta \psi_{j,k+1,c} \otimes \mathbb P\Delta \psi_{j,k+1,c}-\mathbb P\Delta \psi_{j,k+1,B} \otimes \mathbb P\Delta \psi_{j,k+1,B}\right)\\
& \approx \mathbb P\div\sum_{j\in\Lambda_u}a_{j,k+1,u}^2\eta^j_1\otimes\eta^j_1+\mathbb P\div\sum_{j\in\Lambda_B}a_{j,k+1,B}^2(\eta^j_1\otimes\eta^j_1-\eta^j_2\otimes\eta^j_2)
\end{split}
\]
for $t\geq t_k$, where we only keep the leading order terms of the self-interactions in the approximation. Analogously, we have
\[
\begin{split}
&\quad\int_0^t P_{< N_{1,k+1}}\mathbb P\div( v_{k+1}\otimes h_{k+1}-h_{k+1}\otimes v_{k+1}) \, ds \\
&\approx P_{< N_{1, k+1}}\mathbb P\div \sum_{j\in \Lambda_B} \mathbb P\Delta \psi_{j,k+1,c} \otimes \mathbb P\Delta \psi_{j,k+1,B}\\
&\quad-P_{< N_{1, k+1}}\mathbb P\div \sum_{j\in \Lambda_B} \mathbb P\Delta \psi_{j,k+1, B} \otimes \mathbb P\Delta \psi_{j,k+1,c}\\
& \approx \mathbb P\div\sum_{j\in\Lambda_B}a_{j,k+1,B}^2(\eta^j_1\otimes\eta^j_2-\eta^j_2\otimes\eta^j_1).
\end{split}
\]

Choosing the coefficients $a_{j,k,u}$ and $a_{j,k,B}$ so that 
\begin{equation}\label{decomp-heuristic}
\begin{split}
\sum_{j\in \Lambda_u}a_{j,k+1,u}^2\eta^j_1\otimes\eta^j_1 &\approx \mathcal D \sum_{j\in\Lambda_u} N_{j,k}\psi_{j,k,u}+p\Id, \\
\sum_{j\in\Lambda_B}a_{j,k+1,B}^2(\eta^j_1\otimes\eta^j_1-\eta^j_2\otimes\eta^j_2)&\approx \mathcal D \sum_{j\in\Lambda_B}  N_{j,k}\psi_{j,k,c}+p\Id,\\
\sum_{j\in\Lambda_B}a_{j,k+1,B}^2(\eta^j_1\otimes\eta^j_2-\eta^j_2\otimes\eta^j_1) &\approx \mathcal D_s \sum_{j\in\Lambda_B}  N_{j,k}\psi_{j,k,B}
\end{split}
\end{equation}
for some scalar functions $p(x)$ (which may vary from line to line). 
It then follows that
\[
\begin{split}
&\quad\int_0^t P_{< N_{1,k+1}}\mathbb P\div (v_{k+1}\otimes v_{k+1}-h_{k+1}\otimes h_{k+1}) \, ds\\
& \approx \sum_{j\in\Lambda_u} N_{j,k} \Delta \psi_{j,k,u}+\sum_{j\in\Lambda_B} N_{j,k} \Delta \psi_{j,k,c} + \nabla p
\end{split}
\]
and 
\[\int_0^t P_{< N_{1,k+1}}\mathbb P\div( v_{k+1}\otimes h_{k+1}-h_{k+1}\otimes v_{k+1}) \, ds\approx \sum_{j\in\Lambda_B} N_{j,k} \Delta \psi_{j,k,B} \]
for $t\geq t_k$, which justifies \eqref{ansatz-intro}. 

The first decomposition of \eqref{decomp-heuristic} can be achieved using the standard geometric Lemma \ref{le-geometry1}. We note that the coefficient functions $a_{j,k+1,B}$ in the second and third decompositions coincide due to the coupling nature of system \eqref{mhd}. This is exactly the reason we need the coupled geometric Lemma \ref{le-geometry-couple} to ensure the second and third equations of \eqref{decomp-heuristic}.

\begin{remark} In the purely 2D case, that is $x\in \mathbb T^2$ and $u=(u_1(x_1,x_2), u_2(x_1,x_2))$, $B=(B_1(x_1,x_2), B_2(x_1,x_2))$, we can take a stream function $b=b(x_1,x_2)$ such that $B=\nabla^\perp b$. The 2D MHD \eqref{mhd} can be then written as
\begin{equation}\label{mhd-2d}
\begin{split}
\partial_t u -   \Delta u + \div( u \otimes u-B \otimes B) + \nabla p &= 0, \\
\partial_t b -  \Delta b + u\cdot\nabla b &= 0, \\
\div u  =0, \ B&=\nabla^\perp b.
\end{split}
\end{equation}
We note \eqref{mhd-2d} is in a close form to the incrompressible flow system with a passive tracer.
This case is handled in the companion paper \cite{DaiYang2026} of the author and Yang. 
\end{remark}




\medskip

\section{Building blocks}\label{sec:principal}

The choice of frequency scales in Subsection \ref{frequency_scales_section} and the geometry of building blocks in Subsection \ref{sec:geometry} below are adapted from our paper \cite{CDP}. The building blocks were previously employed in \cite{CoiculescuPalasek2025} for the Navier-Stokes equations.
The rest of the scheme is designed to respect the nonlinear coupling and geometry structure of the MHD system \eqref{mhd}. 

\subsection{Frequency scales}\label{frequency_scales_section}


Fix $b>1$ to be relatively large depending on the dimension $d\geq2$. We postpone the definition of $m_*\in\mathbb N$ to a few lines below. Let $A>1$ be sufficiently large depending on $b,d$ and $m_*$.  Now we fix the minimum frequency $N_{1,0}=1$.
For $k\geq 1$ and $j\in\mathcal J_d$, take
\begin{equation}\label{N_definition}
N_{j,k}=\left\{
\begin{aligned}
&m_*A^{\big\lceil b^{k+(j-1)/J_d}\big\rceil}, \quad &d=2,\\
&m_*A^{\big\lceil b^{k}\big\rceil}, \quad &d \geq 3
\end{aligned}
\right.
\end{equation}
with $\lceil \cdot\rceil$ being the ceiling function. We choose $m_*$ such that $m_*\eta^j\in\mathbb Z^d$ for all $\eta^j\in \Lambda_u\cup \Lambda_B$ and $N_{j,k}\eta^j\in\mathbb Z^d$ for all $j$ and $k$.

Note for sufficiently large $A$, we have 
\[N_{j_1,k_1}\ll N_{j_2,k_2} \quad \mbox{if either} \quad k_1<k_2, \quad \mbox{or} \quad k_1=k_2, j_1<j_2.\]

Now we fix $\gamma\in(b^{-1/J_d}, 1)$. 
For integer $k\geq1$ and $j\in\mathcal J_d$, we choose another sequence of frequency scales
\begin{equation}\label{M_definition}
    M_{j,k}=\left\{
\begin{aligned}
&A^{\big\lceil \gamma b^{k}\big\rceil}, \quad &j=1\text{ or }d\geq 3,\\
&A^{\big\lceil \gamma b^{(j-1)/J_d}\big\rceil} M_{1,k}, \quad &j\geq2\text{ and }d=2,
\end{aligned}
\right.
\end{equation}
such that (taking sufficiently large $A$ if needed) for $d=2$,
\begin{equation}\begin{aligned}\label{N_and_M_ordering}
    A^cN_{j-1,k}\leq M_{j,k}\leq A^{-c}N_{j,k}&,\quad 2\leq j\leq J_d,\\
    A^cN_{J_d,k-1}\leq M_{1,k}\leq A^{-c} N_{1,k}&
\end{aligned}\end{equation}
and for $d\geq 3$,
\begin{equation}\begin{aligned}\label{N_and_M_ordering_3d}
   A^cN_{j,k-1}\leq M_{j,k}\leq A^{-c}N_{j,k}&,\quad 1\leq j\leq J_d
\end{aligned}\end{equation}
for a constant $c>0$ depending on $b$ and $\gamma$.

\subsection{Geometry of the building blocks}\label{sec:geometry}

Denote $\mathcal C_{j,k}(\rho)$ by the $2\pi/M_{j,k}$-periodic cylinder with axis in the line $\mathbb R\eta^j$ and 
radius $\rho M_{j,k}^{-1}$. Choose sufficiently small $\delta_0>0$ depending on $d$ such that
\begin{align}\label{pipe_volume_bound}
|\mathcal C_{j,k}(4\delta_0)|\leq1/(10J_d).
\end{align}
When $d\geq3$, it is possible to choose $\delta_0>0$ even smaller so that
\begin{align*}
    \mathcal C_{j_1,k}(4\delta_0)\cap\mathcal C_{j_2,k}(4\delta_0)=\emptyset\qquad\forall j_1\neq j_2.
\end{align*}

Define $\varphi_{j,k}\in C_c^\infty(\mathcal C_{j,k}(\delta_0))$ as a $2\pi/M_{j,k}$-periodic cutoff function for $k\geq1$, satisfying
\[
\begin{split}
\eta^j\cdot\nabla\varphi_{j,k}=0, \qquad \forall j,k, \\
(2\pi)^{-d}\int_{\mathbb T^d}\varphi_{j,k}^2(x)\sin^2(N_{j,k}\eta^j\cdot x) \, dx =1, \qquad \forall j,k,
\end{split}
\]
and 
\begin{equation} \label{eq:varphi_bound}
\|\nabla^n \varphi_{j,k}\|_\infty\lesssim M_{j,k}^{n}, \qquad \forall j,k.
\end{equation}
We further denote for $k\geq 1$
\begin{align*}
    \Omega_k=\bigcap_{k'=1}^k\bigcup_{j\in\mathcal J_d}\mathcal C_{j,k'}((3-2^{-(k-k')})\delta_0),\quad \widetilde\Omega_k=\bigcap_{k'=1}^k\bigcup_{j\in \mathcal J_d}\mathcal C_{j,k'}((3-\frac342^{-(k-k')})\delta_0),
\end{align*}
and $\Omega_0=\widetilde\Omega_0=\mathbb T^d$.

The following lemma was proved in \cite{CDP}.
\begin{lemma}
We have
    \begin{align}
    |\Omega_{k}|\leq2^{-k}|\mathbb T^d|, \qquad k\geq0.\label{Omega_volume_estimate}
\end{align}
There exists a constant $C_0>1$ with the following property: if $Q\subset\mathbb T^d$ is a cube with $\ell(Q)\in [C_0M_{1,k_0}^{-1},2\pi)$, then
\begin{align}\label{cube_intersection_volume_bound}
    |\Omega_{k}\cap Q|\leq2^{-(k-k_0)}|Q|, \qquad\forall k\geq k_0.
\end{align}
Moreover, there exist cutoff functions $\chi_k\in C_c^\infty(\widetilde\Omega_{k-1})$ with $\chi_k\equiv1$ on $\Omega_{k-1}$ satisfying
\begin{align} \label{eq:chi_k_bound}
    \|\nabla^n\chi_k\|_\infty\lesssim_nM_{J_d,k-1}^n.
\end{align}
\end{lemma}

\subsection{The potential profiles}\label{sec:potential}


Define $\phi_k$ as a standard mollifier at length scale \[\ell_k:=N_{1,k}^{-\frac12}N_{1,k+1}^{-\frac12}.\] 
In view of \eqref{N_definition}, we have $\ell_k \leq  N_{J_d,k}^{-1}$ for $d\geq 3$; it also holds for $d=2$ provided $b$ is large enough. Define the profiles of potential
\begin{align}
\psi_{j,0,u}(x)&\coloneqq N_{j,0}^{-2}a_{j,0,u}(x)\eta_1^j\sin(N_{j,0}\eta^j\cdot x), \qquad \eta_j\in \Lambda_u, \label{def_psi_0}\\
\psi_{j,k,u}(x)&\coloneqq N_{j,k}^{-2}\phi_k*(a_{j,k,u}(x)\varphi_{j,k}(x)\eta_1^j\sin(N_{j,k}\eta^j\cdot x)),\quad \eta_j\in \Lambda_u, k\geq1, \label{def_psi}
\end{align}
\begin{align}
\psi_{j,0,B}(x)&\coloneqq N_{j,0}^{-2}a_{j,0,B}(x)\eta_2^j\sin(N_{j,0}\eta^j\cdot x), \qquad \eta_j\in \Lambda_B, \label{def_psi_0b}\\
\psi_{j,k,B}(x)&\coloneqq N_{j,k}^{-2}\phi_k*(a_{j,k,B}(x)\varphi_{j,k}(x)\eta_2^j\sin(N_{j,k}\eta^j\cdot x)),\quad \eta_j\in \Lambda_B, k\geq1, \label{def_psi_b}
\end{align}
\begin{align}
\psi_{j,0,c}(x)&\coloneqq N_{j,0}^{-2}a_{j,0,B}(x)\eta_1^j\sin(N_{j,0}\eta^j\cdot x), \qquad \eta_j\in \Lambda_B,\label{def_psi_0c}\\
\psi_{j,k,c}(x)&\coloneqq N_{j,k}^{-2}\phi_k*(a_{j,k,B}(x)\varphi_{j,k}(x)\eta_1^j\sin(N_{j,k}\eta^j\cdot x)),\quad \eta_j\in \Lambda_B, k\geq1.\label{def_psi_c}
\end{align}
We will choose the amplitude functions $a_{j,k,u}$ and $a_{j,k,B}$ inductively so that the following identities are satisfied:
\begin{equation}\label{a-convex-integration-identity}
\begin{split}
\sum_{j\in \Lambda_u}a_{j,k+1,u}^2\eta^j_1\otimes\eta^j_1 &=2 \mathcal D \sum_{j\in\Lambda_u} N_{j,k}\psi_{j,k,u}+p\Id, \\
\sum_{j\in\Lambda_B}a_{j,k+1,B}^2(\eta^j_1\otimes\eta^j_1-\eta^j_2\otimes\eta^j_2)&=2 \mathcal D \sum_{j\in\Lambda_B}  N_{j,k}\psi_{j,k,c}+p\Id,\\
\sum_{j\in\Lambda_B}a_{j,k+1,B}^2(\eta^j_1\otimes\eta^j_2-\eta^j_2\otimes\eta^j_1) &=2 \mathcal D_s \sum_{j\in\Lambda_B}  N_{j,k}\psi_{j,k,B}
\end{split}
\end{equation}
for some scalar functions $p(x)$ (different from line to line), and
\begin{equation} \label{eq:a_bounds}
    \|\nabla^na_{j,k,u}\|_\infty+  \|\nabla^na_{j,k,B}\|_\infty\lesssim N_{J_d,k-1}^{n}, \qquad \forall j,k,
\end{equation}
\begin{equation}\label{a-support}
    \supp a_{j,k,u}\subset \widetilde\Omega_{k-1}, \qquad \supp a_{j,k,B}\subset \widetilde\Omega_{k-1}.
\end{equation}

We remark that, from now on, we also interpret 
\[a_{j,k}=a_{j,k,u} \qquad \mbox{for}\quad j\in \Lambda_u; \qquad a_{j,k}=a_{j,k,B} \qquad \mbox{for}\quad j\in \Lambda_B\]
and similarly 
\[\psi_{j,k}=\psi_{j,k,u} \qquad \mbox{for}\quad j\in \Lambda_u; \qquad \psi_{j,k}=\psi_{j,k,B} \qquad \mbox{for}\quad j\in \Lambda_B,\]
while $\psi_{j,k}=\psi_{j,k,c}$ when we refer to the term from coupling.

The next lemma shows that one can indeed construct amplitude functions $a_{j,k,u}$ and $a_{j,k,B}$ with the required properties.

\begin{lemma}
    For $j\in\mathcal J_d$ and $k\geq0$, there exist coefficient functions $a_{j,k,u}, a_{j,k,B}\in C^\infty(\mathbb T^d;\mathbb R)$ satisfying \eqref{a-convex-integration-identity} and \eqref{eq:a_bounds}. Moreover, $\supp\psi_{j,k}\subset \Omega_k$ and
    \begin{align}\label{eq:Dpsi_bounds}
        \|\nabla^n\psi_{j,k}\|_{L^\infty}\lesssim_n N_{j,k}^{-2+n}
    \end{align}
    for $n\geq0$.
\end{lemma}

\begin{proof}
\textbf{Construction of $a_{j,k,u}$.}
    We apply the symmetric geometric Lemma \ref{le-geometry1} to show the first identity in \eqref{a-convex-integration-identity}. This part is similar to the proof of Lemma 4.2 from \cite{CDP}. We only provide the explicit construction of $a_{j,k,u}$ and outline the main steps of the estimates.
    
 At the initial frequency level, we choose $a_{1,0,u}(x)=1$, $a_{j,0,u}(x)=0$ for $j\neq1$. For $k\geq 0$, we define iteratively
\[
a_{j,k+1,u}(x)=c^{-\frac12}\chi_{k+1}(x)\Gamma_j \big(\Id+cS_{k,u}(x)\big)
\]
with
\[
S_{k,u}\coloneqq2\mathcal D \sum_{j\in\Lambda_u} N_{j,k}\psi_{j,k,u},
\]
where $\psi_{j,k,u}$ is defined in terms of $a_{j,k,u}$ as in \eqref{def_psi_0}--\eqref{def_psi},  and the constant $c$ is chosen 
small enough so that
\begin{equation} \label{eq:condition_on_c}
c\|S_{k,u}\|_\infty \leq c_0, \quad \forall k\geq 0,
\end{equation}
and hence Lemma \ref{le-geometry1} can be applied. 
We have 
\[\|a_{j,k,u}\|_\infty \lesssim c^{-\frac12}, \quad \forall j, \forall k\geq 0,\]
and the rough estimates
\[
\begin{split}
\|\nabla^n \psi_{j,k,u}\|_\infty &\lesssim N^{-2}_{j,k} \ell^{-n}_k  \|a_{j,k,u}\|_\infty \lesssim N^{-2}_{j,k} \ell^{-n}_k  c^{-\frac12},\\
\|\nabla^n S_{k,u}\|_\infty &\lesssim \sum_j N_{j,k}\|\nabla^{n+1} \psi_{j,k,u}\|_\infty \lesssim N^{-1}_{1,k} \ell_k^{-n-1} c^{-\frac12},\\
\|\nabla^n_x\Gamma_j(\Id+cS_{k,u}(x))\|_\infty & \lesssim c^{\frac{n}2}N^{-n}_{1,k} \ell_k^{-2n},\\
\|\nabla^n a_{j,k+1,u}\|_\infty &\lesssim c^{\frac{n-1}{2}}N^{-n}_{1,k} \ell_k^{-2n}.
\end{split}
\]

Thanks to the scales $\ell_{k-1}^{-1}  =  N_{1,k-1}^{\frac12}N_{1,k}^{\frac12}$ and $M_{j,k} \leq  N_{j,k}$, the above estimates can be improved to
\[
\begin{split}
\|\nabla^{n} \psi_{j,k,u}\|_\infty&\lesssim_n N_{j,k}^{-2} \sum_{i=0}^{n} \|\nabla^i(a_{j,k,u}\varphi_{j,k,u})\|_\infty N_{j,k}^{n-i}\\
&\lesssim_n  N_{j,k}^{-2}\sum_{i=0}^n (c^{\frac{i-1}{2}}N^{-i}_{1,k-1} \ell_{k-1}^{-2i} + c^{-\frac12} M_{j,k}^i)N_{j,k}^{n-i}\\
&\lesssim_n N_{j,k}^{-2} c^{-\frac{1}{2}} N_{j,k}^{n},
\end{split}
\]
and hence
\[
\begin{split}
\|\nabla^nS_{k,u}\|_\infty &\lesssim \sum_{j\in\Lambda_u}  N_{j,k}\|\nabla^{n+1} \psi_{j,k,u}\|_\infty\\
&\lesssim \sum_{j\in\Lambda_u}  N_{j,k} N_{j,k}^{-2} c^{-\frac{1}{2}}  N_{j,k}^{n+1}\\
&\lesssim \sum_{j\in\Lambda_u} c^{-\frac{1}{2}}  N_{j,k}^{n}\\
&\lesssim c^{-\frac{1}{2}}  N_{J_d,k}^{n},
\end{split}
\]
\[
\begin{split}
\|\nabla^n a_{j,k+1,u}\|_\infty &= \big\|\nabla^n \Big(c^{-\frac12}\chi_{k+1}(x)\Gamma_j\big(\Id+cS_{k,u}(x)\big)\Big)\big\|_\infty\\
&\lesssim_n c^{-\frac12} \|\nabla^n_x\Gamma_j(\Id+cS_{k,u}(x))\|_\infty\\
&\lesssim_n N^{n}_{J_d,k}.
\end{split}
\]
	Therefore the estimates \eqref{eq:a_bounds} and \eqref{eq:Dpsi_bounds} hold for $a_{j,k,u}$ and $\psi_{j,k,u}$.

	Now we proceed to verify the first identity in \eqref{a-convex-integration-identity},
	\begin{align*}
	    \sum_{j\in\Lambda_u} a_{j,k+1,u}^2\eta^j\otimes\eta^j&=c^{-1}\chi_{k+1}^2(x) \sum_{j\in\Lambda_u} \Gamma_j^2\big(\Id+cS_{k,u}(x)\big)\eta^j\otimes\eta^j\\
	    &=\chi_{k+1}^2(x)(c^{-1}\Id+S_{k,u}(x)).
	\end{align*}
	By definition $\chi_{k+1}\equiv1$ on $\Omega_k$; it is thus sufficient to show that $\psi_{j,k,u}$, and hence $S_{k,u}$, are supported on $\Omega_k$. Indeed,
	\begin{align*}
	    \supp\psi_{j,k,u}\subset (\supp{a_{j,k,u}}\cap\supp\varphi_{j,k,u})+B(0,\ell_k).
	\end{align*}
	Note \eqref{a-support} for $a_{j,k,u}$ follows from the factor $\chi_{k+1}$ in the definition of $a_{j,k+1,u}$.
	Combining \eqref{a-support},  $\supp\varphi_{j,k,u}\subset\mathcal C_{j,k}(\delta_0)$, and the definitions of $\Omega_k$ and $\widetilde\Omega_k$, we see that $\supp S_{k,u}\subset \Omega_k$, by choosing parameters such that $\ell_kM_{j,k'}\ll2^{-(k-k')}$ for all $j$ and $k'\leq k$.

\textbf{Construction of $a_{j,k,B}$.}
 Set
	\[
	a_{j,0,B}\equiv 0,\qquad j\in\Lambda_B,
	\]
	and for $k\ge0$ define
	\[
	S_{k,B}^{\mathrm{sym}}\coloneqq 2\mathcal D\sum_{j\in\Lambda_B}N_{j,k}\psi_{j,k,c},\qquad
	S_{k,B}^{\mathrm{skw}}\coloneqq 2\mathcal D_s\sum_{j\in\Lambda_B}N_{j,k}\psi_{j,k,B}.
	\]
	By Lemma \ref{le-geometry-couple}, there exist smooth functions $\Theta_j(R,G)>0$ and a smooth scalar map $\Pi(R,G)$, defined on $B_{\varepsilon_*}(0)\subset \mathrm{Sym}^3\times\mathfrak{so}(3)$, such that
	\[
	R-\Pi(R,G)\Id=\sum_{j\in\Lambda_B}\Theta_j^2(R,G)\,(\eta_1^j\otimes\eta_1^j-\eta_2^j\otimes\eta_2^j),
	\]
	\[
	G=\sum_{j\in\Lambda_B}\Theta_j^2(R,G)\,(\eta_1^j\otimes\eta_2^j-\eta_2^j\otimes\eta_1^j).
	\]
	Define recursively
	\[
	a_{j,k+1,B}(x)\coloneqq c^{-1/2}\chi_{k+1}(x)\Theta_j\!\Big(cS_{k,B}^{\mathrm{sym}}(x),\,cS_{k,B}^{\mathrm{skw}}(x)\Big).
	\]

	Exactly as above, from the definitions \eqref{def_psi_0b}--\eqref{def_psi_c}, we get rough bounds
	\[
	\|\nabla^n\psi_{j,k,B}\|_\infty+\|\nabla^n\psi_{j,k,c}\|_\infty
	\lesssim N_{j,k}^{-2}\ell_k^{-n}c^{-1/2},
	\]
	and hence
	\[
	\|\nabla^nS_{k,B}^{\mathrm{sym}}\|_\infty+\|\nabla^nS_{k,B}^{\mathrm{skw}}\|_\infty
	\lesssim N_{1,k}^{-1}\ell_k^{-n-1}c^{-1/2}.
	\]
	Taking $c>0$ sufficiently small ensures
	\[
	c\Big(\|S_{k,B}^{\mathrm{sym}}\|_\infty+\|S_{k,B}^{\mathrm{skw}}\|_\infty\Big)\le \varepsilon_*
	\]
	for all $k$. Applying the multivariate Fa\`a di Bruno formula to $\Theta_j(cS_{k,B}^{\mathrm{sym}},cS_{k,B}^{\mathrm{skw}})$ gives
	\[
	\|\nabla_x^n\Theta_j(cS_{k,B}^{\mathrm{sym}},cS_{k,B}^{\mathrm{skw}})\|_\infty
	\lesssim c^{n/2}N_{1,k}^{-n}\ell_k^{-2n},
	\]
	thus
	\[
	\|\nabla^n a_{j,k+1,B}\|_\infty\lesssim c^{\frac{n-1}{2}}N_{1,k}^{-n}\ell_k^{-2n}.
	\]
	Similarly as for the $u$-part in the proof above, this rough bound implies the optimal estimates
	\[
	\|\nabla^n\psi_{j,k,B}\|_\infty+\|\nabla^n\psi_{j,k,c}\|_\infty
	\lesssim N_{j,k}^{-2}c^{-1/2}N_{j,k}^{n},
	\]
	\[
	\|\nabla^nS_{k,B}^{\mathrm{sym}}\|_\infty+\|\nabla^nS_{k,B}^{\mathrm{skw}}\|_\infty
	\lesssim c^{-1/2}N_{J_d,k}^{n},
	\]
	and therefore
	\[
	\|\nabla^na_{j,k+1,B}\|_\infty\lesssim_n N_{J_d,k}^{n}.
	\]
	So the bound \eqref{eq:a_bounds} holds for $a_{j,k,B}$ as well; and hence \eqref{eq:Dpsi_bounds} also holds.

	Using the coupled decomposition lemma at
	\[
	(R,G)=\big(cS_{k,B}^{\mathrm{sym}}(x),\,cS_{k,B}^{\mathrm{skw}}(x)\big),
	\]
	we obtain
	\begin{align*}
	\sum_{j\in\Lambda_B}a_{j,k+1,B}^2(\eta_1^j\otimes\eta_1^j-\eta_2^j\otimes\eta_2^j)
	&=\chi_{k+1}^2\Big(S_{k,B}^{\mathrm{sym}}-c^{-1}\Pi(cS_{k,B}^{\mathrm{sym}},cS_{k,B}^{\mathrm{skw}})\Id\Big),\\
	\sum_{j\in\Lambda_B}a_{j,k+1,B}^2(\eta_1^j\otimes\eta_2^j-\eta_2^j\otimes\eta_1^j)
	&=\chi_{k+1}^2S_{k,B}^{\mathrm{skw}}.
	\end{align*}
	As above, one checks
	\[
	\supp\psi_{j,k,B}\subset\Omega_k,\qquad \supp\psi_{j,k,c}\subset\Omega_k,
	\]
	which implies $\supp S_{k,B}^{\mathrm{sym}},\supp S_{k,B}^{\mathrm{skw}}\subset\Omega_k$, hence $\chi_{k+1}\equiv1$ on these supports. Therefore
	\[
	\chi_{k+1}^2S_{k,B}^{\mathrm{sym}}=S_{k,B}^{\mathrm{sym}},\qquad
	\chi_{k+1}^2S_{k,B}^{\mathrm{skw}}=S_{k,B}^{\mathrm{skw}},
	\]
	and the second and third identities in \eqref{a-convex-integration-identity} follow (with pressure absorbed into the scalar term $p$).

	Finally, \eqref{a-support} for $a_{j,k,B}$ is immediate from the factor $\chi_{k+1}$ in the definition of $a_{j,k+1,B}$, and the support inclusions above imply $\supp\psi_{j,k,u},\supp\psi_{j,k,B},\supp\psi_{j,k,c}\subset\Omega_k$.
\end{proof}

\section{The principal part of the solution}\label{sec:est}

\subsection{Construction of the principal part of the solution}
With the preparations in the previous section,
we first define the approximate principal velocity and magnetic field, 
\begin{equation}\label{app-prin}
\begin{split}
\bar v_k(x,t)&= \sum_{j\in\Lambda_u\cup\Lambda_B} \bar v_{j,k}(x,t) \\
&= \sum_{j\in\Lambda_u} -N_{j,k} e^{-N_{j,k}
^2t} \Delta \psi_{j,k,u}(x)+\sum_{j\in\Lambda_B} -N_{j,k} e^{-N_{j,k}
^2t} \Delta \psi_{j,k,c}(x),\\
\bar h_k(x,t)&= \sum_{j\in\Lambda_B} \bar h_{j,k}(x,t) = \sum_{j\in\Lambda_B} -N_{j,k} e^{-N_{j,k}
^2t} \Delta \psi_{j,k,B}(x)
\end{split}
\end{equation}
for $k\geq 0$. The principal velocity and magnetic field are then defined to be
\begin{equation}\label{prin}
\begin{split}
v_k(x,t) &= -\int_0^{t} e^{(t-s)\Delta }\mathbb P\div (\bar v_{k+1}\otimes \bar v_{k+1}-\bar h_{k+1}\otimes \bar h_{k+1})(s) \, ds,\\
h_k(x,t) &= -\int_0^{t} e^{(t-s)\Delta }\mathbb P\div (\bar v_{k+1}\otimes \bar h_{k+1}-\bar h_{k+1}\otimes \bar v_{k+1})(s) \, ds
\end{split}
\end{equation}
for $k\geq 0$.

The full principal part $(v,h)$ is thus defined as 
\begin{align}\label{full-prin}
    v(x,t)=\sum_{k=0}^\infty v_k(x,t), \qquad h(x,t)=\sum_{k=0}^\infty h_k(x,t).
\end{align}
We also introduce the corresponding approximate principal fields by setting
\begin{align}\label{full-app}
    \bar v(x,t)=\sum_{k=0}^\infty \bar v_k(x,t), \qquad  \bar h(x,t)=\sum_{k=0}^\infty \bar h_k(x,t).
\end{align}

We also define tensor fields $R_k$, $\bar R_k$ and $H_k$, $\bar H_k$ such that 
\begin{equation}\notag
v_k=\div R_k, \qquad \bar v_k=\div\bar R_k, \qquad h_k=\div H_k, \qquad \bar h_k=\div\bar H_k.
\end{equation}
In particular, we choose
\begin{equation}\label{tensor-app}
\begin{split}
\bar R_k(x,t)&= \sum_{j\in\Lambda_u\cup\Lambda_B} \bar R_{j,k}(x,t)\\
& = \sum_{j\in\Lambda_u} -N_{j,k} e^{-N_{j,k}^2t}  \newD  \psi_{j,k,u}(x)+\sum_{j\in\Lambda_B} -N_{j,k} e^{-N_{j,k}^2t}  \newD  \psi_{j,k,c}(x),\\
&=: \bar R_{k,u}(x,t)+ \bar R_{k,c}(x,t),\\
\bar H_k(x,t)&= \sum_{j\in\Lambda_B} \bar H_{j,k}(x,t)
 = \sum_{j\in\Lambda_B} -N_{j,k} e^{-N_{j,k}^2t}  \mathcal D_s  \psi_{j,k,B}(x)
\end{split}
\end{equation}
and
\begin{equation}\label{tensor-RH}
\begin{split}
R_k(x,t) &= -\int_0^{t} e^{(t-s)\Delta }\mathcal R\mathbb P\div (\bar v_{k+1}\otimes \bar v_{k+1}-\bar h_{k+1}\otimes \bar h_{k+1})(s) \, ds,\\
H_k(x,t) &= -\int_0^{t} e^{(t-s)\Delta }\mathcal R_s\mathbb P\div (\bar v_{k+1}\otimes \bar h_{k+1}-\bar h_{k+1}\otimes \bar v_{k+1})(s) \, ds
\end{split}
\end{equation}
recalling the definitions from Section~\ref{sec:notation}.

In the next subsection we compare the exact principal pair $(v,h)$ with its approximate counterpart $(\bar v,\bar h)$. This comparison will show that $(v,h)$ solves an MHD system with forcing terms that are suitably small.

\subsection{Error estimates}

Now we fix the time scales $t_k\coloneqq N_{J_d,k}^{-4}$ satisfying 
\begin{equation}\label{time-sequence}
N_{1, k+1}^{-2} \ll t_k \ll N_{J_d,k}^{-3}.
\end{equation}
We then show the following estimate for $(\bar R_k, \bar H_k)$. 

\begin{proposition}\label{barv-I_estimate_proposition}
Denote
\[
\begin{split}
\mathcal{I}_{k,u} &= -\int_0^{t} e^{(t-s)\Delta }\sum_{j\in \Lambda_u} N_{j,k+1}^2 e^{-2N_{j,k+1}^2s} \mathbb Q (a_{j,k+1,u}^2 \eta_1^j\otimes \eta_1^j) \, ds,\\
\mathcal{I}_{k,c} &= -\int_0^{t} e^{(t-s)\Delta }\sum_{j\in \Lambda_B} N_{j,k+1}^2 e^{-2N_{j,k+1}^2s} \mathbb Q (a_{j,k+1,B}^2 (\eta_1^j\otimes \eta_1^j-\eta_2^j\otimes \eta_2^j)) \, ds,\\
\mathcal{I}_{k,B} &= -\int_0^{t} e^{(t-s)\Delta }\sum_{j\in \Lambda_B} N_{j,k+1}^2 e^{-2N_{j,k+1}^2s} \mathbb Q_s (a_{j,k+1,B}^2 (\eta_1^j\otimes \eta_2^j-\eta_2^j\otimes \eta_1^j)) \, ds.
\end{split}
\]
Let $A$ and $b$ be large enough.
For any $\varepsilon_0 >0$, $\alpha \in(0, \frac{1}{10})$, and $\bar n \in \mathbb{N}$, we have for $k\geq 1$ and $t\geq 0$:
\[
\|\nabla^n(\bar R_k(t) - \mathcal{I}_{k,u}(t)- \mathcal{I}_{k,c}(t))\|_{\infty} \leq \varepsilon_0 N^{-\alpha}_{1,k}(t^{-\frac{n}{2} +\alpha}+1), \qquad  n=1,2,\dots,\bar n,
\]
\[
\|\bar R_k(t) - \mathcal{I}_{k,u}(t)- \mathcal{I}_{k,c}(t)\|_{\infty} \lesssim_\varepsilon \varepsilon_0N_{1,k}^{-\alpha}+\mathbbm{1}_{t\leq t_k}N_{J_d, k}^{\varepsilon}, \qquad \text{if} \qquad d=2,
\]
\[
\|\bar R_k(t) - \mathcal{I}_{k,u}(t)- \mathcal{I}_{k,c}(t)\|_{\infty} \lesssim \varepsilon_0N_{1,k}^{-\alpha}+\mathbbm1_{t\leq t_k}, \qquad \text{if} \qquad d\geq 3,
\]
\[
\|\nabla^n(\bar H_k(t) - \mathcal{I}_{k,B}(t))\|_{\infty} \leq \varepsilon_0 N^{-\alpha}_{1,k}(t^{-\frac{n}{2} +\alpha}+1), \qquad  n=1,2,\dots,\bar n,
\]
\[
\|\bar H_k(t) - \mathcal{I}_{k,B}(t)\|_{\infty} \lesssim_\varepsilon \varepsilon_0N_{1,k}^{-\alpha}+\mathbbm{1}_{t\leq t_k}N_{J_d, k}^{\varepsilon}, \qquad \text{if} \qquad d=2,
\]
\[
\|\bar H_k(t) - \mathcal{I}_{k,B}(t)\|_{\infty} \lesssim \varepsilon_0N_{1,k}^{-\alpha}+\mathbbm1_{t\leq t_k}, \qquad \text{if} \qquad d\geq 3.
\]
\end{proposition}
\begin{proof}
{\textbf{Step I.}}
The goal of this step is to show 
\begin{equation}\label{step1}
\begin{split}
\|\nabla^n(\bar R_{k,u}(t) - \mathcal{I}_{k,u}(t))\|_{\infty} &\leq \varepsilon_0 N^{-\alpha}_{1,k}(t^{-\frac{n}{2} +\alpha}+1), \qquad  n=1,2,\dots,\bar n,\\
\|\bar R_{k,u}(t) - \mathcal{I}_{k,u}(t)\|_{\infty} &\lesssim_\varepsilon \varepsilon_0N_{1,k}^{-\alpha}+\mathbbm{1}_{t\leq t_k}N_{J_d, k}^{\varepsilon}, \qquad \text{if} \qquad d=2,\\
\|\bar R_{k,u}(t) - \mathcal{I}_{k,u}(t)\|_{\infty} &\lesssim \varepsilon_0N_{1,k}^{-\alpha}+\mathbbm1_{t\leq t_k}, \qquad \text{if} \qquad d\geq 3.
\end{split}
\end{equation}
The estimates can be obtained as in \cite{CDP} for the Navier-Stokes equation. We only outline the main steps to establish \eqref{step1}.
We start with a rough estimate to be used on $[0,t_k]$, for any $\varepsilon>0$
\[
\begin{split}
\|\nabla^n\mathcal{I}_{k,u}\|_{\infty} 
\lesssim_\varepsilon \int_0^{t} \sum_{j\in\Lambda_u} N_{j,k+1}^2 e^{-2N_{j,k+1}
^2s} \|\nabla^{n}(a_{j,k+1,u}^2)\|_{C^\varepsilon} \, ds
\lesssim_\varepsilon N_{J_d, k}^{n+\varepsilon}.
\end{split}
\]
In the case $d\geq 3$, we have $N_{j,k}=N_{1,k}$ for all $j$ and the estimate above can be improved. In view of \eqref{a-convex-integration-identity}, we can write
\[
\begin{split}
&\quad\nabla^n\mathbb Q\sum_{j\in\Lambda_u} (a_{j,k+1,u}^2 \eta_1^j\otimes \eta_1^j) \\
&=2\nabla^n\mathcal R \Delta \sum_{j\in\Lambda_u} N_{1,k}^{-1}\phi_k*(a_{j,k,u}\varphi_{j,k}\eta_1^j\sin(N_{1,k}\eta^j\cdot x))\\
&\quad - 2\nabla^n\mathcal R \nabla \sum_{j\in\Lambda_u} N_{1,k}^{-1}\phi_k*(\nabla(a_{j,k,u}\varphi_{j,k})\cdot \eta_1^j\sin(N_{1,k}\eta^j\cdot x)).
\end{split}
\]
Since $\mathcal R\Delta$ is a local operator, we have
\[
\Big\|\nabla^n\mathbb Q\sum_{j\in\Lambda_u} (a_{j,k+1,u}^2 \eta_1^j\otimes \eta_1^j)\Big\|_\infty \lesssim N_{J_d, k}^n.
\]
Hence we have the improved estimate for $d\geq 3$ 
\[
\begin{split}
\|\nabla^n\mathcal{I}_{k,u}\|_\infty 
\lesssim \int_0^{t}  N_{1,k+1}^2 e^{-2N_{1,k+1}^2s} N_{J_d, k}^n \, ds
\lesssim N_{J_d, k}^n,
\end{split}
\]
and
\[
\|\nabla^n\bar R_{k,u}(t)\|_\infty\lesssim \sum_j N_{j,k}^{n}e^{-N_{j,k}^2t}\lesssim N_{J_d,k}^{n}.
\]
Thus it follows for $t\in [0,t_k]$:
\begin{equation} \label{eq:bound_before_t_k_2D}
\|\nabla^n(\bar R_{k,u}(t) - \mathcal{I}_{k,u}(t))\|_\infty  \lesssim_\varepsilon N_{J_d, k}^{n+\varepsilon}, \qquad \text{if} \qquad d=2,
\end{equation}
\begin{equation} \label{eq:bound_before_t_k_3D}
\|\nabla^n(\bar R_{k,u}(t) - \mathcal{I}_{k,u}(t))\|_\infty  \lesssim_\varepsilon N_{J_d, k}^n, \qquad \text{if} \qquad d\geq 3.
\end{equation}

On the other hand, for $t\geq t_k$, the integral $\mathcal{I}_{k,u}$ can be split into
\begin{equation}\notag
\begin{split}
\mathcal{I}_{k,u}(t) &=-\int_0^{t_k} e^{(t-s)\Delta }\sum_{j\in\Lambda_u} N_{j,k+1}^2 e^{-2N_{j,k+1}
^2s} \mathbb Q (a_{j,k+1,u}^2(x) \eta_1^j\otimes \eta_1^j) \, ds\\
&\quad-\int_{t_k}^t e^{(t-s)\Delta }\sum_{j\in\Lambda_u} N_{j,k+1}^2 e^{-2N_{j,k+1}
^2s} \mathbb Q (a_{j,k+1,u}^2(x) \eta_1^j\otimes \eta_1^j) \, ds\\
&=:\mathcal{I}_{k,u}^{<t_k} + \mathcal{I}_{k,u}^{>t_k},
\end{split}
\end{equation}
where the first term is close to $\bar R_{k,u}$ while the second one is a minor error term. Indeed, define
\[
\tilde{\mathcal{I}}_{k,u}^{<t_k} = -e^{(t-t_k)\Delta}\int_0^{t_k} \sum_{j\in\Lambda_u} N_{j,k+1}^2 e^{-2N_{j,k+1}
^2s} \mathbb Q  \left(a_{j,k+1,u}^2(x) \eta_1^j\otimes \eta_1^j\right) \, ds.
\]
We can show
\[
 \| \nabla^n(\mathcal{I}_{k,u}^{<t_k} - \tilde{\mathcal{I}}_{k,u}^{<t_k})\|_\infty 
\lesssim t_k N_{J_d,k}^{n+2+\varepsilon}, \qquad t \geq t_k.
\]
Applying the first identity of \eqref{a-convex-integration-identity} and \eqref{eq:Q_identities} we can write
\[
\begin{split}
\tilde{\mathcal{I}}_{k,u}^{<t_k} 
&= -\sum_{j\in\Lambda_u} N_{j,k} e^{-N_{j,k}^2t} \newD  \psi_{j,k,u} - \sum_{j\in\Lambda_u} N_{j,k} (e^{-N_{j,k}^2(t-t_k)}-e^{-N_{j,k}^2t}) \newD \psi_{j,k,u} \\
&\qquad + \sum_{j\in\Lambda_u} N_{j,k} (e^{-N_{j,k}^2(t-t_k)}-e^{(t-t_k)\Delta}) \newD  \psi_{j,k,u} \\
&\qquad + \frac{1}{2}e^{(t-t_k)\Delta}\sum_{j\in\Lambda_u} e^{-2N_{j,k+1}
^2t_k} \mathbb Q \left(a_{j,k+1,u}^2(x) \eta_1^j\otimes \eta_1^j\right) -(\Id-2\frac{\nabla\otimes\nabla}{\Delta})p\\
&= \bar R_{k,u} + \tilde{\mathcal{I}}_{k,u}^1 +\tilde{\mathcal{I}}_{k,u}^2 + \tilde{\mathcal{I}}_{k,u}^3 -(\Id-2\frac{\nabla\otimes\nabla}{\Delta})p.
\end{split}
\]
All the terms on the right hand side except $\bar R_{k,u}$ are shown to be minor error terms,
\begin{equation}\notag
\begin{split}
\|\nabla^n\tilde{\mathcal{I}}_{k,u}^1\|_{\infty} &\lesssim \sum_{j\in\Lambda_u}N_{j,k}^{n+2}t_ke^{-N_{j,k}^2t},\\
 \|\nabla^n\tilde{\mathcal I}_{k,u}^2\|_\infty&\lesssim \sum_{j\in\Lambda_u}N_{j,k}^{n}\left(M_{j,k}N_{j,k}^{-1}e^{-N_{j,k}^2t/4} + N_{j,k}^{-m}M_{j,k}^m\right),\\
 \|\nabla^n\tilde{\mathcal{I}}_{k,u}^3\|_{\infty} &\lesssim \sum_{j\in\Lambda_u} N_{J_d, k}^{n+\varepsilon} e^{-2N_{j,k+1}^2t_k},\\
 \|\nabla^n (\Id-2\frac{\nabla\otimes\nabla}{\Delta})p\|_\infty&\lesssim \sum_{j\in\Lambda_u}\big(M_{j,k}N_{j,k}^{n-1+\varepsilon}e^{-N_{j,k}^2(t-t_k)/4}+N_{j,k}^{n-m+\varepsilon}M_{j,k}^{m}\big).
\end{split}
\end{equation}
Moreover we have
\[
\|\nabla^n\mathcal{I}_{k,u}^{>t_k}\|_{\infty} 
\lesssim \sum_{j\in\Lambda_u} N_{J_d, k}^{n+\varepsilon} e^{-2N_{j,k+1}^2t_k}.
\]

Combining the terms and estimates above yields for $t\geq t_k$,
\[
\begin{split}
\|\nabla^n(\bar R_{k,u}(t) - \mathcal{I}_{k,u}(t))\|_\infty 
&\leq \|\nabla^n(\tilde{\mathcal{I}}_{k,u}^1 +\tilde{\mathcal{I}}_{k,u}^2 +\tilde{\mathcal{I}}_{k,u}^3 -(\Id-2\frac{\nabla\otimes\nabla}{\Delta})p)\|_\infty \\
&\quad   +\| \nabla^n(\tilde{\mathcal{I}}_{k,u}^{<t_k} - \mathcal{I}_{k,u}^{<t_k})\|_\infty + \|\nabla^n\mathcal{I}_{k,u}^{>t_k}\|_\infty \\
&\lesssim_{\varepsilon,m} \sum_{j\in\Lambda_u} \big( N_{j,k}^{n}e^{-N_{j,k}^2t/4}( N_{j,k}^{2}t_k + M_{j,k}N_{j,k}^{-1}) + N_{j,k}^{n-m}M_{j,k}^m \big)\\
&\qquad +\sum_{j\in\Lambda_u}\big(M_{j,k}N_{j,k}^{n-1}e^{-N_{j,k}^2(t-t_k)/4}+N_{j,k}^{n-m}M_{j,k}^{m}\big)\\
&\qquad + \sum_{j\in\Lambda_u} N_{J_d, k}^{n+\varepsilon} e^{-2N_{j,k+1}^2t_k} +t_k N_{J_d,k}^{n+2+\varepsilon}.
\end{split}
\]
Choosing $\gamma$ properly in the definition of $M_{j,k}$ \eqref{M_definition}, for instance, $\gamma=\frac12$
one can obtain 
\[
\begin{split}
\|\nabla^n(\bar R_{k,u}(t) - \mathcal{I}_{k,u}(t))\|_{L^\infty_{x,t}} 
&\lesssim N_{j,k}^{n-\frac12}e^{-N_{j,k}^2t/4}, \quad t\geq t_k.
\end{split}
\]
Recall $t_k= N_{J_d,k}^{-4}$. Hence we have
\[
\begin{split}
\|\nabla^n(\bar R_{k,u}(t) - \mathcal{I}_{k,u}(t))\|_{L^\infty_{x,t}}  &\lesssim_\varepsilon N_{J_d, k}^{n+\varepsilon}
 = t_k^{-(n+\varepsilon)/4}, \quad t \in [0,t_k].
\end{split}
\]
Therefore we can choose $A$, $b$, and $m$ sufficiently large and $\varepsilon$ sufficiently small so that
\[
\|\nabla^n(\bar R_{k,u}(t) - \mathcal{I}_{k,u}(t))\|_\infty \leq_n \varepsilon_0 N^{-\alpha}_{1,k}(t^{-\frac{n}{2} +\alpha}+1), \quad \forall \varepsilon_0>0, \quad 0<\alpha < \frac{1}{10}
\]
for all $t\geq 0$ and $1\leq n\leq \bar n$.

In the case $n=0$, we apply estimates \eqref{eq:bound_before_t_k_2D} and \eqref{eq:bound_before_t_k_3D} directly.

{\textbf{Step II.}}
The goal of this step is to show 
\begin{equation}\label{step2}
\begin{split}
\|\nabla^n(\bar R_{k,c}(t) - \mathcal{I}_{k,c}(t))\|_{\infty} &\leq \varepsilon_0 N^{-\alpha}_{1,k}(t^{-\frac{n}{2} +\alpha}+1), \qquad  n=1,2,\dots,\bar n,\\
\|\bar R_{k,c}(t) - \mathcal{I}_{k,c}(t)\|_{\infty} &\lesssim_\varepsilon \varepsilon_0N_{1,k}^{-\alpha}+\mathbbm{1}_{t\leq t_k}N_{J_d, k}^{\varepsilon}, \qquad \text{if} \qquad d=2,\\
\|\bar R_{k,c}(t) - \mathcal{I}_{k,c}(t)\|_{\infty} &\lesssim \varepsilon_0N_{1,k}^{-\alpha}+\mathbbm1_{t\leq t_k}, \qquad \text{if} \qquad d\geq 3.
\end{split}
\end{equation}
First we derive rough estimates on $[0,t_k]$. For $n\ge0$, applying \eqref{eq:a_bounds} yields
\[
\begin{split}
\|\nabla^n\mathcal I_{k,c}\|_{\infty}
&\le \int_0^t\sum_{j\in\Lambda_B}N_{j,k+1}^2e^{-2N_{j,k+1}^2s}
\Big\|\nabla^n\mathbb Q\!\Big(a_{j,k+1,B}^2(\eta_1^j\otimes\eta_1^j-\eta_2^j\otimes\eta_2^j)\Big)\Big\|_\infty ds\\
&\lesssim_\varepsilon \int_0^t\sum_{j\in\Lambda_B}N_{j,k+1}^2e^{-2N_{j,k+1}^2s}
\|\nabla^n(a_{j,k+1,B}^2)\|_{C^\varepsilon}\,ds\\
&\lesssim_\varepsilon N_{J_d,k}^{n+\varepsilon}.
\end{split}
\]
In dimension $d\ge3$, this estimate can be improved. Indeed, due to $N_{j,k}=N_{1,k}$ for all $j$, the second identity in \eqref{a-convex-integration-identity} implies
\[
\sum_{j\in\Lambda_B}a_{j,k+1,B}^2(\eta_1^j\otimes\eta_1^j-\eta_2^j\otimes\eta_2^j)
=2\mathcal D\sum_{j\in\Lambda_B}N_{1,k}\psi_{j,k,c}+p\Id.
\]
Therefore
\[
\begin{split}
&\quad\nabla^n\mathbb Q\sum_{j\in\Lambda_B}a_{j,k+1,B}^2(\eta_1^j\otimes\eta_1^j-\eta_2^j\otimes\eta_2^j)\\
&=\nabla^n\mathcal R\mathbb P\div\left(2\mathcal D\sum_{j\in\Lambda_B}N_{1,k}\psi_{j,k,c}+p\Id\right)\\
&=2\nabla^n\mathcal R\Delta\mathbb P\sum_{j\in\Lambda_B}N_{1,k}\psi_{j,k,c}\\
&=2\nabla^n\mathcal R\Delta\mathbb P\sum_{j\in\Lambda_B}N_{1,k}^{-1}\phi_k*(a_{j,k,B}\varphi_{j,k}\eta_1^j\sin(N_{1,k}\eta^j\!\cdot x))\\
&=2\nabla^n\mathcal R\Delta\sum_{j\in\Lambda_B}N_{1,k}^{-1}\phi_k*(a_{j,k,B}\varphi_{j,k}\eta_1^j\sin(N_{1,k}\eta^j\!\cdot x))\\
&\quad-2\nabla^n\mathcal R\nabla\sum_{j\in\Lambda_B}N_{1,k}^{-1}\phi_k*(\nabla(a_{j,k,B}\varphi_{j,k})\!\cdot\!\eta_1^j\sin(N_{1,k}\eta^j\!\cdot x)).
\end{split}
\]
Since $\mathcal R\Delta$ is a local operator, we infer
\[
\Big\|\nabla^n\mathbb Q\sum_{j\in\Lambda_B}a_{j,k+1,B}^2(\eta_1^j\otimes\eta_1^j-\eta_2^j\otimes\eta_2^j)\Big\|_\infty
\lesssim N_{J_d,k}^n,
\]
and hence
\[
\|\nabla^n\mathcal I_{k,c}\|_\infty\lesssim N_{J_d,k}^n,\qquad d\ge3.
\]

On the other hand, by the definition of $\bar R_{k,c}$ in \eqref{tensor-app},
\[
\|\nabla^n\bar R_{k,c}(t)\|_\infty
\lesssim \sum_{j\in\Lambda_B}N_{j,k}^n e^{-N_{j,k}^2t}
\lesssim N_{J_d,k}^n.
\]
Therefore, for $t\in[0,t_k]$,
\begin{equation}\label{eq:bound_before_t_k_2D_c}
\|\nabla^n(\bar R_{k,c}(t)-\mathcal I_{k,c}(t))\|_\infty
\lesssim_\varepsilon N_{J_d,k}^{n+\varepsilon},\qquad d=2,
\end{equation}
and
\begin{equation}\label{eq:bound_before_t_k_3D_c}
\|\nabla^n(\bar R_{k,c}(t)-\mathcal I_{k,c}(t))\|_\infty
\lesssim N_{J_d,k}^{n},\qquad d\ge3.
\end{equation}

We next treat $t\ge t_k$ and split
\[
\begin{split}
\mathcal I_{k,c}(t)
&=-\int_0^{t_k}e^{(t-s)\Delta}\sum_{j\in\Lambda_B}N_{j,k+1}^2e^{-2N_{j,k+1}^2s}
\mathbb Q\!\left(a_{j,k+1,B}^2(\eta_1^j\otimes\eta_1^j-\eta_2^j\otimes\eta_2^j)\right)\,ds\\
&\quad-\int_{t_k}^{t}e^{(t-s)\Delta}\sum_{j\in\Lambda_B}N_{j,k+1}^2e^{-2N_{j,k+1}^2s}
\mathbb Q\!\left(a_{j,k+1,B}^2(\eta_1^j\otimes\eta_1^j-\eta_2^j\otimes\eta_2^j)\right)\,ds\\
&=:\mathcal I_{k,c}^{<t_k}(t)+\mathcal I_{k,c}^{>t_k}(t).
\end{split}
\]
Define
\[
\tilde{\mathcal I}_{k,c}^{<t_k}
:=-e^{(t-t_k)\Delta}\int_0^{t_k}\sum_{j\in\Lambda_B}N_{j,k+1}^2e^{-2N_{j,k+1}^2s}
\mathbb Q\!\left(a_{j,k+1,B}^2(\eta_1^j\otimes\eta_1^j-\eta_2^j\otimes\eta_2^j)\right)\,ds.
\]
Using \eqref{eq:a_bounds},
\[
\begin{split}
&\quad\|\nabla^n(e^{t\Delta}-\Id)\mathbb Q(a_{j,k+1,B}^2(\eta_1^j\otimes\eta_1^j-\eta_2^j\otimes\eta_2^j))\|_\infty\\
&\lesssim t\|\nabla^n\Delta\mathbb Q(a_{j,k+1,B}^2(\eta_1^j\otimes\eta_1^j-\eta_2^j\otimes\eta_2^j))\|_\infty\\
&\lesssim_\varepsilon t\,N_{J_d,k}^{n+2+\varepsilon}.
\end{split}
\]
Hence
\[
\begin{split}
\|\nabla^n(\mathcal I_{k,c}^{<t_k}-\tilde{\mathcal I}_{k,c}^{<t_k})\|_\infty
&\lesssim_\varepsilon \int_0^{t_k}\sum_{j\in\Lambda_B}N_{j,k+1}^2e^{-2N_{j,k+1}^2s}\,t_kN_{J_d,k}^{n+2+\varepsilon}\,ds\\
&\lesssim t_kN_{J_d,k}^{n+2+\varepsilon}.
\end{split}
\]

Next, by the second identity in \eqref{a-convex-integration-identity} and \eqref{eq:Q_identities},
\begin{equation}\label{eq:quadratic_identity_c}
\begin{split}
&\quad\frac12\sum_{j\in\Lambda_B}\mathbb Q\!\left(a_{j,k+1,B}^2(\eta_1^j\otimes\eta_1^j-\eta_2^j\otimes\eta_2^j)\right)\\
&=\sum_{j\in\Lambda_B}N_{j,k}\Big(\newD\psi_{j,k,c}+(\Id-2\Delta^{-1}\nabla\otimes\nabla)\div\psi_{j,k,c}\Big).
\end{split}
\end{equation}
Therefore
\[
\begin{split}
\tilde{\mathcal I}_{k,c}^{<t_k}
&=-\frac12e^{(t-t_k)\Delta}\sum_{j\in\Lambda_B}(1-e^{-2N_{j,k+1}^2t_k})
\mathbb Q\!\left(a_{j,k+1,B}^2(\eta_1^j\otimes\eta_1^j-\eta_2^j\otimes\eta_2^j)\right)\\
&=-\sum_{j\in\Lambda_B}N_{j,k}e^{(t-t_k)\Delta}\newD\psi_{j,k,c}\\
&\quad+\frac12e^{(t-t_k)\Delta}\sum_{j\in\Lambda_B}e^{-2N_{j,k+1}^2t_k}
\mathbb Q\!\left(a_{j,k+1,B}^2(\eta_1^j\otimes\eta_1^j-\eta_2^j\otimes\eta_2^j)\right)\\
&\quad-(\Id-2\frac{\nabla\otimes\nabla}{\Delta})p_c\\
&=\bar R_{k,c}+\tilde{\mathcal I}_{k,c}^1+\tilde{\mathcal I}_{k,c}^2+\tilde{\mathcal I}_{k,c}^3-(\Id-2\frac{\nabla\otimes\nabla}{\Delta})p_c,
\end{split}
\]
where
\[
\tilde{\mathcal I}_{k,c}^1:=
-\sum_{j\in\Lambda_B}N_{j,k}(e^{-N_{j,k}^2(t-t_k)}-e^{-N_{j,k}^2t})\newD\psi_{j,k,c},
\]
\[
\tilde{\mathcal I}_{k,c}^2:=
\sum_{j\in\Lambda_B}N_{j,k}(e^{-N_{j,k}^2(t-t_k)}-e^{(t-t_k)\Delta})\newD\psi_{j,k,c},
\]
\[
\tilde{\mathcal I}_{k,c}^3:=
\frac12e^{(t-t_k)\Delta}\sum_{j\in\Lambda_B}e^{-2N_{j,k+1}^2t_k}
\mathbb Q\!\left(a_{j,k+1,B}^2(\eta_1^j\otimes\eta_1^j-\eta_2^j\otimes\eta_2^j)\right),
\]
and
\[
-p_c=\sum_{j\in\Lambda_B}N_{j,k}e^{(t-t_k)\Delta}\div\psi_{j,k,c}.
\]
Moreover,
\[
\tilde{\mathcal I}_{k,c}^2=
-\phi_k*\sum_{j\in\Lambda_B}N_{j,k}^{-1}\newD
\big[e^{(t-t_k)\Delta},a_{j,k,B}\varphi_{j,k}\big]
\big(\eta_1^j\sin(N_{j,k}\eta^j\cdot x)\big),
\]
and
\[
-p_c=
\phi_k*\sum_{j\in\Lambda_B}N_{j,k}^{-1}e^{(t-t_k)\Delta}
\big(\eta_1^j\cdot\nabla(a_{j,k,B}\varphi_{j,k})\,\eta_1^j\sin(N_{j,k}\eta^j\cdot x)\big).
\]
Hence, by \eqref{heat-decay-estimate} and Calder\'on--Zygmund boundedness of $\Delta^{-1}\nabla\otimes\nabla$,
\[
\begin{split}
\|\nabla^n(\Id-2\frac{\nabla\otimes\nabla}{\Delta})p_c\|_\infty
\lesssim \sum_{j\in\Lambda_B}\Big(
M_{j,k}N_{j,k}^{n-1+\varepsilon}e^{-N_{j,k}^2(t-t_k)/4}
+N_{j,k}^{n-m+\varepsilon}M_{j,k}^m\Big).
\end{split}
\]
Also,
\[
\|\nabla^n\tilde{\mathcal I}_{k,c}^1\|_\infty
\lesssim \sum_{j\in\Lambda_B}N_{j,k}^{n}e^{-N_{j,k}^2t}(e^{N_{j,k}^2t_k}-1)
\lesssim \sum_{j\in\Lambda_B}N_{j,k}^{n+2}t_k e^{-N_{j,k}^2t},
\]
\[
\|\nabla^n\tilde{\mathcal I}_{k,c}^2\|_\infty
\lesssim \sum_{j\in\Lambda_B}N_{j,k}^{n}
\Big(M_{j,k}N_{j,k}^{-1}e^{-N_{j,k}^2t/4}+N_{j,k}^{-m}M_{j,k}^m\Big)
\]
by Lemma~\ref{l:commutator}, and
\[
\|\nabla^n\tilde{\mathcal I}_{k,c}^3\|_\infty
\lesssim \sum_{j\in\Lambda_B}N_{J_d,k}^{n+\varepsilon}e^{-2N_{j,k+1}^2t_k}.
\]
For the late-time part,
\[
\begin{split}
\|\nabla^n\mathcal I_{k,c}^{>t_k}\|_\infty
&\le \int_{t_k}^{t}\sum_{j\in\Lambda_B}N_{j,k+1}^2e^{-2N_{j,k+1}^2s}
\Big\|\nabla^n\mathbb Q\!\left(a_{j,k+1,B}^2(\eta_1^j\otimes\eta_1^j-\eta_2^j\otimes\eta_2^j)\right)\Big\|_\infty ds\\
&\lesssim \sum_{j\in\Lambda_B}N_{J_d,k}^{n+\varepsilon}e^{-2N_{j,k+1}^2t_k}.
\end{split}
\]

Combining all terms, for $t\ge t_k$,
\[
\begin{split}
\|\nabla^n(\bar R_{k,c}-\mathcal I_{k,c})\|_\infty
&\le \|\nabla^n(\bar R_{k,c}-\tilde{\mathcal I}_{k,c}^{<t_k})\|_\infty
+\|\nabla^n(\tilde{\mathcal I}_{k,c}^{<t_k}-\mathcal I_{k,c}^{<t_k})\|_\infty
+\|\nabla^n\mathcal I_{k,c}^{>t_k}\|_\infty\\
&\lesssim_{\varepsilon,m}\sum_{j\in\Lambda_B}\Big(
N_{j,k}^{n}e^{-N_{j,k}^2t/4}(N_{j,k}^2t_k+M_{j,k}N_{j,k}^{-1})
+N_{j,k}^{n-m}M_{j,k}^m\Big)\\
&\quad +\sum_{j\in\Lambda_B}\Big(
M_{j,k}N_{j,k}^{n-1+\varepsilon}e^{-N_{j,k}^2(t-t_k)/4}
+N_{j,k}^{n-m+\varepsilon}M_{j,k}^m\Big)\\
&\quad +\sum_{j\in\Lambda_B}N_{J_d,k}^{n+\varepsilon}e^{-2N_{j,k+1}^2t_k}
+t_kN_{J_d,k}^{n+2+\varepsilon}.
\end{split}
\]
Again, the choice of $\gamma=\frac12$ in \eqref{M_definition} and $t_k=N_{J_d,k}^{-4}$ gives
\[
N_{j,k}^{n}e^{-N_{j,k}^2t/4}(N_{j,k}^2t_k+M_{j,k}N_{j,k}^{-1})
\lesssim N_{j,k}^{n-\frac12}e^{-N_{j,k}^2t/4},
\]
and the same choice makes all non-decaying terms arbitrarily small by taking $m,A,b$ large enough.
For $t\in[0,t_k]$, \eqref{eq:bound_before_t_k_2D_c}--\eqref{eq:bound_before_t_k_3D_c} imply
\[
\|\nabla^n(\bar R_{k,c}-\mathcal I_{k,c})\|_{L^\infty_{x,t}}
\lesssim_\varepsilon N_{J_d,k}^{n+\varepsilon}
=t_k^{-(n+\varepsilon)/4}.
\]
Hence, for any $\varepsilon_0>0$ and $0<\alpha<\frac1{10}$, we can choose $A,b,m$ sufficiently large and $\varepsilon$ sufficiently small so that
\[
\|\nabla^n(\bar R_{k,c}(t)-\mathcal I_{k,c}(t))\|_\infty
\le \varepsilon_0N_{1,k}^{-\alpha}(t^{-n/2+\alpha}+1),\qquad n=1,\dots,\bar n.
\]
For $n=0$, the claimed two bounds in \eqref{step2} follow directly from
\eqref{eq:bound_before_t_k_2D_c} and \eqref{eq:bound_before_t_k_3D_c} on $[0,t_k]$,
together with the $t\ge t_k$ estimate above.

{\textbf{Step III.}}
The goal of the last step is to show 
\begin{equation}\label{step3}
\begin{split}
\|\nabla^n(\bar H_k(t) - \mathcal{I}_{k,B}(t))\|_{\infty} &\leq \varepsilon_0 N^{-\alpha}_{1,k}(t^{-\frac{n}{2} +\alpha}+1), \qquad  n=1,2,\dots,\bar n,\\
\|\bar H_k(t) - \mathcal{I}_{k,B}(t)\|_{\infty} &\lesssim_\varepsilon \varepsilon_0N_{1,k}^{-\alpha}+\mathbbm{1}_{t\leq t_k}N_{J_d, k}^{\varepsilon}, \qquad \text{if} \qquad d=2,\\
\|\bar H_k(t) - \mathcal{I}_{k,B}(t)\|_{\infty} &\lesssim \varepsilon_0N_{1,k}^{-\alpha}+\mathbbm1_{t\leq t_k}, \qquad \text{if} \qquad d\geq 3.
\end{split}
\end{equation}
As before, we start with showing rough estimates on $[0,t_k]$. For $n\ge0$,
\[
\begin{split}
\|\nabla^n\mathcal I_{k,B}\|_{\infty}
&\le \int_0^t\sum_{j\in\Lambda_B}N_{j,k+1}^2e^{-2N_{j,k+1}^2s}
\Big\|\nabla^n\mathbb Q_s\!\Big(a_{j,k+1,B}^2(\eta_1^j\otimes\eta_2^j-\eta_2^j\otimes\eta_1^j)\Big)\Big\|_\infty ds\\
&\lesssim_\varepsilon \int_0^t\sum_{j\in\Lambda_B}N_{j,k+1}^2e^{-2N_{j,k+1}^2s}
\|\nabla^n(a_{j,k+1,B}^2)\|_{C^\varepsilon}\,ds\\
&\lesssim_\varepsilon N_{J_d,k}^{n+\varepsilon},
\end{split}
\]
where we used \eqref{eq:a_bounds} again. We improve it in dimension $d\ge3$. Since $N_{j,k}=N_{1,k}$ for all $j$, we have by the third identity in \eqref{a-convex-integration-identity}, 
\[
\sum_{j\in\Lambda_B}a_{j,k+1,B}^2(\eta_1^j\otimes\eta_2^j-\eta_2^j\otimes\eta_1^j)
=2\mathcal D_s\sum_{j\in\Lambda_B}N_{1,k}\psi_{j,k,B}.
\]
It then follows 
\[
\begin{split}
&\quad\nabla^n\mathbb Q_s\sum_{j\in\Lambda_B}a_{j,k+1,B}^2(\eta_1^j\otimes\eta_2^j-\eta_2^j\otimes\eta_1^j)\\
&=\nabla^n\mathcal R_s\mathbb P\div\left(2\mathcal D_s\sum_{j\in\Lambda_B}N_{1,k}\psi_{j,k,B}\right)\\
&=2\nabla^n\mathcal R_s\Delta\mathbb P\sum_{j\in\Lambda_B}N_{1,k}\psi_{j,k,B}\\
&=2\nabla^n\mathcal R_s\Delta\mathbb P\sum_{j\in\Lambda_B}N_{1,k}^{-1}\phi_k*(a_{j,k,B}\varphi_{j,k}\eta_2^j\sin(N_{1,k}\eta^j\!\cdot x))\\
&=2\nabla^n\mathcal R_s\Delta\sum_{j\in\Lambda_B}N_{1,k}^{-1}\phi_k*(a_{j,k,B}\varphi_{j,k}\eta_2^j\sin(N_{1,k}\eta^j\!\cdot x))\\
&\quad-2\nabla^n\mathcal R_s\nabla\sum_{j\in\Lambda_B}N_{1,k}^{-1}\phi_k*(\nabla(a_{j,k,B}\varphi_{j,k})\!\cdot\!\eta_2^j\sin(N_{1,k}\eta^j\!\cdot x)).
\end{split}
\]
Since $\mathcal R_s\Delta$ is also local, this yields
\[
\Big\|\nabla^n\mathbb Q_s\sum_{j\in\Lambda_B}a_{j,k+1,B}^2(\eta_1^j\otimes\eta_2^j-\eta_2^j\otimes\eta_1^j)\Big\|_\infty
\lesssim N_{J_d,k}^n,
\]
and then
\[
\|\nabla^n\mathcal I_{k,B}\|_\infty\lesssim N_{J_d,k}^n,\qquad d\ge3.
\]
On the other hand, by the definition of $\bar H_k$ in \eqref{tensor-app},
\[
\|\nabla^n\bar H_k(t)\|_\infty
\lesssim \sum_{j\in\Lambda_B}N_{j,k}^ne^{-N_{j,k}^2t}
\lesssim N_{J_d,k}^n.
\]
Therefore, for $t\in[0,t_k]$,
\begin{equation}\label{eq:bound_before_t_k_2D_B}
\|\nabla^n(\bar H_k(t)-\mathcal I_{k,B}(t))\|_\infty
\lesssim_\varepsilon N_{J_d,k}^{n+\varepsilon},\qquad d=2,
\end{equation}
and
\begin{equation}\label{eq:bound_before_t_k_3D_B}
\|\nabla^n(\bar H_k(t)-\mathcal I_{k,B}(t))\|_\infty
\lesssim N_{J_d,k}^{n},\qquad d\ge3.
\end{equation}

For $t\ge t_k$, we analogously split
\[
\begin{split}
\mathcal I_{k,B}(t)
&=-\int_0^{t_k}e^{(t-s)\Delta}\sum_{j\in\Lambda_B}N_{j,k+1}^2e^{-2N_{j,k+1}^2s}
\mathbb Q_s\!\left(a_{j,k+1,B}^2(\eta_1^j\otimes\eta_2^j-\eta_2^j\otimes\eta_1^j)\right)\,ds\\
&\quad-\int_{t_k}^{t}e^{(t-s)\Delta}\sum_{j\in\Lambda_B}N_{j,k+1}^2e^{-2N_{j,k+1}^2s}
\mathbb Q_s\!\left(a_{j,k+1,B}^2(\eta_1^j\otimes\eta_2^j-\eta_2^j\otimes\eta_1^j)\right)\,ds\\
&=:\mathcal I_{k,B}^{<t_k}(t)+\mathcal I_{k,B}^{>t_k}(t).
\end{split}
\]
Define
\[
\tilde{\mathcal I}_{k,B}^{<t_k}
:=-e^{(t-t_k)\Delta}\int_0^{t_k}\sum_{j\in\Lambda_B}N_{j,k+1}^2e^{-2N_{j,k+1}^2s}
\mathbb Q_s\!\left(a_{j,k+1,B}^2(\eta_1^j\otimes\eta_2^j-\eta_2^j\otimes\eta_1^j)\right)\,ds.
\]
Estimate \eqref{eq:a_bounds} implies
\[
\begin{split}
&\quad\|\nabla^n(e^{t\Delta}-\Id)\mathbb Q_s(a_{j,k+1,B}^2(\eta_1^j\otimes\eta_2^j-\eta_2^j\otimes\eta_1^j))\|_\infty\\
&\lesssim t\|\nabla^n\Delta\mathbb Q_s(a_{j,k+1,B}^2(\eta_1^j\otimes\eta_2^j-\eta_2^j\otimes\eta_1^j))\|_\infty\\
&\lesssim_\varepsilon t\,N_{J_d,k}^{n+2+\varepsilon},
\end{split}
\]
and thus
\[
\begin{split}
\|\nabla^n(\mathcal I_{k,B}^{<t_k}-\tilde{\mathcal I}_{k,B}^{<t_k})\|_\infty
&\lesssim_\varepsilon \int_0^{t_k}\sum_{j\in\Lambda_B}N_{j,k+1}^2e^{-2N_{j,k+1}^2s}\,t_kN_{J_d,k}^{n+2+\varepsilon}\,ds\\
&\lesssim t_kN_{J_d,k}^{n+2+\varepsilon}.
\end{split}
\]
By the third identity in \eqref{a-convex-integration-identity}, \eqref{eq:Qs_identities}, and $\mathcal D_s=\mathbb Q_s\mathcal D_s+\Id\,\div$, we get
\begin{equation}\label{eq:quadratic_identity_B}
\frac12\sum_{j\in\Lambda_B}\mathbb Q_s\!\left(a_{j,k+1,B}^2(\eta_1^j\otimes\eta_2^j-\eta_2^j\otimes\eta_1^j)\right)
=\sum_{j\in\Lambda_B}N_{j,k}\Big(\mathcal D_s\psi_{j,k,B}-\Id\,\div\psi_{j,k,B}\Big).
\end{equation}
Therefore, we deduce
\[
\begin{split}
\tilde{\mathcal I}_{k,B}^{<t_k}
&=-\frac12e^{(t-t_k)\Delta}\sum_{j\in\Lambda_B}(1-e^{-2N_{j,k+1}^2t_k})
\mathbb Q_s\!\left(a_{j,k+1,B}^2(\eta_1^j\otimes\eta_2^j-\eta_2^j\otimes\eta_1^j)\right)\\
&=-\sum_{j\in\Lambda_B}N_{j,k}e^{(t-t_k)\Delta}\mathcal D_s\psi_{j,k,B}\\
&\quad+\frac12e^{(t-t_k)\Delta}\sum_{j\in\Lambda_B}e^{-2N_{j,k+1}^2t_k}
\mathbb Q_s\!\left(a_{j,k+1,B}^2(\eta_1^j\otimes\eta_2^j-\eta_2^j\otimes\eta_1^j)\right)\\
&\quad+\Id\,q_{k,B}\\
&=\bar H_k+\tilde{\mathcal I}_{k,B}^1+\tilde{\mathcal I}_{k,B}^2+\tilde{\mathcal I}_{k,B}^3+\Id\,q_{k,B},
\end{split}
\]
where
\[
\tilde{\mathcal I}_{k,B}^1:=
-\sum_{j\in\Lambda_B}N_{j,k}(e^{-N_{j,k}^2(t-t_k)}-e^{-N_{j,k}^2t})\mathcal D_s\psi_{j,k,B},
\]
\[
\tilde{\mathcal I}_{k,B}^2:=
\sum_{j\in\Lambda_B}N_{j,k}(e^{-N_{j,k}^2(t-t_k)}-e^{(t-t_k)\Delta})\mathcal D_s\psi_{j,k,B},
\]
\[
\tilde{\mathcal I}_{k,B}^3:=
\frac12e^{(t-t_k)\Delta}\sum_{j\in\Lambda_B}e^{-2N_{j,k+1}^2t_k}
\mathbb Q_s\!\left(a_{j,k+1,B}^2(\eta_1^j\otimes\eta_2^j-\eta_2^j\otimes\eta_1^j)\right),
\]
and
\[
q_{k,B}:=e^{(t-t_k)\Delta}\sum_{j\in\Lambda_B}N_{j,k}\div\psi_{j,k,B}.
\]
We can further write $\tilde{\mathcal I}_{k,B}^2$ as a commutator
\[
\tilde{\mathcal I}_{k,B}^2=
-\phi_k*\sum_{j\in\Lambda_B}N_{j,k}^{-1}\mathcal D_s
\big[e^{(t-t_k)\Delta},a_{j,k,B}\varphi_{j,k}\big]
\big(\eta_2^j\sin(N_{j,k}\eta^j\cdot x)\big),
\]
and
\[
q_{k,B}=
\phi_k*\sum_{j\in\Lambda_B}N_{j,k}^{-1}e^{(t-t_k)\Delta}
\big(\eta_2^j\cdot\nabla(a_{j,k,B}\varphi_{j,k})\sin(N_{j,k}\eta^j\cdot x)\big).
\]
It then follows from \eqref{heat-decay-estimate},
\[
\begin{split}
\|\nabla^n q_{k,B}\|_\infty
\lesssim \sum_{j\in\Lambda_B}\Big(
M_{j,k}N_{j,k}^{n-1+\varepsilon}e^{-N_{j,k}^2(t-t_k)/4}
+N_{j,k}^{n-m+\varepsilon}M_{j,k}^m\Big).
\end{split}
\]
Applying Lemma~\ref{l:commutator} yields
\[
\|\nabla^n\tilde{\mathcal I}_{k,B}^1\|_\infty
\lesssim \sum_{j\in\Lambda_B}N_{j,k}^{n}e^{-N_{j,k}^2t}(e^{N_{j,k}^2t_k}-1)
\lesssim \sum_{j\in\Lambda_B}N_{j,k}^{n+2}t_ke^{-N_{j,k}^2t},
\]
\[
\|\nabla^n\tilde{\mathcal I}_{k,B}^2\|_\infty
\lesssim \sum_{j\in\Lambda_B}N_{j,k}^{n}
\Big(M_{j,k}N_{j,k}^{-1}e^{-N_{j,k}^2t/4}+N_{j,k}^{-m}M_{j,k}^m\Big)
\]
and
\[
\|\nabla^n\tilde{\mathcal I}_{k,B}^3\|_\infty
\lesssim\sum_{j\in\Lambda_B}N_{J_d,k}^{n+\varepsilon}e^{-2N_{j,k+1}^2t_k}.
\]
On the other hand, we have
\[
\begin{split}
\|\nabla^n\mathcal I_{k,B}^{>t_k}\|_\infty
&\le \int_{t_k}^{t}\sum_{j\in\Lambda_B}N_{j,k+1}^2e^{-2N_{j,k+1}^2s}
\Big\|\nabla^n\mathbb Q_s\!\left(a_{j,k+1,B}^2(\eta_1^j\otimes\eta_2^j-\eta_2^j\otimes\eta_1^j)\right)\Big\|_\infty ds\\
&\lesssim \sum_{j\in\Lambda_B}N_{J_d,k}^{n+\varepsilon}e^{-2N_{j,k+1}^2t_k}.
\end{split}
\]

Combining all terms, for $t\ge t_k$,
\[
\begin{split}
\|\nabla^n(\bar H_k-\mathcal I_{k,B})\|_\infty
&\le \|\nabla^n(\bar H_k-\tilde{\mathcal I}_{k,B}^{<t_k})\|_\infty
+\|\nabla^n(\tilde{\mathcal I}_{k,B}^{<t_k}-\mathcal I_{k,B}^{<t_k})\|_\infty
+\|\nabla^n\mathcal I_{k,B}^{>t_k}\|_\infty\\
&\lesssim_{\varepsilon,m}\sum_{j\in\Lambda_B}\Big(
N_{j,k}^{n}e^{-N_{j,k}^2t/4}(N_{j,k}^2t_k+M_{j,k}N_{j,k}^{-1})
+N_{j,k}^{n-m}M_{j,k}^m\Big)\\
&\quad +\sum_{j\in\Lambda_B}\Big(
M_{j,k}N_{j,k}^{n-1+\varepsilon}e^{-N_{j,k}^2(t-t_k)/4}
+N_{j,k}^{n-m+\varepsilon}M_{j,k}^m\Big)\\
&\quad +\sum_{j\in\Lambda_B}N_{J_d,k}^{n+\varepsilon}e^{-2N_{j,k+1}^2t_k}
+t_kN_{J_d,k}^{n+2+\varepsilon}.
\end{split}
\]
Again, thanks to the choice of the parameters $\gamma, m, A,b$ and $t_k$, we have
\[
\|\nabla^n(\bar H_k-\mathcal I_{k,B})\|_{L^\infty_{x,t}}
\lesssim N_{j,k}^{n-\frac12}e^{-N_{j,k}^2t/4}, \quad t\geq t_k.
\]
For $t\in[0,t_k]$, \eqref{eq:bound_before_t_k_2D_B} and \eqref{eq:bound_before_t_k_3D_B} imply
\[
\|\nabla^n(\bar H_k-\mathcal I_{k,B})\|_{L^\infty_{x,t}}
\lesssim_\varepsilon N_{J_d,k}^{n+\varepsilon}
=t_k^{-(n+\varepsilon)/4}.
\]
Hence, as before, 
\[
\|\nabla^n(\bar H_k(t)-\mathcal I_{k,B}(t))\|_\infty
\le \varepsilon_0N_{1,k}^{-\alpha}(t^{-n/2+\alpha}+1),\qquad n=1,\dots,\bar n
\]
for any $\varepsilon_0>0$ and $0<\alpha<\frac1{10}$. 
For $n=0$, the claimed two bounds in \eqref{step3} follow directly from
\eqref{eq:bound_before_t_k_2D_B} and \eqref{eq:bound_before_t_k_3D_B} on $[0,t_k]$,
together with the $t\ge t_k$ estimate above.

Finally, using $\bar R_k=\bar R_{k,u}+\bar R_{k,c}$ and Step I--Step II, we get
\[
\|\nabla^n(\bar R_k-\mathcal I_{k,u}-\mathcal I_{k,c})\|_\infty
\le \|\nabla^n(\bar R_{k,u}-\mathcal I_{k,u})\|_\infty+\|\nabla^n(\bar R_{k,c}-\mathcal I_{k,c})\|_\infty,
\]
with the corresponding $n=0$ bounds in $d=2$ and $d\ge3$. Together with Step III, this proves all claims of Proposition~\ref{barv-I_estimate_proposition}.
\end{proof}

We now estimate the discrepancy between $R_k$ and $\bar R_k$ by combining the previous proposition with the decomposition derived above.

\begin{proposition}\label{difference_estimate_proposition}
For any $\varepsilon_0>0$, sufficiently small $\alpha>0$, and $\bar n \in \mathbb{N}$, we have
\[
\|\nabla^n(R_k(t)- \bar R_k(t) )\|_{\infty} \leq_n \varepsilon_0 N_{1,k}^{-\alpha} (t^{-\frac{n}{2}+\alpha}+1), \qquad n=1,2,\dots, \bar n,
\]
\[
\|R_k(t)- \bar R_k(t)\|_{\infty} \lesssim_\varepsilon \varepsilon_0N_{1,k}^{-\alpha}+\mathbbm{1}_{t\leq t_k}N_{J_d, k}^{\varepsilon}, \qquad \text{if} \qquad d=2,
\]
\[
\|R_k(t)- \bar R_k(t)\|_{\infty} \lesssim \varepsilon_0N_{1,k}^{-\alpha}+\mathbbm1_{t\leq t_k}, \qquad \text{if} \qquad d\geq 3
\]
provided the parameters $A$ and $b$ are sufficiently large.
\end{proposition}
\begin{proof}
Recall
\[
\begin{split}
\psi_{j,k,u}(x)&= N_{j,k}^{-2}\phi_k*(a_{j,k,u}(x)\varphi_{j,k}(x)\eta_1^j\sin(N_{j,k}\eta^j\cdot x)),\\
\psi_{j,k,B}(x)&= N_{j,k}^{-2}\phi_k*(a_{j,k,B}(x)\varphi_{j,k}(x)\eta_2^j\sin(N_{j,k}\eta^j\cdot x)),\\
\psi_{j,k,c}(x)&=N_{j,k}^{-2}\phi_k*(a_{j,k,B}(x)\varphi_{j,k}(x)\eta_1^j\sin(N_{j,k}\eta^j\cdot x))
\end{split}
\]
for $k\geq 1$.
We compute
\begin{equation}\label{Delta-psi-u}
\begin{split}
\Delta \psi_{j,k,u}(x)
&=-a_{j,k,u}(x)\varphi_{j,k}(x)\sin(N_{j,k}\eta^j\cdot x)\eta_1^j\\
&\quad+2N_{j,k}^{-1}\nabla(a_{j,k,u}\varphi_{j,k})\cdot\eta^j\cos(N_{j,k}\eta^j\cdot x)\eta_1^j\\
&\quad +N_{j,k}^{-2}\Delta (a_{j,k,u}\varphi_{j,k})\sin(N_{j,k}\eta^j\cdot x)\eta_1^j\\
&\quad +\Delta(\psi_{j,k,u}-N_{j,k}^{-2}a_{j,k,u}\varphi_{j,k}\eta_1^j\sin(N_{j,k}\eta^j\cdot x)),
\end{split}
\end{equation}
where the first term is the major term and the other three are lower order errors (note the last term is the error coming from removing the mollification). We can compute $\Delta \psi_{j,k,B}$ and $\Delta \psi_{j,k,c}$ in a similar way.

We write 
\begin{equation}\label{bar-v-k1}
\bar v_{k+1}\otimes \bar v_{k+1}
=\sum_{j\in\Lambda_u\cup\Lambda_B}\bar v_{j,k+1}\otimes \bar v_{j,k+1}+\sum_{j,j'\in \Lambda_u\cup\Lambda_B, j\neq j'}\bar v_{j,k+1}\otimes \bar v_{j',k+1},
\end{equation}
\begin{equation}\label{bar-h-k1}
\bar h_{k+1}\otimes \bar h_{k+1}
=\sum_{j\in\Lambda_B}\bar h_{j,k+1}\otimes \bar h_{j,k+1}+\sum_{j,j'\in \Lambda_B, j\neq j'}\bar h_{j,k+1}\otimes \bar h_{j',k+1},
\end{equation}
and hence
\begin{equation}\label{bar-v-minus-k1}
\begin{split}
&\quad\bar v_{k+1}\otimes \bar v_{k+1}-\bar h_{k+1}\otimes \bar h_{k+1}\\
&=\sum_{j\in\Lambda_u}\bar v_{j,k+1}\otimes \bar v_{j,k+1}+\sum_{j\in\Lambda_B}(\bar v_{j,k+1}\otimes \bar v_{j,k+1}-\bar h_{j,k+1}\otimes \bar h_{j,k+1})\\
&\quad+\sum_{j,j'\in \Lambda_u\cup\Lambda_B, j\neq j'}\bar v_{j,k+1}\otimes \bar v_{j',k+1}
-\sum_{j,j'\in \Lambda_B, j\neq j'}\bar h_{j,k+1}\otimes \bar h_{j',k+1}.
\end{split}
\end{equation}
Applying \eqref{Delta-psi-u}, we further compute the first term in \eqref{bar-v-minus-k1}
\begin{equation}\label{unidirectional_terms}
\begin{split}
\sum_{j\in\Lambda_u}\bar v_{j,k+1}\otimes \bar v_{j,k+1}&=\sum_{j\in\Lambda_u} N_{j,k+1}^2 e^{-2N_{j,k+1}
^2t} \Delta \psi_{j,k+1,u}\otimes \Delta \psi_{j,k+1,u}\\
&=\sum_{j\in\Lambda_u} N_{j,k+1}^2 e^{-2N_{j,k+1}
^2t} a_{j,k+1,u}^2 \eta_1^j\otimes \eta_1^j\\
&\hspace{-8em}+ \sum_{j\in\Lambda_u} N_{j,k+1}^2 e^{-2N_{j,k+1}
^2t} a_{j,k+1,u}^2(\varphi^2_{j,k+1} \sin^2(N_{j,k+1}\eta^j\cdot x) -1)\eta_1^j\otimes \eta_1^j + \mathfrak{E}_{k,u}^1,
\end{split}
\end{equation}
where $\mathfrak{E}_{k,u}^1$ collects error terms with derivatives on $a_{j,k+1}\varphi_{j,k+1}$ and the terms removing the mollifier (the last term in \eqref{Delta-psi-u}).

Similarly, using the corresponding expansions of $\Delta\psi_{j,k+1,c}$ and $\Delta\psi_{j,k+1,B}$, we decompose the second term of \eqref{bar-v-minus-k1}
\begin{equation}\label{unidirectional_terms_c}
\begin{split}
&\sum_{j\in\Lambda_B}\left(\bar v_{j,k+1}\otimes \bar v_{j,k+1}-\bar h_{j,k+1}\otimes \bar h_{j,k+1}\right)\\
&=\sum_{j\in\Lambda_B} N_{j,k+1}^2 e^{-2N_{j,k+1}^2t}
\left(\Delta\psi_{j,k+1,c}\otimes\Delta\psi_{j,k+1,c}-\Delta\psi_{j,k+1,B}\otimes\Delta\psi_{j,k+1,B}\right)\\
&=\sum_{j\in\Lambda_B} N_{j,k+1}^2 e^{-2N_{j,k+1}^2t}
a_{j,k+1,B}^2(\eta_1^j\otimes\eta_1^j-\eta_2^j\otimes\eta_2^j)\\
&\quad+\sum_{j\in\Lambda_B} N_{j,k+1}^2 e^{-2N_{j,k+1}^2t}
a_{j,k+1,B}^2\big(\varphi^2_{j,k+1}\sin^2(N_{j,k+1}\eta^j\cdot x)-1\big)
(\eta_1^j\otimes\eta_1^j-\eta_2^j\otimes\eta_2^j)\\
&\quad+\mathfrak E_{k,c}^1-\mathfrak E_{k,B}^1,
\end{split}
\end{equation}
where $\mathfrak E_{k,c}^1$ and $\mathfrak E_{k,B}^1$ contain lower order terms with derivatives on $a_{j,k+1,B}\varphi_{j,k+1}$ and mollifier-removal errors from the $\psi_{j,k+1,c}$ and $\psi_{j,k+1,B}$ components respectively. We set
\[
\mathfrak E_k^1:=\mathfrak E_{k,u}^1+\mathfrak E_{k,c}^1-\mathfrak E_{k,B}^1.
\]

The third term in \eqref{bar-v-minus-k1} is written as
\begin{equation} \label{different_directions_terms}
\begin{split}
&\quad\sum_{j,j'\in \Lambda_u\cup\Lambda_B, j\neq j'}\bar v_{j,k+1}\otimes \bar v_{j',k+1}\\
&=\sum_{j,j'\in \Lambda_u, j\neq j'} N_{j,k+1}N_{j',k+1} e^{-N_{j,k+1}
^2t-N_{j',k+1}^2t} \Delta \psi_{j,k+1,u}\otimes \Delta \psi_{j',k+1,u}\\
&\quad +\sum_{j,j'\in \Lambda_B, j\neq j'} N_{j,k+1}N_{j',k+1} e^{-N_{j,k+1}
^2t-N_{j',k+1}^2t} \Delta \psi_{j,k+1,c}\otimes \Delta \psi_{j',k+1,c}\\
&\quad+\sum_{j\in \Lambda_u, j'\in\Lambda_B} N_{j,k+1}N_{j',k+1} e^{-N_{j,k+1}
^2t-N_{j',k+1}^2t} \Delta \psi_{j,k+1,u}\otimes \Delta \psi_{j',k+1,c}.
\end{split}
\end{equation}
We continue to expand the first term of \eqref{different_directions_terms} as
\begin{equation} \label{different_directions_terms_1}
\begin{split}
&\quad\sum_{j,j'\in \Lambda_u, j\neq j'} N_{j,k+1}N_{j',k+1} e^{-N_{j,k+1}
^2t-N_{j',k+1}^2t} \Delta \psi_{j,k+1,u}\otimes \Delta \psi_{j',k+1,u}\\
&=\sum_{j,j'\in \Lambda_u, j\neq j'} N_{j,k+1}N_{j',k+1} e^{-N_{j,k+1}
^2t-N_{j',k+1}^2t}a_{j,k+1,u}a_{j',k+1,u}\varphi_{j,k+1}\varphi_{j',k+1}  \\
&\quad\cdot \sin(N_{j,k+1}\eta^j\cdot x)\sin(N_{j',k+1}\eta_{j'}\cdot x)\eta_1^j\otimes \eta_1^{j'}
+\mathfrak{E}_{k,u}^2,
\end{split}
\end{equation}
where $\mathfrak{E}_{k,u}^2$ contains error terms with derivatives of $a_{j,k+1,u}a_{j',k+1,u}\varphi_{j,k+1}\varphi_{j',k+1}$ and mollifier-removal terms. Similarly we expand the second and third terms of \eqref{different_directions_terms} as
\begin{equation} \label{different_directions_terms_2}
\begin{split}
&\quad\sum_{j,j'\in \Lambda_B, j\neq j'} N_{j,k+1}N_{j',k+1} e^{-N_{j,k+1}
^2t-N_{j',k+1}^2t} \Delta \psi_{j,k+1,c}\otimes \Delta \psi_{j',k+1,c}\\
&=\sum_{j,j'\in \Lambda_B, j\neq j'} N_{j,k+1}N_{j',k+1} e^{-N_{j,k+1}
^2t-N_{j',k+1}^2t}a_{j,k+1,B}a_{j',k+1,B}\varphi_{j,k+1}\varphi_{j',k+1}  \\
&\quad\cdot \sin(N_{j,k+1}\eta^j\cdot x)\sin(N_{j',k+1}\eta_{j'}\cdot x)\eta_1^j\otimes \eta_1^{j'}
+\mathfrak{E}_{k,B}^2,
\end{split}
\end{equation}
where $\mathfrak{E}_{k,B}^2$ collects error terms with derivatives of $a_{j,k+1,B}a_{j',k+1,B}\varphi_{j,k+1}\varphi_{j',k+1}$ and mollifier-removal terms, and
\begin{equation} \label{different_directions_terms_3}
\begin{split}
&\quad\sum_{j\in \Lambda_u, j'\in\Lambda_B} N_{j,k+1}N_{j',k+1} e^{-N_{j,k+1}
^2t-N_{j',k+1}^2t} \Delta \psi_{j,k+1,u}\otimes \Delta \psi_{j',k+1,c}\\
&=\sum_{j\in \Lambda_u, j'\in\Lambda_B}  N_{j,k+1}N_{j',k+1} e^{-N_{j,k+1}
^2t-N_{j',k+1}^2t}a_{j,k+1,u}a_{j',k+1,B}\varphi_{j,k+1}\varphi_{j',k+1}  \\
&\quad\cdot \sin(N_{j,k+1}\eta^j\cdot x)\sin(N_{j',k+1}\eta_{j'}\cdot x)\eta_1^j\otimes \eta_1^{j'}
+\mathfrak{E}_{k,c}^2,
\end{split}
\end{equation}
where $\mathfrak{E}_{k,c}^2$ contains error terms with derivatives on $a_{j,k+1,u}a_{j',k+1,B}\varphi_{j,k+1}\varphi_{j',k+1}$ and mollifier-removal terms.

Denote
\[\mathfrak E_k^2:=\mathfrak E_{k,u}^2+\mathfrak E_{k,B}^2+\mathfrak E_{k,c}^2.\]

The last term in \eqref{bar-v-minus-k1} can be written as
\begin{equation} \label{different_directions_terms_bb}
\begin{split}
&\quad\sum_{j,j'\in \Lambda_B, j\neq j'}\bar h_{j,k+1}\otimes \bar h_{j',k+1}\\
&=\sum_{j,j'\in \Lambda_B, j\neq j'} N_{j,k+1}N_{j',k+1} e^{-N_{j,k+1}
^2t-N_{j',k+1}^2t} \Delta \psi_{j,k+1,B}\otimes \Delta \psi_{j',k+1,B}\\
&=\sum_{j,j'\in \Lambda_B, j\neq j'} N_{j,k+1}N_{j',k+1} e^{-N_{j,k+1}
^2t-N_{j',k+1}^2t}a_{j,k+1,B}a_{j',k+1,B}\varphi_{j,k+1}\varphi_{j',k+1}  \\
&\quad\cdot \sin(N_{j,k+1}\eta^j\cdot x)\sin(N_{j',k+1}\eta_{j'}\cdot x)\eta_2^j\otimes \eta_2^{j'}
+\mathfrak{E}_{k,B}^3,
\end{split}
\end{equation}
where $\mathfrak{E}_{k,B}^3$ contains error terms with derivatives on $a_{j,k+1,B}a_{j',k+1,B}\varphi_{j,k+1}\varphi_{j',k+1}$ and mollifier-removal terms. 

We summarize the error terms
\begin{equation}\label{E1+E2terms}
\mathfrak{E}_k^1 +\mathfrak{E}_k^2 -\mathfrak{E}_{k,B}^3=\mathfrak E_{k,u}^1+\mathfrak E_{k,c}^1-\mathfrak E_{k,B}^1+\mathfrak E_{k,u}^2+\mathfrak E_{k,B}^2+\mathfrak E_{k,c}^2-\mathfrak{E}_{k,B}^3.
\end{equation}
Denote
\[\mathfrak E_{k,u}^1= \sum_{j\in\Lambda_u} N_{j,k+1}^2 e^{-2N_{j,k+1}
^2t} E_{j,k,u}+\mathfrak E_{k,u,m}^1\]
with $E_{j,k,u}$ containing the terms with derivatives of $a_{j,k+1,u}\varphi_{j,k+1}$ and $\mathfrak E_{k,u,m}^1$ collecting the mollification error; and similarly
\[
\begin{split}
\mathfrak E_{k,c}^1&= \sum_{j\in\Lambda_B} N_{j,k+1}^2 e^{-2N_{j,k+1}
^2t} E_{j,k,c}+\mathfrak E_{k,c,m}^1, \\
\mathfrak E_{k,B}^1&= \sum_{j\in\Lambda_B} N_{j,k+1}^2 e^{-2N_{j,k+1}
^2t} E_{j,k,B}+\mathfrak E_{k,B,m}^1
\end{split}
\]
where $E_{j,k,c}$ and $E_{j,k,B}$ contain the terms with derivatives of $a_{j,k+1,B}\varphi_{j,k+1}$, and $\mathfrak E_{k,c,m}^1$ and $\mathfrak E_{k,B,m}^1$ contain the corresponding mollification errors.

Analogously, we denote
\[\mathfrak E_{k,u}^2=\sum_{j,j'\in\Lambda_u, j\neq j'} N_{j,k+1}N_{j',k+1} e^{-N_{j,k+1}
^2t-N_{j',k+1}^2t}E_{j,j',k,u}+\mathfrak E_{k,u,m}^2,\]
\[\mathfrak E_{k,B}^2=\sum_{j,j'\in\Lambda_B, j\neq j'} N_{j,k+1}N_{j',k+1} e^{-N_{j,k+1}
^2t-N_{j',k+1}^2t}E_{j,j',k,B}+\mathfrak E_{k,B,m}^2,\]
\[\mathfrak E_{k,c}^2=\sum_{j\in\Lambda_u, j'\in \Lambda_B} N_{j,k+1}N_{j',k+1} e^{-N_{j,k+1}
^2t-N_{j',k+1}^2t}E_{j,j',k,c}+\mathfrak E_{k,c,m}^2,\]
\[\mathfrak E_{k,B}^3=\sum_{j,j'\in\Lambda_B, j\neq j'} N_{j,k+1}N_{j',k+1} e^{-N_{j,k+1}
^2t-N_{j',k+1}^2t}E_{j,j',k,B}+\mathfrak E_{k,B,m}^3.\]

We put all the mollification errors together as
\[
\begin{split}
&\quad\mathfrak E_{k,u,m}^1+\mathfrak E_{k,c,m}^1-\mathfrak E_{k,B,m}^1+\mathfrak E_{k,u,m}^2+\mathfrak E_{k,c,m}^2+\mathfrak E_{k,B,m}^2-\mathfrak E_{k,B,m}^3\\
&=\sum_{j,j'\in \Lambda_u\cup\Lambda_B}N_{j,k+1}N_{j',k+1}e^{-N_{j,k+1}^2t-N_{j',k+1}^2t}F_{j,j',k}
\end{split}
\]
for some $F_{j,j',k}$.

We use $E_{j,k,u}$ as a prototype. The terms $E_{j,k,c}$ and $E_{j,k,B}$ have the same structure, with substitutions
\[
(a_{j,k+1,u},\eta_1^j)\rightsquigarrow (a_{j,k+1,B},\eta_1^j),\qquad
(a_{j,k+1,u},\eta_1^j)\rightsquigarrow (a_{j,k+1,B},\eta_2^j).
\]
Likewise, $E_{j,j',k,u}$, $E_{j,j',k,c}$, $E_{j,j',k,B}$ and the terms coming from $\mathfrak E_{k,B}^3$ all have the same lower order structure.
In particular, the term $E_{j,k,u}$ takes the form
\begin{equation}\notag
\begin{split}
E_{j,k,u}&=-2N_{j,k+1}^{-1}a_{j,k+1,u}\varphi_{j,k+1}\nabla(a_{j,k+1,u}\varphi_{j,k+1})\cdot\eta^j\sin(2N_{j,k+1}\eta^{j}\cdot x)\eta_1^j\otimes \eta_1^j\\
&\quad+4N_{j,k+1}^{-2}(\nabla(a_{j,k+1,u}\varphi_{j,k+1})\cdot\eta^j)^2\cos^2(N_{j,k+1}\eta^{j}\cdot x)\eta_1^j\otimes \eta_1^j\\
&\quad-2N_{j,k+1}^{-2}a_{j,k+1,u}\varphi_{j,k+1}\Delta(a_{j,k+1,u}\varphi_{j,k+1})\sin^2(N_{j,k+1}\eta^{j}\cdot x)\eta_1^j\otimes \eta_1^j\\
&\quad+2N_{j,k+1}^{-3}\Delta(a_{j,k+1,u}\varphi_{j,k+1})\nabla(a_{j,k+1,u}\varphi_{j,k+1})\cdot\eta^j\sin(2N_{j,k+1}\eta^{j}\cdot x)\eta_1^j\otimes \eta_1^j\\
&\quad+N_{j,k+1}^{-4}(\Delta(a_{j,k+1,u}\varphi_{j,k+1}))^2\sin^2(N_{j,k+1}\eta^{j}\cdot x)\eta_1^j\otimes \eta_1^j\\
&=:E^1_{j,k}+E^2_{j,k}+E^3_{j,k}+E^4_{j,k}+E^5_{j,k},
\end{split}
\end{equation}
 and the mollifier remainder is represented by one of them as
\begin{align*}
    F_{j,j',k,u}&=\Delta(\psi_{j,k+1,u}-N_{j,k+1}^{-2}a_{j,k+1,u}\varphi_{j,k+1}\eta_1^j\sin(N_{j,k+1}\eta^j\cdot x))\otimes\Delta\psi_{j',k+1}\\
    &\quad+\Delta(N_{j,k+1}^{-2}a_{j,k+1,u}\varphi_{j,k+1}\eta_1^j\sin(N_{j,k+1}\eta^j\cdot x))\\
    &\qquad\otimes \Delta(\psi_{j',k+1,u}-N_{j',k+1}^{-2}a_{j',k+1,u}\varphi_{j',k+1}\eta_1^{j'}\sin(N_{j',k+1}\eta^{j'}\cdot x)).
\end{align*}

\noindent
{\bf Summarizing the form of $R_k$.}

By the definition of $R_k$ in \eqref{tensor-RH}, employing \eqref{bar-v-minus-k1}, \eqref{unidirectional_terms}, \eqref{unidirectional_terms_c}, \eqref{different_directions_terms}, and \eqref{different_directions_terms_bb}, we can write
\[
R_k=\mathcal I_k+\mathcal J_k^1+\mathcal J_k^2+\mathcal E_k,
\]
where
\[
\begin{split}
\mathcal I_k
&:= -\int_0^{t} e^{(t-s)\Delta }\sum_{j\in\Lambda_u} N_{j,k+1}^2 e^{-2N_{j,k+1}^2s}\,
\mathbb Q\!\left(a_{j,k+1,u}^2 \eta_1^j\otimes \eta_1^j\right) \, ds\\
&\quad -\int_0^{t} e^{(t-s)\Delta }\sum_{j\in\Lambda_B} N_{j,k+1}^2 e^{-2N_{j,k+1}^2s}\,
\mathbb Q\!\left(a_{j,k+1,B}^2 (\eta_1^j\otimes \eta_1^j-\eta_2^j\otimes \eta_2^j)\right) \, ds,
\end{split}
\]
\[
\begin{split}
\mathcal J_k^1
&:= -\int_0^{t} e^{(t-s)\Delta }\sum_{j\in\Lambda_u} N_{j,k+1}^2 e^{-2N_{j,k+1}^2s}\\
&\qquad\quad\cdot\mathbb Q\!\left(a_{j,k+1,u}^2(\varphi^2_{j,k+1}\sin^2(N_{j,k+1}\eta^j\!\cdot x)-1)\eta_1^j\otimes\eta_1^j\right)\,ds\\
&\quad -\int_0^{t} e^{(t-s)\Delta }\sum_{j\in\Lambda_B} N_{j,k+1}^2 e^{-2N_{j,k+1}^2s}\\
&\qquad\quad\cdot\mathbb Q\!\left(a_{j,k+1,B}^2(\varphi^2_{j,k+1}\sin^2(N_{j,k+1}\eta^j\!\cdot x)-1)(\eta_1^j\otimes\eta_1^j-\eta_2^j\otimes\eta_2^j)\right)\,ds,
\end{split}
\]
\[
\begin{split}
\mathcal J_k^2
&:= -\int_0^t e^{(t-s)\Delta}\mathbb Q\Bigg[
\sum_{j,j'\in\Lambda_u\cup\Lambda_B,\ j\neq j'} N_{j,k+1}N_{j',k+1}e^{-N_{j,k+1}^2s-N_{j',k+1}^2s}\\
&\qquad\cdot a_{j,k+1}a_{j',k+1}\varphi_{j,k+1}\varphi_{j',k+1}\sin(N_{j,k+1}\eta^j\!\cdot x)\sin(N_{j',k+1}\eta^{j'}\!\cdot x)\,\eta^j\otimes\eta^{j'}\\
&\qquad\quad-\sum_{j,j'\in\Lambda_B,\ j\neq j'} N_{j,k+1}N_{j',k+1}e^{-N_{j,k+1}^2s-N_{j',k+1}^2s}\\
&\qquad\cdot a_{j,k+1,B}a_{j',k+1,B}\varphi_{j,k+1}\varphi_{j',k+1}\sin(N_{j,k+1}\eta^j\!\cdot x)\sin(N_{j',k+1}\eta^{j'}\!\cdot x)\,\eta_2^j\otimes\eta_2^{j'}
\Bigg]ds,
\end{split}
\]
and
\[
\mathcal E_k:=-\int_0^{t} e^{(t-s)\Delta }\mathbb Q(\mathfrak E_k^1+\mathfrak E_k^2-\mathfrak E_{k,B}^3)\,ds.
\]
We remark that in $\mathcal J_k^2$, $a_{j,k+1}=a_{j,k+1,u}$ for $j\in \Lambda_u$ and $a_{j,k+1}=a_{j,k+1,B}$ for $j\in \Lambda_B$; and the same convention applies to $a_{j',k+1}$.

Since the main term $\mathcal I_k$ is already analyzed by Proposition~\ref{barv-I_estimate_proposition}, it remains to estimate $\mathcal J_k^1,\mathcal J_k^2,\mathcal E_k$.



To estimate $\mathcal{J}_k^1$, we start by writing
\[
\mathcal J_k^1=\mathcal J_{k,u}^1+\mathcal J_{k,B}^1
\]
according to the $\Lambda_u$ and $\Lambda_B$ sums in \eqref{different_directions_terms}. The estimate of $\mathcal J_{k,u}^1$ can be done similarly as in \cite{CDP}. We only highlight the main steps as follows.
Applying $\eta^j\cdot \eta_1^j=0$ gives
\begin{equation}\label{div-osc}
\begin{split}
&\quad\div (a_{j,k+1,u}^2(x)(\varphi^2_{j,k+1} \sin^2(N_{j,k+1}\eta^j\cdot x) -1) \eta_1^j\otimes \eta_1^j)\\
&= \nabla a_{j,k+1,u}^2 \cdot \eta_1^j (\varphi^2_{j,k+1} \sin^2(N_{j,k+1}\eta^j\cdot x) -1)\eta_1^j.
\end{split}
\end{equation}
 Recall that $\varphi^2_{j,k+1}(x) \sin^2(N_{j,k+1}\eta^j\cdot x)-1$ is zero mean and $2\pi/M_{j,k+1}$-periodic. It then follows from
Lemma~\ref{l:oscillation_estimate} that
\begin{equation}\label{est-anti-div-osc}
\begin{split}
&\quad\big\|e^{(t-s)\Delta }\mathcal R\mathbb P\big(\nabla a_{j,k+1,u}^2 \cdot \eta_1^j (\varphi^2_{j,k+1} \sin^2(N_{j,k+1}\eta^j\cdot x) -1)\big)\big\|_{C^\varepsilon}\\
&\lesssim_\varepsilon M_{j,k+1}^{-1}\|\nabla(a_{j,k+1,u}^2)\|_{C^\varepsilon}e^{-M_{j,k+1}^2(t-s)/4}+M_{j,k+1}^{-m}\|\nabla(a_{j,k+1,u}^2)\|_{C^{m,\varepsilon}},
\end{split}
\end{equation}
\begin{equation}\label{est-derivative-anti-div-osc}
\begin{split}
&\quad\big\|\nabla^n e^{(t-s)\Delta }\mathcal R\mathbb P\big(\nabla a_{j,k+1,u}^2 \cdot \eta_1^j (\varphi^2_{j,k+1} \sin^2(N_{j,k+1}\eta^j\cdot x) -1)\big)\big\|_{C^\varepsilon}\\
&\lesssim N_{j,k+1}^{n-1}N_{J_d,k}^{1+\varepsilon}e^{-M_{j,k+1}^2(t-s)/4}+M_{j,k+1}^{-m+n}N_{J_d,k}^{m+1+\varepsilon}
\end{split}
\end{equation}
for $n\geq1$.

Since $\mathbb Q=\mathcal R\mathbb P\div$, applying \eqref{div-osc} and \eqref{est-anti-div-osc}, we deduce
\[
\begin{split}
\| \mathcal{J}_{k,u}^1 \|_{L^\infty} 
&\lesssim \sum_{j}  N_{J_d,k}^{1+\varepsilon}\big(  M_{j,k+1}^{-1}e^{-M_{j,k+1}^2t/4}+M_{j,k+1}^{-m}N_{J_d,k}^{m} \big),
\end{split}
\]
\[
\begin{split}
\| \nabla^n \mathcal{J}_{k,u}^1 \|_{L^\infty} 
&\lesssim \sum_{j}N_{J_d,k}^{1+\varepsilon}\big( N_{j,k+1}^{n-1} e^{-M_{j,k+1}^2t/4}+M_{j,k+1}^{-m+n}N_{J_d,k}^{m} \big), \quad n\geq 1.
\end{split}
\]

Regarding $\mathcal{J}_{k,B}^1 $, using $\eta_1^j\cdot\eta^j=\eta_2^j\cdot\eta^j=0$, we have
\begin{multline}\label{div-osc-B}
\div\Big(a_{j,k+1,B}^2(x)(\varphi^2_{j,k+1}\sin^2(N_{j,k+1}\eta^j\cdot x)-1)(\eta_1^j\otimes\eta_1^j-\eta_2^j\otimes\eta_2^j)\Big)\\
=\nabla a_{j,k+1,B}^2\cdot\eta_1^j\,(\varphi^2_{j,k+1}\sin^2(N_{j,k+1}\eta^j\cdot x)-1)\eta_1^j\\
\quad-\nabla a_{j,k+1,B}^2\cdot\eta_2^j\,(\varphi^2_{j,k+1}\sin^2(N_{j,k+1}\eta^j\cdot x)-1)\eta_2^j.
\end{multline}
Applying Lemma~\ref{l:oscillation_estimate} to both components in \eqref{div-osc-B}, and using again that
$\varphi^2_{j,k+1}(x)\sin^2(N_{j,k+1}\eta^j\cdot x)-1$ is zero mean and $2\pi/M_{j,k+1}$-periodic, we obtain
\[
\begin{split}
\| \mathcal{J}_{k,B}^1 \|_{L^\infty}
&=\Big\| \int_0^{t} e^{(t-s)\Delta }\sum_{j\in\Lambda_B}N_{j,k+1}^2e^{-2N_{j,k+1}^2s}\\
&\qquad\qquad\qquad\times\mathcal R\mathbb P\div\!\Big(a_{j,k+1,B}^2(\varphi^2_{j,k+1}\sin^2(N_{j,k+1}\eta^j\cdot x)-1)\\
&\qquad\qquad\qquad\qquad\qquad\qquad\cdot(\eta_1^j\otimes\eta_1^j-\eta_2^j\otimes\eta_2^j)\Big)\,ds\Big\|_\infty\\
&\lesssim_\varepsilon \int_0^t \sum_{j} N_{j,k+1}^2e^{-2N_{j,k+1}^2s}
\Big(M_{j,k+1}^{-1}\|\nabla(a_{j,k+1,B}^2)\|_{C^\varepsilon}e^{-M_{j,k+1}^2(t-s)/4}\\
&\qquad\qquad\qquad\qquad\qquad\qquad\qquad\qquad\qquad
+M_{j,k+1}^{-m}\|\nabla(a_{j,k+1,B}^2)\|_{C^{m,\varepsilon}}\Big)\,ds\\
&\lesssim_\varepsilon \int_0^t \sum_j N_{j,k+1}^2e^{-2N_{j,k+1}^2s}
N_{J_d,k}^{1+\varepsilon}\Big(M_{j,k+1}^{-1}e^{-M_{j,k+1}^2(t-s)/4}+M_{j,k+1}^{-m}N_{J_d,k}^m\Big)\,ds\\
&\lesssim \sum_j N_{J_d,k}^{1+\varepsilon}
\Big(M_{j,k+1}^{-1}e^{-M_{j,k+1}^2t/4}+M_{j,k+1}^{-m}N_{J_d,k}^{m}\Big).
\end{split}
\]
For $n\ge1$, similarly,
\[
\begin{split}
\| \nabla^n \mathcal{J}_{k,B}^1 \|_{L^\infty}
&\lesssim_\varepsilon \int_0^t \sum_j N_{j,k+1}^2e^{-2N_{j,k+1}^2s}
N_{J_d,k}^{1+\varepsilon}
\Big(N_{j,k+1}^{n-1}e^{-M_{j,k+1}^2(t-s)/4}+M_{j,k+1}^{-m+n}N_{J_d,k}^{m}\Big)\,ds\\
&\lesssim \sum_j N_{J_d,k}^{1+\varepsilon}
\Big(N_{j,k+1}^{n-1}e^{-M_{j,k+1}^2t/4}+M_{j,k+1}^{-m+n}N_{J_d,k}^{m}\Big).
\end{split}
\]
Combining the $\Lambda_u$ and $\Lambda_B$ bounds gives,
\[
\|\nabla^n \mathcal J_k^1\|_\infty
\lesssim \sum_{j}N_{J_d,k}^{1+\varepsilon}\big( N_{j,k+1}^{n-1} e^{-M_{j,k+1}^2t/4}+M_{j,k+1}^{-m+n}N_{J_d,k}^{m} \big), \quad n\geq 1.
\]

We note the same analysis above applies to:
\begin{itemize}
    \item the $\Lambda_B$--$\Lambda_B$ contribution in \eqref{different_directions_terms_2},
    \item the mixed $\Lambda_u$--$\Lambda_B$ contribution in \eqref{different_directions_terms_3},
    \item the subtracted $\bar h\otimes\bar h$ contribution in \eqref{different_directions_terms_bb}
\end{itemize}
(for the last one, replace $\eta_1^{j}\otimes\eta_1^{j'}$ by $\eta_2^{j}\otimes\eta_2^{j'}$). 
We thus proceed to estimate $\mathcal{J}_k^2$,
\begin{equation}
\begin{split}
&\|\nabla^n  \mathcal{J}_k^2(t) \|_{L^\infty}\\
&\quad=
\Big\| \int_0^{t} \nabla^n e^{(t-s)\Delta } \sum_{j\neq j'} N_{j,k+1}N_{j',k+1} e^{-N_{j,k+1}
^2s-N_{j',k+1}^2s}\\
&\qquad\times\mathbb Q(a_{j,k+1}a_{j',k+1}\varphi_{j,k+1}\varphi_{j',k+1} \sin(N_{j,k+1}\eta^j\cdot x)\sin(N_{j',k+1}\eta_{j'}\cdot x)\eta^j\otimes \eta_{j'}) \, ds \Big\|_\infty\\
&\quad\lesssim \int_0^{t} \sum_{j' < j} N_{j,k+1}N_{j',k+1} e^{-N_{j,k+1}
^2s-N_{j',k+1}^2s} N_{j,k+1}^{n+\varepsilon}e^{-\frac14 N_{j,k+1}
^2(t-s)}  \, ds\\
&\qquad +\int_0^{t} \sum_{j' < j} N_{j,k+1}N_{j',k+1} e^{-N_{j,k+1}
^2s-N_{j',k+1}^2s} N_{j,k+1}^{n-m}M_{j,k+1}^{m+\varepsilon}  \, ds\\
&\quad\eqcolon\mathcal{J}_k^{2,1} + \mathcal{J}_k^{2,2}
\end{split}
\end{equation}
where we bound $L^\infty$ norm by $C^\varepsilon$ norm, and apply the fact that $\mathbb Q$ is a Calder\'on--Zygmund operator and estimate \eqref{heat-decay-estimate}.

The major term $\mathcal{J}_k^{2,1}$ is estimated as
\begin{equation} \label{eq:NonlocalInteractions}
\begin{split}
\mathcal{J}_k^{2,1}&= \sum_{j'<j} e^{-\frac14N_{j,k+1}^2t}N^{n+1+\varepsilon}_{j,k+1}N_{j',k+1}\frac{1-e^{-(\frac34N_{j,k+1}^2+N_{j',k+1}^2)t}}{\frac34N_{j,k+1}^2+N_{j',k+1}^2}\\
& \lesssim \sum_{j'<j} e^{-\frac14N_{j,k+1}^2t}N^{n-1+\varepsilon}_{j,k+1}N_{j',k+1},
\end{split}
\end{equation}
while the term $\mathcal{J}_k^{2,2}$ is small since
\[
\mathcal{J}_k^{2,2} \lesssim \sum_{j'<j} N_{j,k+1}^{-1}N_{j',k+1} N_{j,k+1}^{n-m}M_{j,k+1}^{m+\varepsilon}.
\]
Therefore,
\[
\|\nabla^n\mathcal J_k^2\|_\infty
\lesssim \sum_{j'<j} e^{-\frac14N_{j,k+1}^2t}N^{n-1+\varepsilon}_{j,k+1}N_{j',k+1}
\,+\sum_{j'<j} N_{j,k+1}^{-1}N_{j',k+1} N_{j,k+1}^{n-m}M_{j,k+1}^{m+\varepsilon}.
\]

We move on to estimate $\mathcal{E}_k$.
It follows from \eqref{E1+E2terms} that
\[
\begin{split}
\mathcal{E}_k&=-\int_0^{t} e^{(t-s)\Delta } \mathbb Q(\mathfrak{E}_k^1 +\mathfrak{E}_k^2-\mathfrak E_{k,B}^3) \, ds\\
&= -\int_0^{t} e^{(t-s)\Delta }\sum_j N_{j,k+1}^2 e^{-2N_{j,k+1}
^2s} \mathbb Q E_{j,k}\, ds\\
&\quad -\int_0^{t} e^{(t-s)\Delta }\sum_{j\ne j'}N_{j,k+1}N_{j',k+1} e^{-N_{j,k+1}
^2s-N_{j',k+1}^2s} \mathbb Q E_{j,j',k}\, ds\\
&\quad-\int_0^{t} e^{(t-s)\Delta }\sum_{j,j'}N_{j,k+1}N_{j',k+1} e^{-N_{j,k+1}
^2s-N_{j',k+1}^2s} \mathbb Q F_{j,j',k}\, ds\\
&=:  \mathcal{E}_k^1+ \mathcal{E}_k^2 + \mathcal{E}_k^3.
\end{split}
\]
The terms $\mathcal E^1_k$ and $\mathcal E^3_k$ can be estimated as in \cite{CDP},
\[
\begin{split}
\|\nabla^n \mathcal E^1_k\|_\infty&\lesssim \sum_{j} \big(N_{j,k+1}^{n-1+\varepsilon}e^{-\frac14 N_{j,k+1}^2t} +N_{j,k+1}^{-2}M_{j,k+1}^{n+2+\varepsilon}\big),\\
 \|\nabla^n\mathcal E_k^3\|_\infty&\lesssim \sum_jN_{j,k+1}^{n+1+\varepsilon}\ell_{k+1}.
\end{split}
\]
In view of the structure of $\mathcal E^2_k$, it satisfies the same estimate of $\mathcal J^2_k$. Combining these estimates yields
\[ 
\begin{split}
&\quad\|\nabla^n(\mathcal{J}_k^1+\mathcal{J}_k^2 + \mathcal{E}_k)\|_\infty\\
&\leq \|\nabla^n\mathcal{J}_k^1\|_\infty+\|\nabla^n\mathcal{J}_k^2\|_\infty+\|\nabla^n\mathcal{E}_k^1\|_\infty+\|\nabla^n\mathcal{E}_k^2\|_{L^\infty}+\|\nabla^n\mathcal{E}_k^3\|_\infty\\
&\lesssim_{\varepsilon,m} \sum_{j} N_{j,k+1}^{n -1+ \varepsilon} N_{J_d,k} \big(e^{-\frac14 M_{j,k+1}^2t} + M_{j,k+1}^{-m}N_{J_d,k}^{m} \big)\\
&\quad \, \,  +\sum_{j'<j} N^{n-1+\varepsilon}_{j,k+1}N_{j',k+1}e^{-\frac14N_{j,k+1}^2t}\\
&\quad \, \, +\sum_{j'<j} N_{j,k+1}^{-1}N_{j',k+1} N_{j,k+1}^{n-m}M_{j,k+1}^{m+\varepsilon}\\
&\quad \, \, +\sum_{j}\big( N_{j,k+1}^{n-1+\varepsilon}e^{-\frac14 N_{j,k+1}^2t}+N_{j,k+1}^{-2}M_{j,k+1}^{n+ 2+\varepsilon}\big)
+\sum_jN_{j,k+1}^{n+1+\varepsilon}\ell_{k+1}.
\end{split}
\]
Again, for sufficiently large $A$, $b$, and $m$, there exists $\alpha> 0$, such that for any $\varepsilon_0 >0$ and $\bar n \in \mathbb{N}$,
\begin{equation}\label{est-error}
\|\nabla^n(\mathcal{J}_k^1+\mathcal{J}_k^2 + \mathcal{E}_k)\|_\infty \leq \varepsilon_0 N^{-\alpha}_{1,k+1}(t^{-\frac{n}{2} +\alpha}+1), \qquad n=1,2,\dots, \bar n.
\end{equation}

Finally, recalling 
\[\bar R_k=\bar R_{k,u}+\bar R_{k,B}, \quad \mathcal I_k=\mathcal I_{k,u}+\mathcal I_{k,c}, \quad R_k=\mathcal{I}_k+\mathcal{J}_k^1+\mathcal{J}_k^2 + \mathcal{E}_k,\]
it follows from the triangle inequality, \eqref{est-error} and Proposition~\ref{barv-I_estimate_proposition}
\[
\begin{split}
\|\nabla^n(R_k-\bar R_k)\|_\infty&=\|\nabla^n(R_k-I_k)+\nabla^n(I_k-\bar R_k)\|_\infty\\
& \leq \|\nabla^n(R_k-\mathcal{I}_k)\|_\infty + \|\nabla^n(\bar R_k-\mathcal{I}_k)\|_\infty\\
&\leq \|\nabla^n(\mathcal{J}_k^1+\mathcal{J}_k^2 + \mathcal{E}_k)\|_\infty + \|\nabla^n(\bar R_k-\mathcal{I}_k)\|_\infty\\
&\lesssim \varepsilon_0 N^{-\alpha}_{1,k+1}(t^{-\frac{n}{2} +\alpha}+1), \qquad n=1,2,\dots,\bar n,
\end{split}
\]
for any $\varepsilon_0 >0$, $\alpha \in(0, \frac{1}{10})$, $\bar n \in \mathbb{N}$, and sufficiently large $A$ and $b$.

The estimates for $n=0$ stated in the proposition follow in a similar way as in Proposition~\ref{barv-I_estimate_proposition}.

\end{proof}

\begin{proposition}\label{difference_estimate_H}
Taking $A$ and $b$ sufficiently large, we have for any $\varepsilon_0>0$, sufficiently small $\alpha>0$ and $\bar n \in \mathbb{N}$
\[
\|\nabla^n(H_k(t)- \bar H_k(t) )\|_{L^\infty} \leq_n \varepsilon_0 N_{1,k}^{-\alpha} (t^{-\frac{n}{2}+\alpha}+1), \qquad n=1,2,\dots, \bar n,
\]
\[
\|H_k(t)- \bar H_k(t)\|_{L^\infty} \lesssim_\varepsilon \varepsilon_0N_{1,k}^{-\alpha}+\mathbbm{1}_{t\leq t_k}N_{J_d, k}^{\varepsilon}, \qquad \text{if} \qquad d=2,
\]
\[
\|H_k(t)- \bar H_k(t)\|_{L^\infty} \lesssim \varepsilon_0N_{1,k}^{-\alpha}+\mathbbm1_{t\leq t_k}, \qquad \text{if} \qquad d\geq 3.
\]
\end{proposition}
\begin{proof}
Recall $\mathcal I_{k,B}$ from Proposition~\ref{barv-I_estimate_proposition}:
\[
\mathcal{I}_{k,B}
= -\int_0^{t} e^{(t-s)\Delta }\sum_{j\in \Lambda_B} N_{j,k+1}^2 e^{-2N_{j,k+1}^2s}
\mathbb Q_s \!\left(a_{j,k+1,B}^2 (\eta_1^j\otimes \eta_2^j-\eta_2^j\otimes \eta_1^j)\right) \, ds.
\]
By the triangle inequality,
\begin{equation}\label{eq:Hk_barHk_triangle}
\|\nabla^n(H_k-\bar H_k)\|_\infty
\le \|\nabla^n(H_k-\mathcal I_{k,B})\|_\infty+\|\nabla^n(\bar H_k-\mathcal I_{k,B})\|_\infty.
\end{equation}

We first estimate $H_k-\mathcal I_{k,B}$. Expanding
\[
\bar v_{k+1}\otimes \bar h_{k+1}-\bar h_{k+1}\otimes \bar v_{k+1}
=\sum_{j\in\Lambda_B}\big(\bar v_{j,k+1}\otimes \bar h_{j,k+1}-\bar h_{j,k+1}\otimes \bar v_{j,k+1}\big)
+\sum_{(j,j')\in\mathcal A_k}\mathcal T_{j,j',k},
\]
where $\mathcal A_k=((\Lambda_u\cup\Lambda_B)\times\Lambda_B)\setminus\{(j,j):j\in\Lambda_B\}$ and
\[
\mathcal T_{j,j',k}
:=\bar v_{j,k+1}\otimes \bar h_{j',k+1}-\bar h_{j',k+1}\otimes \bar v_{j,k+1},
\]
and using the same $\Delta\psi$ expansions as in Proposition~\ref{difference_estimate_proposition}, we can write
\[
H_k-\mathcal I_{k,B}=\widetilde{\mathcal J}_k^1+\widetilde{\mathcal J}_k^2+\widetilde{\mathcal E}_k,
\]
where:
\[
\begin{split}
\widetilde{\mathcal J}_k^1
&=-\int_0^t e^{(t-s)\Delta}\sum_{j\in\Lambda_B}N_{j,k+1}^2e^{-2N_{j,k+1}^2s}\\
&\qquad\cdot\mathbb Q_s\!\left(a_{j,k+1,B}^2(\varphi_{j,k+1}^2\sin^2(N_{j,k+1}\eta^j\!\cdot x)-1)
(\eta_1^j\otimes\eta_2^j-\eta_2^j\otimes\eta_1^j)\right)\,ds,
\end{split}
\]
$\widetilde{\mathcal J}_k^2$ collects the different-direction interactions $\mathcal T_{j,j',k}$, and $\widetilde{\mathcal E}_k$ collects all lower-order derivative and mollifier-removal terms.

For $\widetilde{\mathcal J}_k^1$, set $g_j=a_{j,k+1,B}^2(\varphi_{j,k+1}^2\sin^2(N_{j,k+1}\eta^j\!\cdot x)-1)$. Since
$\eta_1^j\cdot\eta^j=\eta_2^j\cdot\eta^j=0$,
\[
\begin{split}
\div\!\big(g_j(\eta_1^j\otimes\eta_2^j-\eta_2^j\otimes\eta_1^j)\big)
&=(\nabla a_{j,k+1,B}^2\cdot\eta_1^j)(\varphi_{j,k+1}^2\sin^2(N_{j,k+1}\eta^j\!\cdot x)-1) \eta_2^j\\
&\quad-(\nabla a_{j,k+1,B}^2\cdot\eta_2^j)(\varphi_{j,k+1}^2\sin^2(N_{j,k+1}\eta^j\!\cdot x)-1)\eta_1^j.
\end{split}
\]
Hence the same oscillation estimate as for $\mathcal J_k^1$ in Proposition~\ref{difference_estimate_proposition} applies (with $\mathbb Q_s$ in place of $\mathbb Q$), giving for $n\ge1$:
\[
\|\nabla^n\widetilde{\mathcal J}_k^1\|_\infty
\lesssim \sum_j N_{J_d,k}^{1+\varepsilon}
\Big(N_{j,k+1}^{n-1}e^{-M_{j,k+1}^2t/4}+M_{j,k+1}^{-m+n}N_{J_d,k}^{m}\Big),
\]
and the corresponding $n=0$ bound as in the $\mathcal J_k^1$ estimate above.

For $\widetilde{\mathcal J}_k^2$, the proof is identical to the $\mathcal J_k^2$ estimate above: one uses heat decay and Lemma~\ref{l:commutator}; replacing
$\eta_1^j\otimes\eta_1^{j'}$ by $\eta_1^j\otimes\eta_2^{j'}-\eta_2^{j'}\otimes\eta_1^j$ does not change the size bounds. Therefore $\widetilde{\mathcal J}_k^2$ satisfies the same estimate as $\mathcal J_k^2$ (up to an absolute constant).

For $\widetilde{\mathcal E}_k$, all lower-order terms have the same derivative/mollifier structure as in the $\mathcal E_k$ estimate above, and $\mathbb Q_s$ has the same Calder\'on--Zygmund mapping properties used there. Thus $\widetilde{\mathcal E}_k$ obeys the same bound as $\mathcal E_k$ (up to an absolute constant).

Combining these three bounds exactly as in the end of the proof of Proposition~\ref{difference_estimate_proposition}, we obtain: for any $\varepsilon_0>0$, $\bar n\in\mathbb N$, and $\alpha\in(0,\frac1{10})$, after choosing $A,b,m$ large enough,
\begin{equation}\label{eq:H_minus_I_bound}
\|\nabla^n(H_k-\mathcal I_{k,B})\|_\infty
\le \varepsilon_0 N_{1,k+1}^{-\alpha}\big(t^{-n/2+\alpha}+1\big),\qquad n=1,\dots,\bar n,
\end{equation}
and for $n=0$:
\[
\|H_k-\mathcal I_{k,B}\|_\infty
\lesssim_\varepsilon \varepsilon_0N_{1,k+1}^{-\alpha}+\mathbbm 1_{t\le t_k}N_{J_d,k}^{\varepsilon}, \quad d=2,
\]
\[
\|H_k-\mathcal I_{k,B}\|_\infty
\lesssim \varepsilon_0N_{1,k+1}^{-\alpha}+\mathbbm 1_{t\le t_k}, \quad d\ge3.
\]

Next, Proposition~\ref{barv-I_estimate_proposition} gives
\[
\|\nabla^n(\bar H_k-\mathcal I_{k,B})\|_\infty
\le \varepsilon_0 N_{1,k}^{-\alpha}\big(t^{-n/2+\alpha}+1\big),\qquad n=1,\dots,\bar n,
\]
and the corresponding two $n=0$ bounds in $d=2$ and $d\ge3$.
Insert these together with \eqref{eq:H_minus_I_bound} into \eqref{eq:Hk_barHk_triangle}; since
$N_{1,k+1}^{-\alpha}\le N_{1,k}^{-\alpha}$, after adjusting constants we conclude
\[
\|\nabla^n(H_k-\bar H_k)\|_\infty
\le \epsilon_0 N_{1,k}^{-\alpha}(t^{-n/2+\alpha}+1),\qquad n=1,\dots,\bar n,
\]
\[
\|H_k-\bar H_k\|_\infty
\lesssim_\varepsilon \varepsilon_0N_{1,k}^{-\alpha}+\mathbbm 1_{t\le t_k}N_{J_d,k}^{\varepsilon}, \quad d=2,
\]
\[
\|H_k-\bar H_k\|_\infty
\lesssim \varepsilon_0N_{1,k}^{-\alpha}+\mathbbm 1_{t\le t_k}, \quad d\ge3.
\]
This concludes the proof.
\end{proof}

\subsection{Estimates of the principal solution}

We show that the principal part $(v,h)$ satisfies a forced MHD system. At the same time we establish appropriate estimates for $(v,h)$ and the forcing terms. 

\begin{proposition}\label{prop_f_estimate}
The principal part $(v,h)$ satisfies
    \begin{equation}\label{v_equation}
    \begin{split}
 \partial_tv-\Delta v+\mathbb P\div (v\otimes v-h\otimes h)&=\mathbb P\div f_u,\\
 \partial_th-\Delta h+\mathbb P\div (v\otimes h-h\otimes v)&=\mathbb P\div f_B,
 \end{split}
    \end{equation}
    with errors $f_u$ and $f_B$ satisfying
\begin{align}\label{f_bound}
    \|\nabla^n f_u\|_{C^\kappa}+  \|\nabla^n f_B\|_{C^\kappa}\lesssim_n \varepsilon_0 (t^{-1 - \frac{n}{2} +\alpha}+1)
\end{align}
for some $0<\kappa<\alpha<1$.
In addition, we have
\begin{align}
    \|\nabla^nv_k(t)\|_{L^p}+ \|\nabla^nh_k(t)\|_{L^p}&\lesssim ((t/N_{1,k+1})^{\alpha}+2^{-k/p})t^{-\frac12(1+n)}+N_{1,k+1}^{-\alpha},\label{vk-pointwise-bounds}
    \end{align}
    \begin{align}
    \|\nabla^nv(t)\|_{L^\infty}+\|\nabla^n\bar v(t)\|_{L^\infty}+ \|\nabla^nh(t)\|_{L^\infty}+\|\nabla^n\bar h(t)\|_{L^\infty}&\lesssim t^{-\frac12(1+n)}\label{v-pointwise-bounds}
\end{align}
for $n=0,1,2,\ldots,\overline n$, and
\begin{align}\label{v-critical-bounds}
    \|v\|_{L^1([t',t], t^{-\frac12}dt; L^\infty)}+\|v\|_{L^2([t',t]; L^\infty)}^2\lesssim 1+(\log A)^{-1}\log(t/{t'}),
\end{align}
\begin{align}\label{h-critical-bounds}
    \|h\|_{L^1([t',t], t^{-\frac12}dt; L^\infty)}+\|h\|_{L^2([t',t]; L^\infty)}^2\lesssim 1+(\log A)^{-1}\log(t/{t'}).
\end{align}
\end{proposition}

\begin{proof}
We first prove \eqref{vk-pointwise-bounds} and \eqref{v-pointwise-bounds}.
It follows from the definitions in \eqref{app-prin} and the estimate \eqref{eq:Dpsi_bounds} that
\begin{equation}\label{vk_bar_estimate}
\begin{split}
  &\quad  \|\nabla^n\overline v_k(t)\|_{p}+\|\nabla^n\overline h_k(t)\|_{p}\\
    &\lesssim |\Omega_k|^\frac1p\left(\|\nabla^n\overline v_k(t)\|_{\infty}+\|\nabla^n\overline h_k(t)\|_{\infty}\right)\\
    &\lesssim 2^{-k/p}\sum_jN_{j,k}^{1+n}e^{-N_{j,k}^2t}.
    \end{split}
\end{equation}
Thus the estimate for $(\bar v, \bar h)$ in \eqref{v-pointwise-bounds} follows immediately from \eqref{vk_bar_estimate} by taking the sum in $k\geq0$. Now applying the estimate of $(\bar v_k, \bar h_k)$ and Propositions~\ref{difference_estimate_proposition}-~\ref{difference_estimate_H} gives
\begin{equation}\label{alt_vk_bound}
\begin{split}
    \|\nabla^nv_k(t)\|_p&\leq \|\nabla^n(v_k(t)-\bar{v}_k(t))\|_p+\|\nabla^n\bar{v}_k(t)\|_p\\
    &\lesssim \|\nabla^{n+1}(R_k(t)-\bar{R}_k(t))\|_p+\|\nabla^n\bar{v}_k(t)\|_p\\
    &\lesssim N_{1,k+1}^{-\alpha}(t^{-\frac12-\frac n2+\alpha}+1)+2^{-k/p}\sum_jN_{j,k}^{1+n}e^{-N_{j,k}^2t},
\end{split}
\end{equation}
\begin{equation}\label{alt_hk_bound}
\begin{split}
    \|\nabla^nh_k(t)\|_p&\leq \|\nabla^n(h_k(t)-\bar{h}_k(t))\|_p+\|\nabla^n\bar{h}_k(t)\|_p\\
    &\lesssim \|\nabla^{n+1}(H_k(t)-\bar{H}_k(t))\|_p+\|\nabla^n\bar{h}_k(t)\|_p\\
    &\lesssim N_{1,k+1}^{-\alpha}(t^{-\frac12-\frac n2+\alpha}+1)+2^{-k/p}\sum_jN_{j,k}^{1+n}e^{-N_{j,k}^2t}.
\end{split}
\end{equation}
We thus conclude the estimates of \eqref{vk-pointwise-bounds}--\eqref{v-pointwise-bounds}.

We move on to define $f_u$ and $f_B$ and prove \eqref{f_bound}. For $k\geq0$, we can write
\[
\begin{split}
v_k(t) &= -\int_0^{t} e^{(t-s)\Delta }\mathbb P\div (\bar v_{k+1}\otimes \bar v_{k+1}-\bar h_{k+1}\otimes \bar h_{k+1})(s) \, ds\\
&=\int_0^{t} e^{(t-s)\Delta }\mathbb P\div (-v_{k+1}\otimes  v_{k+1}+h_{k+1}\otimes  h_{k+1}+f_{k+1,u})(s) \, ds
\end{split}
\]
with
\[
\begin{split}
f_{k,u}&=\left(v_{k} \otimes  v_{k}-\bar v_{k}\otimes \bar v_{k}\right)-\left(h_{k} \otimes  h_{k}-\bar h_{k}\otimes \bar h_{k}\right)\\
&=(v_{k}-\bar v_{k}) \otimes v_{k} + \bar v_{k} \otimes (v_{k}-\bar v_{k})-(h_{k}-\bar h_{k}) \otimes h_{k} - \bar h_{k} \otimes (h_{k}-\bar h_{k});
\end{split}
\]
and
\[
\begin{split}
h_k(t) &= -\int_0^{t} e^{(t-s)\Delta }\mathbb P\div (\bar v_{k+1}\otimes \bar h_{k+1}-\bar h_{k+1}\otimes \bar v_{k+1})(s) \, ds\\
&=\int_0^{t} e^{(t-s)\Delta }\mathbb P\div (-v_{k+1}\otimes  h_{k+1}+h_{k+1}\otimes  v_{k+1}+f_{k+1,B})(s) \, ds
\end{split}
\]
with
\[
\begin{split}
f_{k,B}&=\left(v_{k} \otimes  h_{k}-\bar v_{k}\otimes \bar h_{k}\right)-\left(h_{k} \otimes  v_{k}-\bar h_{k}\otimes \bar v_{k}\right)\\
&=(v_{k}-\bar v_{k}) \otimes h_{k} + \bar v_{k} \otimes (h_{k}-\bar h_{k})-(h_{k}-\bar h_{k}) \otimes v_{k} - \bar h_{k} \otimes (v_{k}-\bar v_{k}).
\end{split}
\]
Taking the sums $\sum _{k\geq 0}v_k(t)$ and $\sum _{k\geq 0}h_k(t)$ using the equations above, we obtain
\begin{equation}\label{v-h-mild}
\begin{split}
    v(t)&=\int_0^te^{(t-s)\Delta}\mathbb P\div (-v(s)\otimes v(s)+h(s)\otimes h(s)+f_u(s))ds,\\
    h(t)&=\int_0^te^{(t-s)\Delta}\mathbb P\div (-v(s)\otimes h(s)+h(s)\otimes v(s)+f_B(s))ds,
\end{split}    
\end{equation}
where $f_u$ and $f_B$ satisfy 
\begin{align*}
    \mathbb P\div f_u&=\mathbb P\div\left(v\otimes v-h\otimes h+\sum_{k\geq0}(-v_{k+1}\otimes v_{k+1}+h_{k+1}\otimes h_{k+1}+f_{k+1,u})\right)\\
    &=\mathbb P\div\left(v_0\otimes v_0-h_0\otimes h_0+\sum_{k\geq1} f_{k,u}   +\sum_{k_1\neq k_2}v_{k_1}\otimes v_{k_2}-\sum_{k_1\neq k_2}h_{k_1}\otimes h_{k_2}\right)
\end{align*}
and 
\begin{align*}
    \mathbb P\div f_B&=\mathbb P\div\left(v\otimes h-h\otimes v+\sum_{k\geq0}(-v_{k+1}\otimes h_{k+1}+h_{k+1}\otimes v_{k+1}+f_{k+1,B})\right)\\
    &=\mathbb P\div\left(v_0\otimes h_0-h_0\otimes v_0+\sum_{k\geq1} f_{k,B}   +\sum_{k_1\neq k_2}v_{k_1}\otimes h_{k_2}-\sum_{k_1\neq k_2}h_{k_1}\otimes v_{k_2}\right).
\end{align*}
Note the mild form \eqref{v-h-mild} is equivalent to \eqref{v_equation}.
Therefore we define
\begin{equation}\label{def-fu-fb}
\begin{split}
f_u&=\sum_{k\geq0} f_{k,u}   +\sum_{k_1\neq k_2}v_{k_1}\otimes v_{k_2}-\sum_{k_1\neq k_2}h_{k_1}\otimes h_{k_2},\\
f_B&=\sum_{k\geq0} f_{k,B}   +\sum_{k_1\neq k_2}v_{k_1}\otimes h_{k_2}-\sum_{k_1\neq k_2}h_{k_1}\otimes v_{k_2}.
\end{split}
\end{equation}
Regarding the zero modes, in view of \eqref{def_psi_0} and the fact that $a_{j,0,u}$ are constant, we have
\[\mathbb P\div (\bar v_0\otimes\bar v_0)=0.\] 
Since $\bar h_0=0$, we also have
\begin{align*}
    \mathbb P\div (v_0\otimes v_0-h_0\otimes h_0)=\mathbb P\div (\bar v_0\otimes\bar v_0-\bar h_0\otimes\bar h_0+f_{0,u})=\mathbb P\div f_{0,u},\\
    \mathbb P\div (v_0\otimes h_0-h_0\otimes v_0)=\mathbb P\div (\bar v_0\otimes\bar h_0-\bar h_0\otimes\bar v_0+f_{0,B})=\mathbb P\div f_{0,B}.
\end{align*}

Applying Proposition~\ref{difference_estimate_proposition}, Proposition~\ref{difference_estimate_H} and \eqref{v-pointwise-bounds}, 
we deduce for sufficiently small $\alpha>0$
\begin{align*}
\|\nabla^nf_{k,u}(t)\|_{C^\kappa} &\lesssim \sum_{i=0}^n\|\nabla^i(v_{k}-\bar v_{k})\|_{C^\kappa}(\|\nabla^{n-i}v_{k}\|_{C^\kappa}+\|\nabla^{n-i}\bar v_{k}\|_{C^\kappa})\\
&\lesssim\sum_{i=0}^n\varepsilon_0N_{1,k+1}^{-2\alpha}(t^{-\frac12(1+i+\kappa)+2\alpha}+1)(t^{-\frac12(1+n-i+\kappa)}+1)\\
&\lesssim \varepsilon_0N_{1,k+1}^{-2\alpha}(t^{-1-\frac n2+\alpha+(\alpha-\kappa)}+1),
\end{align*}
\begin{align*}
\|\nabla^nf_{k,B}(t)\|_{C^\kappa} &\lesssim \sum_{i=0}^n\|\nabla^i(h_{k}-\bar h_{k})\|_{C^\kappa}(\|\nabla^{n-i}h_{k}\|_{C^\kappa}+\|\nabla^{n-i}\bar h_{k}\|_{C^\kappa})\\
&\lesssim\sum_{i=0}^n\varepsilon_0N_{1,k+1}^{-2\alpha}(t^{-\frac12(1+i+\kappa)+2\alpha}+1)(t^{-\frac12(1+n-i+\kappa)}+1)\\
&\lesssim \varepsilon_0N_{1,k+1}^{-2\alpha}(t^{-1-\frac n2+\alpha+(\alpha-\kappa)}+1).
\end{align*}
It follows that $\sum_{k\geq0} f_{k,u} $ and $\sum_{k\geq0} f_{k,B}$ satisfy the estimate of \eqref{f_bound}.

The quadratic terms in \eqref{def-fu-fb} can be estimated similarly as in \cite{CDP}, 
\begin{equation}\notag
\begin{split}
   &\quad \Big\|\nabla^n\sum_{k_1\neq k_2}v_{k_1}\otimes v_{k_2}\Big\|_{C^\kappa}+\Big\|\nabla^n\sum_{k_1\neq k_2}h_{k_1}\otimes h_{k_2}\Big\|_{C^\kappa}\\
   &+\Big\|\nabla^n\sum_{k_1\neq k_2}v_{k_1}\otimes h_{k_2}\Big\|_{C^\kappa}+\Big\|\nabla^n\sum_{k_1\neq k_2}h_{k_1}\otimes v_{k_2}\Big\|_{C^\kappa}\\
    &\lesssim  \varepsilon_0 t^{-1-\frac n2+2\alpha-\kappa}+t^{-1-\frac n2+\frac12(1-b^{-1/J_d}-4\kappa)}\\
    &\lesssim \varepsilon_0 (t^{-1 - \frac{n}{2} +\alpha}+1)
\end{split}    
\end{equation}
upon choosing $\alpha,\kappa>0$ small.  We complete the proof of \eqref{f_bound}.

Since $v_k$ and $h_k$ satisfy the same estimates,  \eqref{v-critical-bounds} and \eqref{h-critical-bounds} can be obtained analogously as in \cite{CDP}, and hence we omit the details.

\end{proof}

\section{Construction of perturbation}\label{sec:corrector}

In this section we introduce a perturbation that removes the forcing terms in \eqref{v_equation} and hence obtain a solution to system \eqref{mhd}. To do so, we apply semigroup theory and a fixed point argument. 

Denote $(U,H)$ by any classical solution to the MHD system \eqref{mhd} on the time interval $[0,T]$. For notation convenience, we translate $[0,T]$ to $[-T_*, T-T_*]$ so the blowup occurs at $t=0$ after translation. Define
\begin{align*}
    \sup_{0\leq n\leq 10}\left(\|\nabla^nU\|_{L_{t,x}^\infty(\mathbb T^d\times[0,T-T_*])}+\|\nabla^n H\|_{L_{t,x}^\infty(\mathbb T^d\times[0,T-T_*])}\right)\eqcolon C_{U,H}<\infty.
\end{align*}

We extend $(v,h)$ on $\mathbb T^d\times[0,\infty)$ constructed in Section \ref{sec:est} by $0$ to the full time interval $\mathbb R$. For $N_0>0$, we rescale the pair to $(v^{N_0},h^{N_0})(x,t)=(N_0v(N_0x,N_0^2t), N_0 h(N_0x,N_0^2t))$, which is viewed as $2\pi/N_0$-periodic on $\mathbb R^d\times \mathbb R$. 

The goal is to construct a small perturbation $(w,\zeta)$ such that the pair $(u,B)$ given by
\[
u = U + v^{N_0} + w^{N_0},\qquad B = H + h^{N_0} + \zeta^{N_0},
\]
is a solution to the MHD system \eqref{mhd}.
In the rest of the section we work with the rescaled solution
\[
u^{1/N_0}=U^{1/N_0}+v+w,\qquad B^{1/N_0}=H^{1/N_0}+h+\zeta.
\]
Since $(v,h)$ satisfies \eqref{v_equation} and $(U^{1/N_0}, H^{1/N_0})$ is a classical solution to \eqref{mhd}, we derive that 
 $(w,\zeta)$ satisfies
\[
\begin{split}
&\quad\partial_tw-\Delta w+\mathbb P\div\!\Big(w\otimes w-\zeta\otimes\zeta+2(U^{1/N_0}+v)\odot w-2(H^{1/N_0}+h)\odot\zeta\Big)\\
&=-\mathbb P\div\!\Big(f_u+2U^{1/N_0}\odot v-2H^{1/N_0}\odot h\Big),\\
&\quad\partial_t\zeta-\Delta \zeta+\mathbb P\div\!\Big(w\otimes\zeta-\zeta\otimes w+(U^{1/N_0}+v)\otimes\zeta-\zeta\otimes(U^{1/N_0}+v)\\
&\qquad\qquad\qquad\qquad\qquad\qquad +w\otimes(H^{1/N_0}+h)-(H^{1/N_0}+h)\otimes w\Big)\\
&=-\mathbb P\div\!\Big(f_B+U^{1/N_0}\otimes h+v\otimes H^{1/N_0}-H^{1/N_0}\otimes v-h\otimes U^{1/N_0}\Big)
\end{split}
\]
with initial data $(w,\zeta)(x,0)=(0,0)$.


Recall the lifetime of $(U^{1/N_0},H^{1/N_0})$ is $[-N_0^2T_*,N_0^2(T-T_*)]$ after translation and rescaling.
Take $\bar T=\min\{N_0^2(T-T_*),\bar C\}$ for some large constant $\bar C>1$. Define the spaces
\begin{equation}\label{def-X}
\begin{split}
X&=\Big\{V\in C^0((0,\bar T]; C^{1,\kappa}(\mathbb T^d; \mathbb R^d)): \\
& \qquad\qquad \|V\|_{X}:= \sup_{t\in(0,\bar T]}(t^{\frac{1-\alpha}2} \|V\|_{L^\infty}+t^{\frac{2-\alpha}2} \|\nabla V\|_{C^{\kappa}})<\infty\Big\}
\end{split}
\end{equation}
\begin{equation}\label{def-Y}
\begin{split}
Y&=\Big\{\phi\in C^0((0,\bar T]; C^{1,\kappa}(\mathbb T^d; \mathbb R^{d\times d})): \\
& \qquad\qquad \|\phi\|_{Y}:= \sup_{t\in(0,\bar T]}(t^{1-\alpha} \|\phi\|_{L^\infty}+t^{\frac32-\alpha} \|\nabla \phi\|_{C^{\kappa}})<\infty\Big\},
\end{split}
\end{equation}
and the product spaces
\[
\mathcal X\coloneqq X\times X,\qquad \mathcal Y\coloneqq Y\times Y\]
with
\[
\|(W,Z)\|_{\mathcal X}\coloneqq \|W\|_X+\|Z\|_X,\quad \|(\phi_u,\phi_B)\|_{\mathcal Y}\coloneqq \|\phi_u\|_Y+\|\phi_B\|_Y.
\]
\begin{lemma}\label{product_X_Y_lemma}
    For $g,h\in X$, we have $g\otimes h,g\odot h\in Y$ and
    \begin{align*}
        \|g\otimes h\|_Y+\|g\odot h\|_Y\lesssim \|g\|_X\|h\|_X.
    \end{align*}
\end{lemma}
The lemma follows from the estimate
\[
\|gh\|_{C^{1,\kappa}}\lesssim \|g\|_{C^{1,\kappa}}\|h\|_{L^\infty}+\|g\|_{L^\infty}\|h\|_{C^{1,\kappa}}.
\]

Set
\[
\tilde v\coloneqq U^{1/N_0}+v,\qquad \tilde h\coloneqq H^{1/N_0}+h.
\]
For $\Phi=(\phi_u,\phi_B)\in\mathcal Y$ and $0<t'\leq t\leq \bar T$, denote
\[
\mathbb S(t,t')\Phi=(W,Z)
\]
which solves the system
\begin{equation}\label{semi-group}
\begin{split}
\partial_tW-\Delta W+\mathbb P\div\!\Big(2\tilde v\odot W-2\tilde h\odot Z\Big)&=0,\\
\partial_tZ-\Delta Z+\mathbb P\div\!\Big(\tilde v\otimes Z-Z\otimes \tilde v+W\otimes \tilde h-\tilde h\otimes W\Big)&=0,\\
(W,Z)\vert_{t=t'}=\big(\mathbb P\div \phi_u(t'),\,  \mathbb P\div \phi_B(t')\big)&.
\end{split}
\end{equation}

\begin{proposition}\label{prop-semi}
For any $\Phi=(\phi_u,\phi_B)\in\mathcal Y$, we have
\begin{equation}\notag
\begin{aligned}
\|W(t)\|_{L^\infty}+\|Z(t)\|_{L^\infty}
&+(t-t')^{\frac12}\Big(\|\nabla W(t)\|_{C^\kappa}+\|\nabla Z(t)\|_{C^\kappa}\Big)\\
&\lesssim_{\bar C} t^{-\frac12}(t')^{-1+\alpha}(t/t')^{\epsilon}\|\Phi\|_{\mathcal Y}, \quad 0<t'\leq t\leq \bar T.
\end{aligned}
\end{equation}
\end{proposition}

\begin{proof}
In light of \eqref{semi-group}, it follows from Duhamel's formula
\begin{align*}
W(t)&=e^{(t-t')\Delta}\mathbb P\div\phi_u(t')-\int_{t'}^te^{(t-s)\Delta}\mathbb P\div\Big(2\tilde v\odot W-2\tilde h\odot Z\Big)(s)\,ds,\\
Z(t)&=e^{(t-t')\Delta}\mathbb P\div\phi_B(t')\\
&\quad-\int_{t'}^te^{(t-s)\Delta}\mathbb P\div\Big(\tilde v\otimes Z-Z\otimes \tilde v+W\otimes \tilde h-\tilde h\otimes W\Big)(s)\,ds.
\end{align*}
Define
\[
\mathcal H(t)\coloneqq t^{1/2}\big(\|W(t)\|_{L^\infty}+\|Z(t)\|_{L^\infty}\big).
\]
Using the same splitting of $[t',t]$ with $t<\bar C$ as in the argument for the Navier-Stokes equations in \cite{CDP}, the heat kernel bound, and
\[
\|\tilde v(s)\|_{L^\infty}+\|\tilde h(s)\|_{L^\infty}\lesssim C_{U,H}+\|v(s)\|_{L^\infty}+\|h(s)\|_{L^\infty},
\]
we obtain
\begin{equation}\notag
\begin{split}
\mathcal H(t)&\lesssim (t')^{-1+\alpha}\|\Phi\|_{\mathcal Y}
+\int_{t'}^t\Big(s^{-1/2}+(t-s)^{-1/2}\Big)\\
&\quad\cdot\Big(C_{U,H}+\|v(s)\|_{L^\infty}+\|h(s)\|_{L^\infty}\Big)\mathcal H(s)\,ds.
\end{split}
\end{equation}
Now \eqref{v-critical-bounds} and \eqref{h-critical-bounds} yield the same logarithmic control used earlier, hence fractional Gr\"onwall implies
\[
\|W(t)\|_{L^\infty}+\|Z(t)\|_{L^\infty}\lesssim t^{-1/2}(t')^{-1+\alpha}(t/t')^{\epsilon/2}\|\Phi\|_{\mathcal Y}.
\]
Next, differentiating the system and using
\begin{equation}\label{heat}
\|e^{t\Delta}\nabla^m \mathbb P g\|_{C^{r}}\lesssim t^{-\frac{m+r-s}2}\| g\|_{C^{s}},
\end{equation}
we estimate $\nabla W,\nabla Z$ exactly in the same way (split at $(t+t')/2$, use the previous $L^\infty$ bound for $(W,Z)$, and the same coefficient bound for $(\tilde v,\tilde h)$). This gives
\[
(t-t')^{1/2}\big(\|\nabla W(t)\|_{C^\kappa}+\|\nabla Z(t)\|_{C^\kappa}\big)\lesssim t^{-1/2}(t')^{-1+\alpha}(t/t')^{\epsilon}\|\Phi\|_{\mathcal Y},
\]
after enlarging the implicit constant and replacing $\epsilon/2$ by $\epsilon$.
\end{proof}



With the preparations above, we are ready to apply the fixed point argument now.

\begin{proposition}\label{w-exists-fixed-point-proposition}
There exists $\varepsilon>0$ such that for any $\delta\in(0,\varepsilon)$ and sufficiently large $A>0$, we can find $w,\zeta\in B_X(0,\delta)$ with
\[
u^{1/N_0}=U^{1/N_0}+v+w,\qquad B^{1/N_0}=H^{1/N_0}+h+\zeta
\]
satisfying:
\begin{itemize}
\item [(i)]
 $(u^{1/N_0}, B^{1/N_0})$ solves \eqref{mhd} on $\mathbb R^d\times [0,\bar T]$;
\item [(ii)] 
$(u^{1/N_0},B^{1/N_0})$ can be extended to a classical solution on $[0,N_0^2(T-T_*)]$, provided $\bar C$ and $N_0$ are large depending on $(U,H)$, $T$ and $T_*$.
\end{itemize}
\end{proposition}

\begin{proof}
For $(W,Z)\in\mathcal X$, define
\begin{align*}
\Phi_u(W,Z)&\coloneqq W\otimes W-Z\otimes Z+f_u+2U^{1/N_0}\odot v-2H^{1/N_0}\odot h,\\
\Phi_B(W,Z)&\coloneqq W\otimes Z-Z\otimes W+f_B+U^{1/N_0}\otimes h+v\otimes H^{1/N_0}\\
&\quad-H^{1/N_0}\otimes v-h\otimes U^{1/N_0},
\end{align*}
and $\Psi(W,Z)\coloneqq (\Phi_u(W,Z),\Phi_B(W,Z))\in\mathcal Y$. Set
\[
\mathcal F(W,Z)(t)\coloneqq-\int_0^t\mathbb S(t,t')\Psi(W,Z)\,dt'.
\]

Fix $\epsilon\in(0,\alpha)$. Applying Proposition~\ref{prop-semi} gives
\begin{align*}
\|\mathbb S(t,t')\Psi\|_{L^\infty}
&\lesssim_{\bar C} t^{-1/2}(t')^{-1+\alpha}(t/t')^\epsilon\|\Psi\|_{\mathcal Y},\\
\|\nabla\mathbb S(t,t')\Psi\|_{C^\kappa}
&\lesssim_{\bar C} (t-t')^{-1/2}t^{-1/2}(t')^{-1+\alpha}(t/t')^\epsilon\|\Psi\|_{\mathcal Y}
\end{align*}
for $0<t'\le t\le \bar T$.
It then follows
\begin{align*}
\|\mathcal F(W,Z)(t)\|_{L^\infty}
&\lesssim_{\bar C} t^{-1/2+\epsilon}\int_0^t (t')^{-1+\alpha-\epsilon}\,dt'\,\|\Psi(W,Z)\|_{\mathcal Y}\\
&\lesssim_{\bar C} t^{-1/2+\alpha}\|\Psi(W,Z)\|_{\mathcal Y},
\end{align*}
\begin{align*}
\|\nabla\mathcal F(W,Z)(t)\|_{C^\kappa}
&\lesssim_{\bar C} t^{-1/2+\epsilon}\int_0^t (t-t')^{-1/2}(t')^{-1+\alpha-\epsilon}\,dt'\,\|\Psi(W,Z)\|_{\mathcal Y}\\
&\lesssim_{\bar C} t^{-1+\alpha}\|\Psi(W,Z)\|_{\mathcal Y}.
\end{align*}
We thus conclude
\[
\|\mathcal F(W,Z)\|_{\mathcal X}\lesssim_{\bar C}\|\Psi(W,Z)\|_{\mathcal Y}.
\]
In view of Lemma~\ref{product_X_Y_lemma}, we have
\[
\|W\otimes W-Z\otimes Z\|_Y+\|W\otimes Z-Z\otimes W\|_Y\lesssim \|(W,Z)\|_{\mathcal X}^2.
\]
It follows from \eqref{f_bound} and the definition of $Y$ \eqref{def-Y} that
\[
\|f_u\|_Y+\|f_B\|_Y\lesssim \varepsilon_0.
\]
On the other hand, since $U,H$ are smooth, we have
\[
\|U^{1/N_0}\|_{L_t^\infty C_x^{1,\kappa}}+\|H^{1/N_0}\|_{L_t^\infty C_x^{1,\kappa}}\lesssim N_0^{-1}C_{U,H},
\]
and by \eqref{v-pointwise-bounds},
\begin{align*}
\|U^{1/N_0}\odot v\|_Y+\|H^{1/N_0}\odot h\|_Y+\|U^{1/N_0}\otimes h\|_Y+\|v\otimes H^{1/N_0}\|_Y\\
+\|H^{1/N_0}\otimes v\|_Y+\|h\otimes U^{1/N_0}\|_Y
\lesssim_{\bar C,U,H} N_0^{-1}.
\end{align*}
Hence we infer
\[
\|\Psi(W,Z)\|_{\mathcal Y}\le C_{\bar C,U,H}\Big(\|(W,Z)\|_{\mathcal X}^2+\varepsilon_0+N_0^{-1}\Big),
\]
which immediately gives
\[
\|\mathcal F(W,Z)\|_{\mathcal X}\le C_{\bar C,U,H}\Big(\|(W,Z)\|_{\mathcal X}^2+\varepsilon_0+N_0^{-1}\Big).
\]

Now take $(W_i,Z_i)\in B_{\mathcal X}(0,\delta)$, $i=1,2$. Expanding quadratic differences and using Lemma~\ref{product_X_Y_lemma} yields
\[
\|\Psi(W_1,Z_1)-\Psi(W_2,Z_2)\|_{\mathcal Y}
\lesssim \delta\,\|(W_1-W_2,Z_1-Z_2)\|_{\mathcal X}.
\]
Applying the linear estimate again gives
\[
\|\mathcal F(W_1,Z_1)-\mathcal F(W_2,Z_2)\|_{\mathcal X}
\le C_{\bar C,U,H}\delta\,\|(W_1-W_2,Z_1-Z_2)\|_{\mathcal X}.
\]

Choose $\delta>0$ such that $C_{\bar C,U,H}\delta\le \frac12$. Then choose $A\gg1$ (hence $\varepsilon_0$ small through \eqref{f_bound}) and $N_0$ large so that
\[
C_{\bar C,U,H}\big(\delta^2+\varepsilon_0+N_0^{-1}\big)\le \delta.
\]
Therefore $\mathcal F$ is a contraction on $B_{\mathcal X}(0,\delta)$ and has a unique fixed point $(w,\zeta)\in B_{\mathcal X}(0,\delta)$.

By construction, $(w,\zeta)$ solves the coupled corrector equations, hence $(u^{1/N_0},B^{1/N_0})$ solves \eqref{mhd} on $[0,\bar T]$. Finally choose $\bar C$ and then $N_0$ large so that
\[
\bar T=\min\{N_0^2(T-T_*),\bar C\}=N_0^2(T-T_*).
\]
Rescaling and translation back yields a classical solution on $[0,T]$.
\end{proof}

\section{Proof of main results}\label{sec:proof}

We prove Theorem \ref{main-thm} and Theorem \ref{thm-non-unique} in this section. We first claim that the constructed pair $(u,B)$ in Section \ref{sec:corrector} is a weak solution of \eqref{mhd}. 

\begin{lemma}\label{le-weak}
Assume a pair $(u,B)$ satisfies
\begin{itemize}
\item [(i)]
$u,B\in L^2(\mathbb{T}^d \times [0,T])\cap L^\infty([0,T];H^s(\mathbb{T}^d))$ for some $s\in \mathbb{R}$; 
\item [(ii)]
$(u,B)|_{[0,T_*)}$ and $(u,B)|_{(T_*,T]}$ are classical solutions of \eqref{mhd} respectively on $[0,T_*)$ and $(T_*,T]$;
\item [(iii)]
the two side limits coincide, i.e.
\[
\lim_{t\to T_*^-} (u(t),B(t)) = \lim_{t\to T_*^+} (u(t), B(t)), \qquad \text{in} \quad \mathcal{D}'(\mathbb{T}^d),
\]
\end{itemize}
Then $(u,B)$ is a weak solution of \eqref{mhd} on $[0,T]$.
\end{lemma}
The proof follows the lines of that of Lemma 7.1 from \cite{CDP}. It immediately implies $(u,B)$ constructed above is a weak solution of \eqref{mhd}. In addition, we show that $(u,B)$ is in some borderline spaces. 

\begin{proposition}\label{prop:borderline}
    The pair $(u,B)$ defined in Section \ref{sec:corrector} satisfies
\begin{itemize}
\item [(i)]
  $u,B\in L_t^{2,\infty}L_x^\infty\cap L_{t}^2L_x^p, \quad p<\infty$;
    \item [(ii)]
both $u$ and $B$ belong to the Koch--Tataru space $X_{T-T_*}$ and are weak-* continuous in time with values in $BMO^{-1}$. 
    \item [(iii)]
  if $d\geq3$, $u$ and $B$ are in the improved space
    \[
   u,B\in L_t^\infty \dot W^{-1,\infty}_x.
    \]
\end{itemize}
\end{proposition}

\begin{proof}
Recall from Section~\ref{sec:corrector},
\[
u=U+v^{N_0}+w^{N_0},\qquad B=H+h^{N_0}+\zeta^{N_0}.
\]
All norms in the proposition are critical (or subcritical) under the MHD scaling, so it is enough to estimate the rescaled pair
\[
u^{1/N_0}=U^{1/N_0}+v+w,\qquad B^{1/N_0}=H^{1/N_0}+h+\zeta.
\]

\textbf{Step 1:} $\mathbf{L_t^\infty\dot W^{-1,\infty}_x}$ for both $u,B$ when $d\ge3$.
By Proposition~\ref{w-exists-fixed-point-proposition}, $(w,\zeta)\in B_X(0,\delta)\times B_X(0,\delta)$. Using the coupled mild corrector equations and the heat estimate in negative order,
\begin{align*}
\|w(t)\|_{\dot W^{-1,\infty}}+\|\zeta(t)\|_{\dot W^{-1,\infty}}
\lesssim_\varepsilon \int_0^t \big(\|\mathcal N_u(s)\|_{C^\varepsilon}+\|\mathcal N_B(s)\|_{C^\varepsilon}\big)\,ds.
\end{align*}
By Lemma~\ref{product_X_Y_lemma}, \eqref{v-pointwise-bounds}, \eqref{f_bound}, and smoothness of $(U,H)$,
\[
\|\mathcal N_u(s)\|_{C^\varepsilon}+\|\mathcal N_B(s)\|_{C^\varepsilon}
\lesssim s^{-1+\alpha-\varepsilon}+s^{-\frac12(1+\varepsilon)}+1,
\]
hence
\[
\|w(t)\|_{\dot W^{-1,\infty}}+\|\zeta(t)\|_{\dot W^{-1,\infty}}
\lesssim t^{\alpha-\varepsilon}+t,
\]
which implies $w,\zeta\in L_t^\infty\dot W^{-1,\infty}_x$.

For the principal pair, write
\[
v_k=\div R_k,\quad h_k=\div H_k,\quad \bar v_k=\div \bar R_k,\quad \bar h_k=\div \bar H_k.
\]
It follows from \eqref{eq:Dpsi_bounds},
\[
\|\bar R_k(t)\|_\infty+\|\bar H_k(t)\|_\infty\lesssim \sum_j e^{-N_{j,k}^2t}.
\]
Employing Proposition~\ref{difference_estimate_proposition} and Proposition~\ref{difference_estimate_H} yields
\[
\|R_k(t)\|_\infty+\|H_k(t)\|_\infty
\lesssim N_{1,k+1}^{-\alpha}+\mathbbm1_{t\le N_{1,k}^{-2}},
\]
and the mild solution formulas also give
\[
\|R_k(t)\|_\infty+\|H_k(t)\|_\infty\lesssim_\varepsilon tN_{J_d,k+1}^{2+\varepsilon}.
\]
Therefore
\begin{align*}
\|R(t)\|_\infty+\|H(t)\|_\infty
&\lesssim \sum_{k:tN_{J_d,k+1}^{2+2\varepsilon}\le1}tN_{J_d,k+1}^{2+\varepsilon}
+\sum_{\substack{k:tN_{J_d,k+1}^{2+2\varepsilon}>1\\ \wedge\, N_{1,k}^2t\le1}}1
+\sum_{k:N_{1,k}^2t>1}N_{1,k+1}^{-\alpha}\\
&\lesssim 1+\#\{k:N_{J_d,k+1}^{-2-2\varepsilon}<t\le N_{1,k}^{-2}\}+t^{b\alpha/2}.
\end{align*}
Using \eqref{N_definition}, the right-hand side is uniformly bounded in $t$, so
\[
v,h\in L_t^\infty\dot W^{-1,\infty}_x.
\]
Adding $(U^{1/N_0},H^{1/N_0})$ yields
\[
u^{1/N_0},B^{1/N_0}\in L_t^\infty\dot W^{-1,\infty}_x,\qquad d\ge3.
\]

\textbf{Step 2:} $\mathbf{X_{T-T_*}}$ (hence weak-* continuity into $\mathbf{BMO^{-1}}$) for both $u,B$.
Decompose
\begin{align*}
u^{1/N_0}&=U^{1/N_0}+w+u_1+u_2+u_3,\\
B^{1/N_0}&=H^{1/N_0}+\zeta+b_1+b_2+b_3,
\end{align*}
with
\begin{align*}
u_1&=\sum_{j;k\in \Lambda_u\cup\Lambda_B}\big(e^{t\Delta}N_{j,k}\Delta\psi_{j,k}+\bar v_{j,k}\big),&
u_2&=v-\bar v,&
u_3&=-e^{t\Delta}\sum_{j;k\in \Lambda_u\cup\Lambda_B}N_{j,k}\Delta\psi_{j,k},\\
b_1&=\sum_{j; k\in \Lambda_B}\big(e^{t\Delta}N_{j,k}\Delta\psi_{j,k,B}+\bar h_{j,k}\big),&
b_2&=h-\bar h,&
b_3&=-e^{t\Delta}\sum_{j; k\in \Lambda_B}N_{j,k}\Delta\psi_{j,k,B}.
\end{align*}
By smoothness, $U,H\in X_{T-T_*}$; by Proposition~\ref{w-exists-fixed-point-proposition}, $w,\zeta\in X_{T-T_*}$. For $(u_1,b_1)$, the commutator estimate of Lemma~\ref{l:commutator} gives
\[
\|u_1(t)\|_\infty+\|b_1(t)\|_\infty\lesssim t^{-1/2+\epsilon}+1.
\]
For $(u_2,b_2)$, Proposition~\ref{difference_estimate_proposition} and Proposition~\ref{difference_estimate_H} imply the same $t^{-1/2+\epsilon}$ control. Applying the fact
\begin{align}\label{X_T_criterion}
\sup_{t\in(0,T-T_*]}t^{\frac12-\epsilon}\|f(t)\|_{L^\infty}<\infty\Longrightarrow f\in X_{T-T_*},
\end{align}
we conclude $(u_1,b_1)$ and $(u_2,b_2)$ are in $X_{T-T_*}$.
For $(u_3,b_3)$, the standard BMO argument based on \eqref{eq:Dpsi_bounds}, \eqref{Omega_volume_estimate}, and \eqref{N_and_M_ordering} gives $(u_3,b_3)\in X_{T-T_*}$. Hence
\[
u^{1/N_0},B^{1/N_0}\in X_{T-T_*}.
\]
Therefore each component is weak-* continuous in time into $BMO^{-1}$.

\textbf{Step 3:} $\mathbf{L_t^{2,\infty}L_x^\infty\cap L_t^2L_x^p}$ for both $u,B$.
From the $X$-bounds,
\[
\|U^{1/N_0}(t)\|_\infty+\|H^{1/N_0}(t)\|_\infty+\|w(t)\|_\infty+\|\zeta(t)\|_\infty\lesssim t^{-1/2+\epsilon}.
\]
From \eqref{v-pointwise-bounds} (with $n=0$),
\[
v,h\in L_t^{2,\infty}L_x^\infty.
\]
Hence $u^{1/N_0},B^{1/N_0}\in L_t^{2,\infty}L_x^\infty$. Also, by \eqref{vk-pointwise-bounds}, for each $k$ and all $p<\infty$,
\[
\|v_k\|_{L_t^2L_x^p}+\|h_k\|_{L_t^2L_x^p}\lesssim N_{1,k}^{-\alpha}+2^{-k/p}.
\]
Summing in $k$ gives $v,h\in L_t^2L_x^p$, and therefore
\[
u,B\in L_t^{2,\infty}L_x^\infty\cap L_t^2L_x^p,\qquad \forall p<\infty.
\]
This completes the proof of the proposition.
\end{proof}

\textbf{Proof of Theorem \ref{main-thm}:}
From Section~\ref{sec:corrector}, the pair $(u,B)$ defined by
\[
u=U+v^{N_0}+w^{N_0},\qquad B=H+h^{N_0}+\zeta^{N_0},
\]
solves \eqref{mhd} in the classical sense away from the blowup time $T_*$. Lemma~\ref{le-weak} implies that $(u,B)$ is a weak solution of \eqref{mhd} with initial data
\[
(u_0,B_0)=\big(U(0,\cdot),H(0,\cdot)\big).
\]

The upper bound follows from smoothness of $(U,H)$ and \eqref{v-pointwise-bounds},
\[
\|U(t)\|_{L^\infty}+\|H(t)\|_{L^\infty}+\|v^{N_0}(t)\|_{L^\infty}+\|h^{N_0}(t)\|_{L^\infty}\lesssim t^{-1/2}.
\]
For the perturbation, we conclude by the fixed-point bound in $X$,
\[
\|w^{N_0}(t)\|_{L^\infty}+\|\zeta^{N_0}(t)\|_{L^\infty}\lesssim t^{-1/2+\alpha/2}\lesssim t^{-1/2},
\]
hence
\[
\|u(t)\|_{L^\infty}+\|B(t)\|_{L^\infty}\lesssim t^{-1/2},\qquad 0<t<T.
\]

Regarding the lower bound, there is a sequence $t_n\to0^+$ along which the principal part $(v,h)$ satisfies
\[
\|v(t_n)\|_{L^\infty}\sim e^{-1}t_n^{-1/2}, \quad \|h(t_n)\|_{L^\infty}\sim e^{-1}t_n^{-1/2}.
\]
Therefore
\[
\|u(t_n)\|_{L^\infty}
\ge \|v^{N_0}(t_n)\|_{L^\infty}-\|U(t_n)\|_{L^\infty}-\|w^{N_0}(t_n)\|_{L^\infty}
\gtrsim t_n^{-1/2},
\]
and similarly $\|B(t_n)\|_{L^\infty}\gtrsim t_n^{-1/2}$.
So both $u$ and $B$ blow up instantaneously with critical rate. The borderline-space conclusions for $(u,B)$ follow from Proposition~\ref{prop:borderline}.

\textbf{Proof of Theorem \ref{thm-non-unique}:}
In Section~\ref{sec:corrector}, we construct a solution $(u^{(\sigma)},B^{(\sigma)})$ for each scaling parameter $N_0>0$, with $\sigma=C/N_0\in(0,1]$ for some large constant $C$. The constant $C$ can be chosen large enough in order to have Proposition~\ref{w-exists-fixed-point-proposition} hold with $N_0=C$.

We then define
\begin{align*}
    (u^{(\sigma)},B^{(\sigma)})=\begin{cases}
        (u,B)|_{N_0=C/\sigma}=\left(U+v^{C/\sigma}+w^{C/\sigma}, H+h^{C/\sigma}+\zeta^{C/\sigma}\right),&\sigma\in(0,1],\\
        (U,H),&\sigma=0
    \end{cases}
\end{align*}
and obtain a family of blowup solutions.

The regularity properties of $u^{(\sigma)}$ and $B^{(\sigma)}$ claimed in the theorem follow from Proposition \ref{prop:borderline}.

\qed

\appendix

\section{}
We collect the following technical lemmas proved in \cite{CDP}.
\begin{lemma}[Commutator estimate for $e^{t\Delta}$]\label{l:commutator}
    Let $a\in C^\infty(\mathbb T^d)$ and $\xi\in\mathbb Z^d\setminus \{0\}$, and define
    \begin{align*}
        A_i\coloneqq |\xi|^{-i}\|\nabla^ia\|_{L^\infty}.
    \end{align*}
    Then
    \begin{equation}\begin{aligned}\label{commutator_inequality}
        &\|\nabla^n[e^{t\Delta},a(x)]\sin(\xi\cdot x)\|_{L^\infty}\\
        &\qquad\lesssim_{m,n} |\xi|^{n}\sum_{i=0}^n \left((A_i^{1-1/m}A_{m+i}^{1/m}+A_i^{1-2/m}A_{m+i}^{2/m})e^{-|\xi|^2t/4}+A_{m+i}\right)
        \end{aligned}\end{equation}
        for $m\geq3+n$. Furthermore,
        \begin{align}\label{heat-decay-estimate}
             \|\nabla^n e^{t\Delta}(a(x)\sin(\xi\cdot x))\|_\infty&\lesssim_{m,n} |\xi|^{n}\left(\sum_{i=0}^{n}A_i e^{-|\xi|^2t/4}+\sum_{i=m}^{n+m}A_{i}\right).
        \end{align}
\end{lemma}

\begin{lemma}\label{l:heat_stationary_phase}
Let $\xi\in\mathbb Z^d\setminus \{0\}$, $a\in C^\infty(\mathbb T^d; \mathbb R^d)$, and denote $A_{i,\kappa}\coloneqq|\xi|^{-i-\kappa}\|\nabla^ia\|_{C^\kappa}$ for $\kappa\in(0,1)$. Then for $n\geq 0$ and $m\geq 1$
\begin{align*}
\|\nabla^ne^{t\Delta}\mathcal R\mathbb P(ae^{i\xi\cdot x})\|_{C^\kappa}&\lesssim |\xi|^{n+\kappa}\left(|\xi|^{-1}\sum_{i=0}^{m-1}A_{i,\kappa}e^{-|\xi|^2t/4}+\sum_{i=m}^{n+2m}
A_{i,\kappa}\right).
\end{align*}
\end{lemma}

\begin{lemma}[Oscillation estimate]\label{l:oscillation_estimate}
Let $a\in C^\infty(\mathbb T^d;\mathbb R^d)$ and $b\in C^\infty(\mathbb T^d;\mathbb R)$. Let $\lambda\in\mathbb N$ and $\Xi\in \mathbb Z^d$ with $|\Xi|\in\frac12\lambda\mathbb N$. Then, for any $n\geq0$, $m\geq1$, and $\kappa\in(0,1)$, we obtain 
\begin{align*}
&\Big\|\nabla^ne^{t\Delta}\mathcal R\mathbb P\Big(a(x)\Big(b(\lambda x)\sin^2(\Xi\cdot x)-\fint_{\mathbb T^d}b(\lambda y)\sin^2(\Xi\cdot y)dy\Big)\Big)\Big\|_{C^\kappa}\\
&\qquad\lesssim_{m,n,\kappa} \Bigg(\sum_{i=0}^{m-1}(\lambda^{n-i-1}+|\Xi|^{n-i-1})\|a\|_{C^{i,\kappa}}e^{-\lambda^2t/4}\\
&\qquad\qquad\qquad\qquad+\sum_{i=m}^{n+2m}
(\lambda^{n-i}+|\Xi|^{n-i})\|a\|_{C^{i,\kappa}}\Bigg)\|\nabla^{n+2d}b\|_{L^\infty}
\end{align*}
\end{lemma}

\bigskip

\section*{Data Availability Statement}
This article does not contain any data sets generated or analyzed during the current study.

\section*{Conflict-of-Interest Statement}
The author declares that there is not any known competing financial interest or other type of interest that could have appeared to influence the work reported in this article.

\bigskip

\bibliographystyle{amsrefs}
\bibliography{NSE}

\end{document}